\documentclass{amsart}[11pt]
 
\usepackage{amscd, amsfonts, amsmath,amsthm, amssymb, color,enumerate , epsf, faktor,graphicx, mathtools, 
 tikz, url ,verbatim ,wasysym , xypic}

\usepackage[all]{xy}
\usetikzlibrary{matrix,arrows}
\usepackage[mathscr]{euscript}

\DeclarePairedDelimiter\floor{\lfloor}{\rfloor}

\usepackage[final]{hyperref}
\hypersetup{ colorlinks = true, linkcolor=blue,   citecolor = black
}
\newcommand*{\Scale}[2][4]{\scalebox{#1}{$#2$}}%

\theoremstyle{plain}

 \newtheorem*{theorem*}{Theorem}
\newtheorem{thm}{Theorem}[subsection]
\newtheorem{prop}{Proposition}[subsection]
\newtheorem{defin}{Definition}[subsection]
\newtheorem{lemma}{Lemma}[subsection]
\newtheorem{coro}{Corollary}[subsection]
\newtheorem{remark}{Remark}[subsection]
\newcommand{\OL}{\mathcal{O}_L}
\newcommand{\OK}{\mathcal{O}_K}
\newcommand{\OKa}{\mathcal{O}_{\mathcal{K}}}
\newcommand{\OMa}{\mathcal{O}_{\mathcal{M}}}
\newcommand{\OLa}{\mathcal{O}_{\mathcal{L}}}

\newcommand{\OS}{\mathcal{O}_S}
\newcommand{\Z}{\mathbb{Z}}

\newcommand{\Qp}{\mathbb{Q}_p}

\newcommand{\Tr}{\mathbb{T}_{\La/S}}
\newcommand{\OZ}{\mathcal{O}}

\newcommand{\Ka}{\mathcal{K}}
\newcommand{\Ma}{\mathcal{M}}
\newcommand{\La}{\mathcal{L}}
\newcommand{\Ea}{\mathcal{E}}

\newcommand{\LTT}{L\{\{T_1\}\}\cdots \{\{T_{d-1}\}\}}

\newcommand{\RL}{R_{\La,1}}
\newcommand{\RM}{R_{\Ma,1}}
\newcommand{\RKt}{R_{\Ka_t,1}}

\newcommand{\TT}{\{\{T_1\}\}\cdots \{\{T_{d-1}\}\}}
\newcommand{\ad}{\{a_1,\dots,a_d \}}
\newcommand{\adp}{(a_1,\dots,a_d)}

\newcommand{\ada}{a_1,\dots,a_{d-1} }

\newcommand{\Td}{T_1,\dots,T_{d-1} }
\newcommand{\Tdc}{T_1\cdots T_{d-1} }

\title{The norm residue symbol  for Higher Local Fields }
\author{Jorge Fl\'{o}rez}
\address{Department of Mathematics, Borough of Manhattan Community College, City University of New York, 199 Chambers Street, New York, NY 10007, USA. jflorez@bmcc.cuny.edu}



\begin{document}
\maketitle


\begin{abstract}

In this paper we investigate  
the Kummer pairing associated to an arbitrary (one-dimensional) formal group. 
In particular, we obtain formulae describing  the values of the pairing  in terms
of  multidimensional $p$-adic differentiation, the logarithm of the formal group, the generalized trace and the norm on Milnor K-groups. The results are  a generalization to higher-dimensional local fields of 
Kolyvagin's explicit reciprocity laws. In particular, they constitute a generalization of Artin-Hasse, Iwasawa and Wiles reciprocity laws.

\end{abstract}

\tableofcontents


\section{Introduction}

\subsection{Background}\label{Introduction}

The theory of finding explicit formulations for class field 
theory has a long and extensive history. 
Among the different formulations we highlight  
the reciprocity law of Artin-Hasse \cite{Hasse} for the Hilbert 
symbol $(,)_{p^n}:\ L^{\times}\times L^{\times}\to \langle \zeta_{p^n} \rangle$, 
for $L=\mathbb{Q}_p(\zeta_{p^n})$, i.e.,
\begin{equation}\label{Hasselin}
(u\,,\,\zeta_{p^n})_{p^n}=\zeta_{p_n}^{_{\,\text{Tr}_{L/\Qp}(-\log u)\,/p^n}}, 
\end{equation}
where $p>2$ is a prime number, $\zeta_{p^n}$ is a $p^n$th primitive root of unity, and $u$ is any unit in $L$ such that 
$v_{L}(u-1)>2p^{n-1}$. 
   
   From this formula Iwasawa in \cite{Iwasawa} described the values $(u,w)_{p^n}$, for every principal unit $w$, in terms of $p$-adic differentiation:
\begin{equation}\label{Iwasawa}
(u\,,\,w)_{p^n}=\zeta_{p_n}^{_{\,\text{Tr}_{L/\Qp}(\psi(w)\, \log u)\,/p^n}},\quad \mathrm{where} \quad \psi(w)=-\zeta_{p^n}\, 
w^{-1}\,\frac{dw}{d\pi_{_n}}.
\end{equation}
Here $dw/d\pi_n$ denotes the derivative of any power series $g(x)\in \Z_p[[X]]$, such that $w=g(\pi_n)$, evaluated at the uniformizer $\pi_n:=\zeta_{p^n}-1$; i.e., $g'(\pi_n)$.
Also, it follows that the Hilbert symbol is characterized by the Artin-Hasse formula (\ref{Hasselin}). In other words: the symbol is characterized by its values on the torsion subgroup of $\mathbb{Q}_p(\zeta_{p^n})^{\times}$.

Following the work of Iwasawa, Wiles \cite{Wiles} derived analogous
 formulae to describe the Kummer pairing associated to a Lubin-Tate formal group. This pairing is a generalization of the Hilbert symbol in which the multiplicative structure of the field is replaced by  the formal group structure. Soon after, Kolyvagin \cite{koly} extended the formulae of Wiles to the Kummer pairing associated to an arbitrary formal group (of finite height.) The formulae of Kolyvagin describe the Kummer pairing in terms of $p$-adic derivations.  Kolyvagin's results also subsume those of de Shalit  in \cite{de Shalit}.
 
In this article we generalize the formulae of Kolyvagin to  arbitrary higher local fields (of mixed characteristic).
 Our formulae express the Kummer pairing associated to an arbitrary formal group with values in a higher local field--also called \textit{generalized Kummer pairing}-- in terms of  
  multidimensional $p$-adic derivations, the logarithm of the formal group,
 the generalized trace and the norm of Milnor $K$-groups. Moreover, as in the work of Iwasawa, we show 
how to construct  explicitly the multidimensional $p$-adic derivations from
 an Artin-Hasse type formula for the generalized Kummer pairing (cf. Equation \eqref{Artin-Hasse}). This shows in particular that the generalized Kummer paring is characterized by its values on the torsion points associated to the formal group.
 
 In a subsequent paper  (cf. \cite{Florez2}) we provide a refinement of these formulae to the special case of a  Lubin-Tate formal group. This has important consequences as it  gives an exact generalization of Wiles' reciprocity laws to higher  local fields. As a byproduct we  obtain exact generalizations of the formulae \eqref{Hasselin} (cf. \cite{Florez2} Corollary 5.3.1 ) and \eqref{Iwasawa} (cf. \cite{Florez2} Theorem 5.5.1 ) to higher local fields. These formulae are described in more detail below.

Generalizations of \eqref{Iwasawa} are also given by    Kurihara (cf. \cite{Kurihara} Theorem 4.4) and Zinoviev (cf. \cite{Zinoviev} Theorem 2.2 )  for the \textit{generalized Hilbert symbol} associated to an arbitrary local field.  \cite{Florez2} Theorem 5.5.1 further generalizes \eqref{Iwasawa}  to the Kummer pairing of an arbitrary Lubin-Tate formal group and an important family of higher local fields. In particular, when we take for the Lubin-Tate formal group   the multiplicative formal group $X+Y+XY$ and the higher local field to be $\Qp(\zeta_{p^n})\TT$, \cite{Florez2} Theorem 5.5.1 coincides with the results of  Kurihara and Zoniviev.  
   
   Fukaya  remarkably describes also in \cite{Fukaya}  similar formulae to those of \cite{Florez2} Theorem 5.5.1 for  the Kummer pairing  associated to an arbitrary $p$-divisible group $G$. Furthermore, Fukaya's formulae  encompass also arbitrary higher local fields (containing the $p^n$th torsion group of $G$). However, \cite{Florez2} Theorem 5.5.1 in its specific conditions is sharper than the results in \cite{Fukaya}  for Lubin-Tate formal groups as we explain in more detail below.

We finally point out to a  higher dimensional version of \eqref{Hasselin}  that can be found in Zinoviev's work ( \cite{Zinoviev} Corollary 2.1). The Corollary 5.3.1 in \cite{Florez2}  further extends   \eqref{Hasselin} to
arbitrary  Lubin-Tate formal groups and arbitrary higher local fields, in particular subsuming the formulae of Zinoviev. Moreover, we  prove  stronger results (cf.  Proposition 5.3.3 and Equation (31)  in \cite{Florez2}) which are not found, \textit{a priori}, in any of the formulae in literature. 

It is worth mentioning that the techniques  used here to
 obtain the explicit reciprocity laws were inspired by the work of Kolyvagin in
 \cite{koly}. This allows for a classical and conceptual approach
 to the higher-dimensional reciprocity laws.
\subsection{Description of the formulae}

Let  $F$ be a formal group with coefficients in the ring of integers of the local field $K/\Qp$ of finite height $h$. 
Let $S$ be a local field whose ring of integers $C$ is contained in the endomorphism ring of $F$; if $a\in C$, then $[a]_F(X)=aX+\cdots$ will denote the corresponding endomorphism.
For a fixed  uniformizer $\pi$ of $C$ we let $f:=[\pi]_F$. Let $\kappa_n\,(\simeq (C/\pi^n C)^h)$ be the $\pi^n$th torsion group of $F$ and let $\kappa=\varprojlim \kappa_n\,(\simeq C^h)$ be the Tate module.  We will fix  a basis $\{e^i\}_{1\leq i\leq h}$ for $\kappa$ and let $\{e_n^i\}_{1\leq i\leq h}$ be the corresponding reductions to  the group $\kappa_n$.

 In order to describe our formulae, let $\La\supset K$ be a $d$-dimensional local field containing the torsion group
$\kappa_n$, with ring of integers $\mathcal{O}_{\La}$ and maximal ideal $\mu_{\La}$.  
 We will denote by $F(\mu_{\La})$ the set $\mu_{\La}$ endowed with the group structure from $F$. For $m\geq 1$, we let $\La_m=\La(\kappa_m)$ and also fix a uniformizer $\gamma_m$
for $\La_m$.

We define the Kummer pairing (cf. \S \ref{The pairing (,)})
\[
  (,)_{\La,n}: K_d(\La) \times F(\mu_{\La}) \to \kappa_n \quad \mathrm{by}\quad
           (\alpha, x)\mapsto (\alpha, x)_{\La,n}=\Upsilon_\La(\alpha)(z)\ominus_Fz,
\]
where $K_d(\La)$ is the $d$th 
Milnor $K$-group of $\La$ (cf. \ref{Milnor-K-groups and their norms}), 
$\Upsilon_\La:K_d(\La)\to G_\La^{ab}$  is  
Kato's reciprocity map for $\La$ (cf. \S \ref{The reciprocity map}), $f^{(n)}
(z)=x$ and $\ominus_F$ is the subtraction 
in the formal group $F$. Denote by 
$(,)_{\La,n}^{i}$ the $i$th coordinate of 
$(,)_{\La,n}$ with respect to the 
basis $\{e_n^i\}$ of the group $\kappa_n$.

The main result in this paper is 
the following (cf. Theorem \ref{main-theorem} for the precise formulation).
\begin{theorem*}[Thm \ref{main-theorem}]
Let  $\Ma=\La_t$  for $t>>n$. Then there exists a $d-$dimensional 
derivation $\mathfrak{D}_{\Ma,m}^i$ (cf. Definition \ref{defmult}) such that
\begin{equation}\label{Main-form}
\big(\,N_{\Ma/\La} (\alpha)\,,\, x\,\big)^i_{\La,n}
=\mathbb{T}_{\Ma/S}
\left( \frac{\ \mathfrak{D}^i_{\Ma,m}(\alpha)\ }{a_1\cdots a_d}\ l_F(x)\right)  
\end{equation}
for all $\alpha=\{a_1,\dots,a_d\}\in K_d(\Ma)$ and all $x\in F(\mu_\La)$. 
Here $l_F$ is the formal logarithm, $\mathbb{T}_{\Ma/S}$ is the generalized trace (cf. \S \ref{The generalized trace}) and 
$N_{\Ma/\La}$ is the norm on Milnor 
$K$-groups. 

Moreover, we can give an explicit description for $\mathfrak{D}_{\Ma,m}^i$ as follows. Let  $\mathfrak{e}_t\in \kappa_t$ be any torsion point as in Remark \ref{special-torsion-element}, then
$$\mathfrak{D}_{\Ma,m}^i(\alpha)=- T_1\cdots T_{d-1}\,\frac{ \overline{c}_{\beta:i}}{\,l'(\mathfrak{e}_t)\,\frac{\partial \mathfrak{e}_t}{\partial T_d}\,}\,\det \left[ \frac{\partial a_{\text{k}}} {\partial T_j}\right]_{1\leq k,j\leq d}.$$ 
Here $\frac{\partial } {\partial T_j}$, $j=1,\dots, d$, denotes the partial derivatives of an element in the ring of integers of $\Ma$ with respect to the local uniformizers $T_1,\dots, T_{d-1}$, $T_d=\pi_\Ma$ (cf. Section \ref{Derivations and the module of differentials}). The constant $\overline{c}_{\beta:i}$ is an invariant of the formal group $F$
that is determined by the Artin-Hasse-type formula
 \begin{equation}\label{Artin-Hasse}
 \big(\, \{T_1,\dots, T_{d-1},u\},\, \mathfrak{e}_t\, \big)^i_{\Ma,m}=\mathbb{T}_{\Ma/S}\big(\, \log(u)\, \overline{c}_{\beta:i}\, \big),
 \end{equation}
where $ u\in V_{\Ma,1}=\{u\in \OMa:v_{\Ma}(u-1)>v_{\Ma}(p)/(p-1)\}$ and $\log$ is the usual logarithm.
\end{theorem*}

Additionally, we point to the following remarks. First, the bound on $t$ is explicit as in \cite{koly} (cf. Theorem \ref{main-theorem}). Second, the constant $\overline{c}_{\beta:i}$ has a further interpretation  as an invariant coming from the Galois representation associated to the 
Tate-module $\kappa\simeq \varprojlim \kappa_n$ ( cf. Section \ref{Definition of the invariants}). Furthermore, we can give an explicit description of this invariant 
in the important case when $F$ is a Lubin-Tate formal group as it is explained in the theorem below.

We highlight also that when  $F$ is a $p$-divisible group, Benois \cite{Benois} and Fukaya \cite{Fukaya} provide similar formulae for the Kummer pairing and remarkably give a further description of the invariant $\overline{c}_{\beta:i}$  in terms of differentials of the \textit{second kind} associated to the $p$-divisible group and the theory of Fontaine \cite{Fontaine}. On the other hand,  since Fukaya's reciprocity laws  are \`{a} la  Sen \cite{Sen}, one main difference between Fukaya's formulae and \eqref{Main-form} is the following: While Fukaya's have no restrictions on $\alpha$ for the symbol $(\alpha,x)_{\La,n}$, the formulae  \eqref{Main-form} have no restrictions on $x$. However, for a Lubin-Tate formal group we may show that, under certain conditions, the theorem above can lead to sharper results than \cite{Fukaya}, as it is explained below.

Finally, we now describe in more detail the refinement of the formulae \eqref{Main-form} in the case of a Lubin-Tate formal group $F$ that it is  addressed in the subsequent paper  \cite{Florez2}.  Let $f$ be a power series in $\OK[[X]]$ such that $f(X)\equiv \pi X \pmod{\deg2}$ and $f(X)\equiv X^q\pmod{\pi}$ for a uniformizer $\pi$ of $K$. Denote by $F_f$ the Lubin-Tate formal group associated to $f$.  In this case we have that $S=K$, $C=\mathcal{O}_K$ and $h=1$. We  fix a generator $e_f$ of $\kappa$ and let $e_{f,n}$ be the corresponding projection onto $\kappa_n$, for all $n\geq 1$. Additionally, let 
 $\kappa_{\infty}=\cup_{n\geq 1}\kappa_n$. Denote by $\Ka$  the  field $K\TT$, and by  
 $K_n$, $K_{\infty}$, $\Ka_n$ and $\Ka_{\infty}$, respectively, the  fields 
 $K(\kappa_{n})$, $K(\kappa_{\infty})$, $\Ka(\kappa_{n})$ and $\Ka(\kappa_{\infty})$, respectively. In this context, we   show in \cite{Florez2}  the following refinement of the above theorem.
\begin{theorem*}[\cite{Florez2}]
Let $L$ and $\La$  be as above. Let $r$ be maximal and $r'$  minimal  such that $K_r\subset \La\cap K_{\infty}\subset K_{r'}$. 
Take $s\geq \max\{r', n+r+\log_q(e(\La/\Ka_r))\}$; 
here $e(\La/\Ka_r)$ is the ramification index of $\La/\Ka_r$.  Then 
\begin{equation}
(\,N_{\La_s/\La}(\alpha)\,,x\,)_{\La,n}
=\left[\, 
    \mathbb{T}_{\La_s/K}
      \bigg(\, 
         QL_s(\alpha)\, l_{F_f}(x)\,
       \bigg)\,
           \right]_{F_f}(e_{f,n}),
\end{equation}
where
\[
QL_s(\alpha)=\frac{  T_1\cdots T_{d-1} }{\pi^{\,s}\,l'(e_{f,s})\, \frac{\partial e_{f,s}}{\partial T_d}\, }\, \frac{ \det \left[ \frac{\partial a_i} {\partial T_j}\right]_{1\leq i,j \leq d}}{a_1\dots a_d},
\]
for all $x\in F(\mu_\La)$ and all $\alpha=\{a_1,\dots, a_d\}\in  K_d(\La_s)':=\cap _{t\geq s}N_{\La_t/\La_s}\big(\,K_d(\La_t)\,\big)$. Here $T_d$ denotes the uniformizer $\gamma_s$ of $\La_s$.
\end{theorem*} 

The above theorem is an exact generalization of Wiles reciprocity laws to arbitrary higher local fields (cf. \cite{Wiles} Theorem 1.). 

By studying  how the theorem above is  transformed  when  varying the uniformizer $\pi$ of $K$  and the power series $f\in \Lambda_{\pi}$   we may prove, in  \cite{Florez2} Theorem 5.5.1, the following higher dimensional version of Iwasawa's reciprocity laws \eqref{Iwasawa} for $\La=\Ka_{n}$:
\begin{equation}\label{Iwasawa-Gen-Intro}
(\,\alpha\,,x\,)_{\La,n}
=\big[\, 
    \mathbb{T}_{\La/K}
      \big(\, 
         QL_n(\alpha)\, l_F(x)\,
       \big)\,
           \big]_{F_f} (e_{f,n})
\end{equation}
for all $\alpha\in K_d(\La)$ and all $x\in F_f(\mu_{\La})$ such that $v_{\La}(x)\geq 2\,v_{\La}(p)/(\varrho\,(q-1))$, where  $\varrho$ denotes the ramification index $e(K/\Qp)$. In particular, taking $K=\Qp$, $f(X)=(1+X)^{p^n}-1$, $F_f(X,Y)$ the multiplicative formal group $X+Y+XY$ and $\La$ the  cyclotomic higher local field $\Qp(\zeta_{p^n})\TT$, then \eqref{Iwasawa-Gen-Intro} coincides with   \cite{Kurihara} Theorem 4.4 and  \cite{Zinoviev} Theorem 2.2.

As we mentioned above,  Fukaya \cite{Fukaya} has similar formulae to \eqref{Iwasawa-Gen-Intro} that, more remarkably, extend to arbitrary formal groups and arbitrary higher local fields. However, for Lubin-Tate formal groups formula \eqref{Iwasawa-Gen-Intro} is sharper for $\La=\Ka_{n}$ , as the condition on $x\in F(\mu_{\La})$ in \cite{Fukaya} is $v_{\La}(x)>2v_{\La}(p)(p-1)+1$.

In the deduction of \eqref{Iwasawa-Gen-Intro} it is also shown, in \cite{Florez2} Corollary 5.3.1, the following Artin-Hasse formula for an arbitrary higher local field $\Ma\supset \Ka_{n}$:
\begin{equation}\label{Artin-Hasse-Intro}
\,
\big(\,\{T_1,\dots, T_{d-1},e_{g,n}\}\,,\,x\,\big)_{\Ma,n}
=\left[ 
\mathbb{T}_{\Ma/S}\left(\,\frac{1}{\,\xi^n\,l'_g(e_{g,n})\,e_{g,n}\,}\, l_{F_f}(x)\,\right)\,
\right]_{F_f} \,(e_{f,n})
\end{equation}
for all $x\in F_f(\mu_{\Ma})$, where $g$ is a Lubin-Tate series in $\Lambda_{\pi}$ which is also a monic polynomial and $e_{g,n}=[1]_{f,g}(e_{f,n})$; here $[1]_{f,g}$ is the isomorphism of $F_f$ and $F_g$ congruent to $X\pmod{\deg 2}$ and $l_g$ is the logarithm of $F_g$. By  taking $K=\Qp$, $f(X)=g(X)=(1+X)^{p^n}-1$, $F_f(X,Y)$ the multiplicative formal group $F_m(X,Y)=X+Y+XY$, $e_{f,n}=e_{g,n}=\zeta_{p^n}-1$, and $\Ma$ the  cyclotomic higher local field $\Qp(\zeta_{p^n})\TT$,  then \eqref{Artin-Hasse-Intro} coincides with the Artin-Hasse formula of   Zinoviev in \cite{Zinoviev} Corollary 2.1 (25).

Furthermore, \eqref{Artin-Hasse-Intro} will be deduced as a consequence of the following stronger result (cf. \cite{Florez2} Proposition 5.3.3). 
Let $\La =\Ka_{n}$ and take $e_{g,n}$ and $l_g$ as above, then
\begin{equation}\label{Artin-Hasse-complete}
\,
\big(\,\{u_1,\dots, u_{d-1},e_{g,n}\}\,,\,x\,\big)_{\La,n}
=\left[ 
\Tr\left(\,\frac{\det \left[\frac{\partial u_i}{\partial T_j}\right]}{u_1\cdots u_{d-1}}\,\frac{\Tdc}{\,\xi^n\,l'_g(e_{g,n})\,e_{g,n}\,}\, l_{F_f}(x)\,\right)\,
\right]_{F_f} \,(e_{f,n})
\end{equation}
for all units $u_1,\dots, u_{d-1}$ of  $\La$ and  all $x\in F_f(\mu_{\La})$.

 Moreover, in the particular situation where  $K=\Qp$, $\pi=p$, $f(X)=(X+1)^p-1$, $F_f(X,Y)=F_m(X,Y)$, $l_f(X)=\log(X+1)$, $\La=\Qp(\zeta_{p^n})\TT$,  we further have  an additional formula (cf. \cite{Florez2} Equation (31)):
\begin{equation}\label{Stronger-Artin-Hasse}
\big(\,\{u_1,\dots, u_{d-1},\zeta_{p^n}\}\,,\,x\,\big)_{\La,n}
=\left[ 
\mathbb{T}_{\La/\Qp}\left(\,\frac{1}{\,p^n\,}\, \frac{\det \left[\frac{\partial u_i}{\partial T_j}\right]}{u_1\cdots u_{d-1}} \,l_{F_f}(x)\,\right)\,
\right]_{F_f} (e_{f,n})
\end{equation}
for all units $u_1,\dots, u_{d-1}$ of  $\La$ and  all $x\in F_f(\mu_{\La})$.  For $u_1=T_1,\dots, u_{d-1}= T_{d-1}$ we obtain  \cite{Zinoviev} Corollary 2.1 (24).

The sharper Artin-Hasse formulae \eqref{Artin-Hasse-complete} and \eqref{Stronger-Artin-Hasse} are not contained in any of the reciprocity laws in the literature.

This paper is organized as follows. In Section \ref{The Kummer Pairing} we  review Kato's higher dimensional Local Class Field Theory
and introduce the Kummer pairing along with its properties.
In Section  \ref{The maps psi and rho} we 
introduce the different components that appear in the formulae, namely the generalized trace (\S  \ref{The generalized trace}),
 the Iwasawa maps $\psi_{\La,n}^i$ (\S \ref{The map psi_L}) and the derivations $D_{\La,n}^i$ (\S \ref{The maps D_L}). 
In Section \ref{Multidimensional derivations} we review the definitions and prove   basic properties of multidimensional derivations.
In Section \ref{Deduction of the formulas} we finally deduce the  formulae and show how to construct them  explicitly  from the Artin-Hasse-type formula (\ref{Artin-Hasse}). For convenience, we included an appendix with  statements and proofs  of several auxiliary results needed here.

The author would like to thank V. Kolyvagin for suggesting the problem treated in this article, for reading the manuscript and for providing valuable comments and improvements.
 
\subsection{Notation}\label{Terminology and Notation}

We will fix a prime number $p>2$. If $x$ is a real number then $\floor*{x}$ denotes the greatest integer $\leq x$.

 For a complete discrete valuation field $L$ we  define 
\begin{equation}\label{standar}
   L\{\{T\}\}=\left\{\; \sum_{-\infty}^{\infty}a_iT^i:\ a_i\in
   L,\ \inf v_L(a_i)>-\infty,\ \lim_{i\to -\infty}v_L(a_i)
    =+\infty \  \right\}.
\end{equation}
This is a complete discrete valuation field with valuation $v_{\La}(\sum a_iT^i)=\min_{i\in \Z} v_L(a_i)$, ring of integers $\OZ_{\La}=\OZ_{L}\{\{T\}\}$ and maximal ideal
$\mu_{\La}=\mu_L\{\{T\}\}$; here $v_L$, $\OL$ and $\mu_\La$ denote the  discrete valuation, ring of integers and maximal ideal of $L$, respectively. Observe that the residue field $k_{\La}$ of $\La$ is $k_L((\overline{T}))$, where $k_L$ is the residue field of  $L$.

The field $\La=\LTT$ is defined inductively. In particular, if $L$ is a local field, then $\La$  is a $d$-dimensional local field and  we will  
endow it with the Parshin topology (see Chapter 1 of \cite{Zhukov} or \S \ref{higher fields} of the Appendix).

 For a $d$-dimensional local field $\La\supset K$ let $T_1,\dots,T_{d-1}$ and $\pi_{\La}$ denote a system of uniformizers,  and let $k_\La=\mathbb{F}((\overline{T_1}))\cdots ((\overline{T_{d-1}}))$ be its residue field. Let $\La_{(0)}$ be the standard field $L_{(0)}\TT$, where $L_{(0)}$ is a local field unramified over $K$ with residue field $\mathbb{F}$. In particular, $\La/\La_{(0)}$ is a finite totally ramified extension. 




\section{The Kummer Pairing}\label{The Kummer Pairing}

\subsection{Higher local class field theory}\label{Higher-dimensional local class field theory}
We are now going to describe briefly the higher-dimensional
class field theory from the point of view of Milnor K--groups. This theory parametrizes the abelian extensions  of a higher local field in terms of norm subgroups of its  Milnor K-group.

\subsubsection{Milnor-$K$-groups}\label{Milnor-K-groups and their norms}

Let $R$ be a ring and $m\geq 0$. We denote by 
$K_m(R)$ the group
\[
\underbrace{R^{\times}\otimes _\mathbb{Z} \cdots \otimes _\mathbb{Z} R^{\times}}_{m-\text{times}}/I
\]
where $I$ is the subgroup of $(R^{\times})^{\otimes m}$ 
generated by 
\[
\{a_1\otimes  \cdots \otimes a_m\ :\ a_1,\dots, a_m\in R^{\times}\ 
\text{such that $a_i+a_j=1$ for some $i\neq j$}  \}
\]

$K_m(R)$ is called the $m^{th}$ Milnor-K-group of $R$. The element $a_1\otimes  \cdots \otimes a_m$ is denoted by $\{a_1, \dots ,a_m\}$.

The natural map
\[
R^{\times}\times\cdots \times R^{\times}\to K_d(R): \quad \adp\to \ad
\]
is called the symbol map and will be denoted by $\{\ \}$.

\begin{prop}\label{milnorelation} The elements of the Milnor $K$-group satisfy the relations

       \begin{enumerate}
       
              \item $\{a_1, \dots, a_i,\dots,-a_i, \dots ,a_m\}=1$
              \item $\{a_1, \dots, a_i,\dots,a_j,\dots ,a_m\}=\{a_1, \dots, a_j,\dots,a_i,\dots ,a_m\}^{-1}$       
              
       \end{enumerate}
       
\end{prop}

\begin{proof}

See Appendix \S \ref{Proofs of Chapter A}.

\end{proof}

From the definition we have $K_1(R)=R^{\times}$
 and we define $K_0(R):=\mathbb{Z}$. 
We also have a product
\[
K_n(R)\times K_m(R)\to K_{n+m}(R),
\]
where $\{a_1,\dots, a_n\}\times \{a_{n+1},\dots, a_{n+m}\}\mapsto 
\{a_1,\dots, a_{n+m}\}$.

\subsubsection{Norm on Milnor-$K$-groups}
Suppose $L/E$ is a finite extension of fields.
Then the norm $N_{L/E}:L^*\to E^*$ induces a norm on the corresponding K-groups of $L$ and $E$,
satisfying analogous properties to those of $N_{L/E}$. We recollect some of the properties in the following
\begin{prop}\label{norm}
There is a 
group homomorphism
\[
N_{L/E}: K_m(L)\to K_m(E)
\]
satisfying

\begin{enumerate}

\item When $m=1$ this maps coincides with the usual norm.

\item For the tower $L/E_1/E_2$ of finite extensions  we
 have $N_{L/E_2}=N_{E_1/E_2}\circ N_{L/E_1}$.

\item The composition $K_m(E)\to K_m(L)\overset{N_{L/E}}{\longrightarrow} 
K_m(E)$ coincides with multiplication by $[L/E]$.

\item If $\{a_1,\dots,a_m\}\in K_m(L)$, with $a_1,\dots, a_i 
\in L^{\times}$ and $a_{i+1},\dots, a_m \in E^{\times}$, then 
\[
N_{L/E}(\{a_1,\dots,a_m\})=
N_{L/E}\big(\,\{a_1,\dots,a_i\}\,\big)\,.\,\{a_{i+1},\dots,a_m\}\in K_m(E).
\]
The right hand side is the product of a norm in $ K_i(L)$ and a symbol
 in $K_{m-i}(E)$.

\end{enumerate}

\end{prop}

\begin{proof}
See \cite{Fesenko} Chapter IV and \cite{Gille} Section 7.3.
\end{proof}

Note in particular that if $a_1 \in L^{\times}$ and $a_2,\dots, a_m \in E^{\times}$, $(1)$ and $(4)$ imply
\[
N_{L/E}(\{a_1,\dots,a_m\})=\{\,N_{L/E}(a_1)\,,\,a_2,\dots,a_m\,\}.
\]

\subsubsection{Topological Milnor-$K$-groups}

Suppose $\La$ is a $d$-dimensional local field. We endow $\La^*$ 
with the product Parshin topology (see Chapter 1 of \cite{Zhukov} or \S \ref{topology on K-star} of the Appendix). This topology induces a topology on the Milnor K-group as follows 
\begin{defin}\label{topo-milnor}
We endow $K_d(\La)$ with the finest topology $\lambda_d$ for which the map
\[
(\La^*)^{\otimes d}\to K_d(\La): \quad  (a_1,\dots, a_d)\mapsto \{a_1,\dots, a_d\}
\]
is sequentially continuous in each component with respect to the 
product topology on $\La^*$ and for which subtraction in $K_d(\La)$ is
sequentially continuous. Define
\[
K^{\normalfont{top}}_d(\La)=K_d(\La)/\Lambda_d(\La),
\]
with the quotient topology where $\Lambda_d(\La)$ denotes the intersection of all neighborhoods of 0 with respect
to $\lambda_d$ (and therefore it is a subgroup).
\end{defin}

 Fesenko proved, in \cite{Zhukov} Chapter 6 Theorem 3,  that
\begin{equation}\label{Fesenko}
\Lambda_d(\La)=\cap_{l\geq 1}lK_d(\La),
\end{equation}
and also the following
\begin{prop}\label{topo-norm}
Let $\Ma/\La$ be a finite extension of $d$-dimensional local fields, then the norm 
$N_{\Ma/\La}:K_d(\Ma)\to K_d(\La)$ induces a norm
\[
N_{\Ma/\La}:K_d^{top}(\Ma)\to K_d^{top}(\La).
\]
For this norm we have $N_{\Ma/\La}(open\ subgroup)$ is open in $K^{top}_d(\La)$. In particular,
$N_{\Ma/\La}(K_d^{top}(\Ma))$ is open in $K_d^{top}(\La)$.
\end{prop}
\begin{proof}
See Section 4.8, claims (1) and (2) of page 15 of \cite{Fesenko1}.
\end{proof}

\subsubsection{The higher-dimensional reciprocity map}\label{The reciprocity map}
In the following theorem we recollect the main properties of Kato's reciprocity map that  will be needed in our formulations of the Kummer pairing. The reader is referred to \cite{Kato} for a complete account on the topic.

\begin{thm}[A. Parshin, K. Kato]\label{Kato} 
Let $\La$ be a $d$-dimensional local field. Then there 
exist a reciprocity map
\[
\Upsilon_{\La} : K_d(\La)\to \normalfont \text{Gal}(\La^{\text{ab}}/\La),
\]
satisfying the properties

\begin{enumerate}
\item If $\Ma/\La$ is a finite extension of $d$-dimensional local fields, then the following 
       diagrams commute:

\[
\begin{CD}
 K_d(\Ma) @>{\Upsilon_{\Ma}}>>  \normalfont \text{Gal}(\Ma^{\text{ab}}/\Ma)  \\
@V{N_{\Ma/\La}}VV @VV{restriction}V \\
 K_d(\La)  @>>{\Upsilon_{\La}}> \normalfont \text{Gal}(\La^{\text{ab}}/\La)
\end{CD} \hspace{30pt}
\begin{CD}
 K_d(\Ma) @>{\Upsilon_{\Ma}}>> \normalfont \text{Gal}(\Ma^{\text{ab}}/\Ma)  \\
@AAA @AA{transfer}A \\
 K_d(\La)  @>>{\Upsilon_{\La}}> \normalfont \text{Gal}(\La^{\text{ab}}/\La)
\end{CD}
\]

If moreover $\Ma/\La$ is abelian, then $\Upsilon_{\La}$ induces 
an isomorphism
\[
\begin{CD}
K_d(\La)/N_{\Ma/\La}(K_d(\Ma))@>{\Upsilon_{\La}}>> \normalfont \text{Gal}(\Ma/\La)
\end{CD}
\]

\item The reciprocity map $\Upsilon_{\La}$ is sequentially 
continuous if we endow $K_d(\La)$ with the topology $\lambda_d$ from definition \ref{topo-milnor}. 
\end{enumerate}

\end{thm}

\begin{proof}
The first assertion can be found in \cite{Kato}, Section 1, Theorem 2. 
The author has not found a formal proof of the second assertion, although it is a property that has been mentioned in other papers (cf. \cite{Zinoviev} Page 4809). We provide one in the Appendix \S \ref{Proofs of Chapter A}.
 \end{proof}


\subsection{The Kummer pairing $(,)_{\La,n}$}\label{The pairing (,)}
Let $\La$ be a $d$-dimensional local field containing 
$K$ and the group $\kappa_n$. The Kummer pairing
\[
(,)_{\La,n}: K_d(\La) \times F(\mu_{\La}) \to \kappa_n
\]
is defined by $(\alpha, x)_{\La,n}=\Upsilon_\La(\alpha)(z)\ominus_Fz$, where $f^{(n)}
(z)=x$ and $\ominus_F$ is the subtraction in the formal group $F$.

\begin{prop}\label{pair}

The pairing above satisfies the following:

\begin{enumerate}

        \item $(,)_{\La,n}$ is bilinear and $C$-linear 
               on the right.
        \item The kernel on the right is $f^{(n)}(F(\mu_\La))$.
        \item $(\alpha,x)_{\La,n}=0$ if and only if $\alpha\in N_{\La(z)/\La}(K_d(\La(z)))$, 
               where $f^{(n)}(z)=x$.
        \item If $\Ma/\La$ is finite, $x\in F(\mu_\La)$ and $\beta\in K_d(\Ma)$. 
               Then
              $$
                (\beta,x)_{\Ma,n}=(N_{\Ma/\La}(\beta),x)_{\La,n}.
             $$
        \item Let $\La\supset \kappa_m$, $m\geq n$. Then
              $$
                (\alpha,x)_{\La,n}=f^{(m-n)}\big(\,(\alpha,x)_{\La,m}\,\big)=\big(\,\alpha,f^{(m-n)}(x)\,\big)_{\La,m}.
              $$
        \item For a given $x\in K_d(\La)$, the map 
               $
                     K_d(\La)\to \kappa_n: \alpha\mapsto (\alpha,x)_{\La,n} 
               $
               is sequentially continuous.
         \item Let $\Ma$ be a finite extension of $\La$, $\alpha\in K_d(\La)$ and $y\in F(\mu_\Ma)$.
           Then
         $
          (\alpha,y)_{\Ma,n}=\big(\,\alpha,N_{\Ma/\La}^F(y)\,\big)_{\La,n},
         $
         where $N_{\Ma/\La}^F(y)=\oplus_{\sigma}y^{\sigma}$, where 
         $\sigma$ ranges over all embeddings  of $\Ma$ in $\overline{\La}$ over $\La$.
                  \item Let $t:F\to \tilde{F}$ be a isomorphism. Then
         $
         (\alpha,t(x))_{\La,n}^{\tilde{F}}=t((\alpha,x)_{\La,n}^{F})
         $
         for all $\alpha\in K_d(\La)$, $x\in F(\mu_\La)$.
\end{enumerate}

\end{prop}
\begin{proof}
         
         We will only prove property 6. The proof of the other properties is the same as in \cite{koly}  Section 3.3, or   can be found in 
                  Section \ref{Proofs of Chapter A}  of the  appendix. 
                  
                  Property 6 follows from the fact that the reciprocity map 
         $\Upsilon_{\La} : K_m(\La)\to \text{Gal}(\La^{\text{ab}}/\La)$ is sequentially continuous
          (cf. Theorem \ref{Kato} (2)). 
         Indeed, for $z$ such that $f^{(n)}(z)=x$ consider the extension $\La(z)/\La$. 
         The group $\text{Gal}(\La^{ab}/\La(z))$ is a neighborhood of $G^{ab}_\La$, so for 
         any sequence   $\{\alpha_m\}$ converging to zero in $K_d(\La)$  we can take $m$   large enough 
            such that    $\Upsilon_\La (\alpha_m)\in \text{Gal}(\La^{ab}/\La(z))$, that is $\Upsilon_\La (\alpha_m)(z)=z$,
          so $(\alpha_m, x)_{\La,n}=0$ for large enough $m$.

\end{proof}

\subsection{Sequential continuity of the pairing}\label{seqsection}

In this subsection we will show that the Kummer pairing is sequentially continuous in the second argument. This will play a vital role in showing the 
existence of the so-called Iwasawa map. This map allows us to
express the Kummer pairing in terms of the generalized trace and the logarithm 
of the formal group (cf. \S \ref{The map psi_L}).

Before we prove the sequential continuity of the pairing we need to introduce the following notation. 

\begin{defin}\label{admissible-pair}
Let $\varrho$ denote the ramification 
index of $S$ over $\Qp$. We  say that 
a pair $(n,t)$ is admissible if there 
exist an integer $k$ such that $t-1-n\geq \varrho k \geq n$. 
\end{defin}
For example, the pair $(n, 2n+\varrho+1)$ is 
admissible with $k=\floor*{(n+\varrho)/\varrho}$. Moreover,
in the special case where $\varrho=1$, then  the pair $(n,2n+1)$ is admissible with $k=n$.

Let $L$ be a local field. We will denote by $K_n$ the field $K$ after adjoining the group $\kappa_n$.  For Sections \ref{The Kummer Pairing} and \ref{The maps psi and rho} we will make the following assumptions
\begin{equation}\label{assumptions}
\begin{cases}
\quad (n,t)\ \mathrm{admissible\ pair},\\
\quad \La \supset K_t.
\end{cases}
\end{equation}

We can now formulate the following

\begin{prop}\label{continuity}
       Let  $\La$ be as in \eqref{assumptions}. For a given $\alpha \in K_d(\La)$, the map 
       \[
           F(\mu_{\La,1})\to \kappa_n\ :\ x\mapsto (\alpha,x)_{\La,n},
       \]
       is sequentially continuous in the Parshin topology; if $x_j\to x$ then $(\alpha,x_j)_{\La,n}\to (\alpha,x)_{\La,n}$. Here $F(\mu_{\La,1})$ is the set
       $
       \left\lbrace x\in \La:\, v_\La(x)\geq \floor*{v_L(p)/(p-1)} +1\right\rbrace
        $
        considered with the operation induced by the formal group $F$.
\end{prop}
\begin{remark}  We will make the following 
two assumptions during the proof of this proposition. First,   assume that $\alpha=\{a_1,\dots,a_d\}\in K_d(\La)$ is
 such that
$v_{\La}(a_1)=1$. This will imply the result for $v_{\La}(a_1)=0$ as well,
by considering $a_1=\pi_L$ and $a_1=\pi_Lu$ for any $u\in \La^{\times}$ such that
$v_\La(u)=0$; here $\pi_L$ denotes a uniformizer for $L$.

Second, we will assume that the series $r(X)=X$ 
is  $t$-normalized (cf. \S \ref{Norm Series}), 
otherwise we go to the isomorphic 
group law $r(F(r^{-1}(X),r^{-1}(Y)))$. Thus for any $m\leq t$  we will assume 
\begin{equation}\label{r=X}
\big(\{a,a_2\dots,a_d\},a\big)_{\La,m}=0\quad \forall a\neq 0\in F(\mu_{\La}).
\end{equation}
        \end{remark}     
\begin{proof}
We will drop the subscript $\La$ from the pairing 
notation. Let $x\in F(\mu_{\La,1})$  and  $\alpha=\{a_1,\dots,a_d\}\in K_d(\La)$ 
with $v_{\La}(a_1)=1$. Also, let $\varrho$ and $k$ as in Definition  \ref{admissible-pair}. Let $\mathfrak{k}=\varrho k+1$, $A(x)=a_1\oplus_Ff^{(\mathfrak{k})}(x)$ and 
$\alpha(x)=\{A(x),a_2\dots,a_d\}\in K_d(\La)$.
Then
\begin{equation}\label{salvacion}
(\alpha,x)_n=
    (\,\alpha\, \alpha(x)^{-1},\,x\,)_n\,\oplus_F\,(\alpha(x),x)_n
\end{equation}

We will show that the first term 
on the right-hand side is always zero, regardless of $x\in F(\mu_{\La,1})$, and 
that the second term goes to zero 
when we take a sequence $\{x_j\}_{j\geq 1}$ 
converging to zero. This completes the proof. 

Let us start with the second term. Let $m=n+\mathfrak{k}$. 
By (5) in the Proposition \ref{pair}  
\begin{multline}\label{identidades}
\left(\alpha(x) ,x\right)_n 
=\big(\alpha(x) ,f^{^{(m-n)}}(x)\,\big)_m
=\left(\alpha(x) ,\,A(x)\ominus_F a_1\, \right)_m\\
=\left(\alpha(x), A(x)\right)_m
    \oplus_F
      \left(\alpha(x) ,\, \ominus_F\, a_1 \, \right)_m
\end{multline}
Here $\ominus _Fa$ denotes the inverse of $a$ in the formal group law determined by $F$. From (\ref{r=X}) in the remark above,  both $(\alpha(x),A(x))_m$ and $(\alpha,a_1)_m$ are equal to zero, so we may replace $(\alpha(x),A(x))_m$
by $(\alpha^{-1},\ominus_F\, a_1)_m$ in (\ref{identidades}) to obtain
\begin{align*}
\left(\alpha(x) ,x\right)_n 
= \left(\alpha^{-1},\ \ominus_F a_1 \ \right)_m
    \oplus
          \left(\alpha(x) ,\ \ominus_F a_1 \ \right)_m 
=\left(\alpha(x)\,\alpha^{-1} ,\ \ominus_F a_1\ \right)_m.
\end{align*}
On the other hand, since $F(X,Y)\equiv X+Y \pmod{XY}$, then  $A (x)
\equiv a_1+f^{^{(\mathfrak{k})}} (x)\pmod{a_1\ f^{^{(\mathfrak{k})}} (x)}$ so dividing by $a_1$ 
\begin{equation}\label{salvita}
A(x)/a_1
   \equiv 
      1+\big(f^{^{(\mathfrak{k})}}  (x)/a_1\big)
            \pmod{f^{^{(\mathfrak{k})}}  (x)}.
\end{equation}
But $v_{\La}(a_1)=1$, so $A(x)/a_1$ is a principal unit in $\La$ for every $x\in F(\mu_{\La,1})$.
If we take a sequence $\{x_j\}_{j\geq 1}$ converging to zero in the 
Parshin topology then,  as $f:\mu_{\La,1}\to \mu_{\La,1}$ 
is sequentially continuous in the Parshin 
topology by Lemma \ref{seqcont}, we see that $A(x_j)$ approaches to 1 as $j\to \infty$. Hence 
$
\alpha(x_j)\,\alpha^{-1}\,\xrightarrow \, \{1,a_2\dots,a_d\}
$ as $j\to \infty$,
in the topology of $K_d(\La)$. Notice that 
$\{1,a_2\dots,a_d\}=\bf 1$ is the identity element in $K_d(\La)$. 
Then  by (6) in the Proposition \ref{pair} 
\[
\left(\ \alpha(x_j)\,\alpha^{-1} ,\ \ominus_F \,a_1\ \right)_m
\xrightarrow{\,j\to \infty\,} (\textbf{1},\ominus_F\, a_1)_m=0.
\]

Now we will show that first term  on the right hand 
side of equation (\ref{salvacion}) is zero by showing
 that $A(x)/a_1$
is a $p^k$th power in $\La^*$ for $x\in F(\mu_{\La,1})$. This is enough since it would imply that $\alpha(x)\,\alpha^{-1}$ is $p^k$-divisible in $K_d(\La)$, and 
from the fact that $n\leq \varrho k$, by  Definition \ref{admissible-pair}, 
we have that $\pi^n$ divides $p^k$. These two observations combined imply 
$\left(\alpha\,\alpha(x)^{-1}, x \right)_n=0$.  

To show that $A(x)/a_1$
is a $p^k$th power, let us start by observing that from Proposition \ref{logarithm}
\begin{equation}\label{exprloga}
f^{^{(\mathfrak{k})}}  (x)
      =l_F^{-1}\circ l_F(f^{^{(\mathfrak{k})}}  (x))
          =l_F^{-1}(\pi^{\mathfrak{k}}l_F(x))=\pi^{\mathfrak{k}}w,
\end{equation}
for some $w\in \mu_{\La,1}$. Then equation (\ref{salvita}) implies 
$
A(x)/a_1=1+p^kw_2
$
for some $w_2\in \mu_{\La,1}$, since $\pi^{\varrho k}=\epsilon p^k$ for some unit $\epsilon$. Then 
$
\log\left(A(x)/a_1 \right) =\log(1+p^kw_2)=p^kw_3,
$
where $w_3\in \mu_{\La,1}$. Again, by Proposition \ref{logarithm} (2), there exist a $w_4\in \mu_{\La,1}$ such that $\log(1+w_4)=w_3$. Thus
\[
A(x)/a_1 =(1+w_4)^{p^k}.
\]
\end{proof}

\subsection{Norm Series} \label{Norm Series}
A power series $r(X)\in \OK[[X]]$ such that $r(0)=0$ and $c(r)\in \OK^{*}$ is 
called $n-normalized$ if 
for every $d$-dimensional field local $\La$ containing $ \kappa_n$, it satisfies  that
\[
(\,\alpha\,,\, x\,)_{\La,n}=0,
\]
for all $x\in F(\mu_\La)$ and all $\alpha=\{a_1,\dots,a_d \}\in K_d(\La)$ such that
$a_i=r(x)$ for some $1\leq i \leq d$.

The following proposition will provide a way of 
constructing norm series.

\begin{prop}\label{normseries}

Let $g\in \OK[[X]]$, $g(0)=0$ and $c(g)\in \OK^{*}$. 
The series $s=\prod_{v\in \kappa_n}g(F(X,v))$ belongs 
to $\OK[[X]]$ and has the form $r_g(f^{(n)})$, where 
$r_g\in \OK[[X]]$. Then, the series $r_g$ is 
$n-normalized$ and  
\[
r'_g(0)=\frac{\prod_{v\neq  0\in \kappa_n} g(v)}{\pi^n}g'(0).
\]
\end{prop}

\begin{proof}

The proof of this Proposition is actually the same as in  \cite{koly} Proposition 3.1. 
It can be found in Section \ref{Proofs of Chapter A} of the Appendix.

\end{proof}



\section{The Iwasawa map}\label{The maps psi and rho}

In this section we introduce some basic properties of the generalized trace.  We also introduce the modules $R_{\La,1}$ and $R_{\La}$, necessary for the definition of the logarithmic derivatives. The main result in this section is Lemma \ref{Riezs2} which guarantees the existence of the so called Iwasawa map, and thus giving a representation of the Kummer pairing 
in terms of the generalized trace and the logarithm.
\subsection{The generalized trace }\label{The generalized trace}

Let $E$ be a complete discrete valuation field. 
Following Kurihara \cite{Kurihara}, we define a map 
\[
c_{E\{\{T\}\}/E}: E\{\{T\}\}\to E
\quad \mathrm{by}\quad c_{E\{\{T\}\}/E}(\sum_{i\in \mathbb{Z}}a_iT^i)=a_0.
\] 
Let $\Ea=E\{\{T_1\}\}\dots \{\{T_{d-1}\}\}$, we can define $c_{\Ea/E}$ 
by the composition
\[
c_{E\{\{T_1\}\}/E} \circ \cdots \circ c_{\Ea/E\{\{T_1\}\}\dots \{\{T_{d-2}\}}.
\]

\begin{lemma} \label{lemazo}

This map satisfies the following properties

\begin{enumerate}

\item $c_{\Ea/E}$ is $E$-linear.
\item $c_{\Ea/E}(a)=a$, for all $a\in E$.
\item $c_{\Ea/E}$ is continuous with respect to the the Parshin topology on $\Ea$ 
      and the discrete valuation topology on $E$.


\end{enumerate}
\begin{proof}
See \cite{Zinoviev} Lemma 2.1.
\end{proof}

\end{lemma}

Let $\La$ be a $d$-dimensional local field and let $\La_{(0)}$ be as in Section \ref{Terminology and Notation}. 
 We define
generalized trace as the the composition
\[
\Tr:= \text{Tr}_{L_{(0)}/S}\circ c_{\La_{(0)}/L_{(0)}}\circ \text{Tr}_{\La/\La_{(0)}}.
\]
Notice that if $\Ma/\La$ is a finite extension of $d$-dimensional local fields then
\[
\mathbb{T}_{\Ma/S}=\Tr\circ \text{Tr}_{\Ma/\La}
\]

The generalized trace induces the pairing
\begin{equation}\label{pareo}
\langle, \rangle: \La\times \La \to S
\quad \mathrm{defined\ by}\quad \langle x,y\rangle=\Tr(xy).  
\end{equation}

We  denote by $D(\La/S)$ the inverse of the dual of $\OLa$ 
with respect to this pairing. If $\La$ is the standard 
higher local field $L\TT$, then $D(\La/S)=D(L/S)\OLa$.
 If $\La$ is a 1-dimensional local field, the above definitions coincide with the classical definitions of local field theory.

Let  
$\text{Hom}_{C}^{c}(\La,S)$ and $\text{Hom}_{C}^{seq}(\La,S)$ 
be the group of continuous and sequentially continuous, respectively,
$C$-homomorphisms with respect to the Parshin topology on $\La$.

\begin{prop}\label{Riezs} We have an isomorphism of $C$-modules
\[
 \La \xrightarrow{\ \sim } \normalfont \text{Hom}_{C}^{seq}(\La,S):\ 
 \alpha \mapsto (x\mapsto \Tr(\alpha x)).
\]
In particular, $\normalfont \text{Hom}_{C}^{seq}(\La,S)=\text{Hom}_{C}^c(\La,S)$ since the generalized trace is continuous.
\end{prop}

\begin{proof} 
See  Section \ref{Proofs of Chapter B} of Appendix
\end{proof}
\

\subsection{Modules associated to the generalized trace}

\subsubsection{The module $R_{\La,1}$}

Let $\La$ be a $d$-dimensional local field and let $v_\La$ 
denote the valuation $\La$.
Consider
$$
\mu_{\La,1}:=\left\lbrace x\in \La:\ v_\La(x)
\geq \floor*{ v_\La(p)/(p-1)}+1\right\rbrace 
$$
 with the additive structure. Denote 
by $R_{\La,1}$ the dual of $\mu_{\La,1}$  with 
respect to the pairing (\ref{pareo}), i.e., $R_{\La,1}:=\left\lbrace x\in \La:\ \Tr(\ x\ \mu_{\La,1})\subset C\ \right\rbrace$. 
Then it can be shown (cf. \S \ref{Proofs of Chapter B} of the Appendix) that 
\begin{equation}\label{RL1}
R_{\La,1}=
   \left\lbrace x\in \La:\, v_\La(x)\geq\,  -v_\La(D(\La/S))-\floor*{ v_\La(p)/(p-1)}-1\,  \right\rbrace, 
\end{equation}
where $D(\La/S)$ is as in  Section \ref{The generalized trace}.

\begin{lemma}\label{Riezs2}
We have the isomorphism
\[
R_{\La,1}/\pi^nR_{\La,1}\ \xrightarrow{\ \sim } \ \normalfont\text{Hom}_{C}^{seq}(\mu_{\La,1},C / \pi^n C),
\]
defined by
\[
\alpha \mapsto (\ x\mapsto \Tr(\alpha x)\ ).
\]
In particular, $\normalfont\text{Hom}_{C}^{seq}(\mu_{\La,1},C / \pi^n C)=\normalfont\text{Hom}_{C}^{c}(\mu_{\La,1},C / \pi^n C)$.
\end{lemma}

\begin{proof}
          Assume first that $\La$  is the standard higher local field $L\TT$. The proof
            is done by induction in $d$. If $d=1$ the result is known (cf. \S3 and \S4 of \cite{koly}). 
            Suppose the result is true for $d\geq 1$ and let $\La=E\{\{T_d\}\}$
           where $E=L\TT$.

          Take 
           $\Phi\in \text{Hom}^{seq}_C(\mu_{\La,1},C / \pi^n C)$ and let 
           $\Phi_i(x_i)=\Phi(x_iT_d^i)$ for all $x_i \in \mu_{E,1}$. 
           Then $\Phi_i\in  \text{Hom}^c_C(\mu_{E,1},C / \pi^n C)$ and so 
           by the induction hypothesis there exists 
           $\overline{a}_{-i}\in R_{E,1}/\pi^nR_{E,1}$ such that
           \[
                  \Phi_i(x_i)=\mathbb{T}_{E/S}(\ \overline{a_{-i}}\ x_i\ ).
           \]
           Let $a_{-i}\in R_{E,1}$ be a representative of $\overline{a}_{-i}$. 
           Thus for $x=\sum x_iT_d^i\in \mu_{\La,1}$, the sequential 
           continuity of $\Phi$ implies
                       \begin{equation}\label{convergentin}
                   \Phi(x)=\sum_{i\in \Z}\Phi(x_iT_d^i)
                   =\sum_{i\in \Z}\mathbb{T}_{E/S}(a_{-i}x_i) \quad \pmod{\pi^nC}.
           \end{equation}
                      Let $\alpha=\sum a_{i}T_d^i$ and denote by $u_{i}$ the unit 
           $a_{i}/\pi_L^{v_E(a_{i})}$. We must show that
                      \begin{enumerate}[I.]
                \item $\min \{v_E(a_{i})\}>-\infty$.
                \item $v_{E}(a_{-i})\geq v_E(\pi^nR_{E,1})$ as $i \to \infty$ 
                       ( i.e., conditions 1 and 2 imply that $\alpha\in R_{\La,1}/\pi^nR_{\La,1}$).
                \item $\Phi (x)=\Tr(\alpha x)\pmod{\pi^n C}$, $\forall x\in \mu_{\La,1}$.
           \end{enumerate}
 
           Condition (I) follows immediately since 
           $a_{-i}\in R_{E,1}$, i.e., by equation (\ref{RL1})
                      \[
                 v_E(a_{-i})\geq -\big(\ 
                 \floor*{v_L(p)/(p-1)} +1+v_L\big(D(L/S)\big)\ \big)\quad \forall i\in \Z.
           \]
           Suppose condition (II) was not true. Instead of passing to a 
           subsequence we may assume for simplicity that 
           $v_{E}(a_{-i})< v_E(\pi^nR_{E,1})$ for all $i\geq 0$. 
           For an arbitrary $y\in  \mathcal{O}_L$, let
           \[
           x=y(\pi_L^{\delta_0}u_0^{-1}+\pi_L^{\delta_1}u_1^{-1}T_d
           +\pi_L^{\delta_2}u_2^{-1}T_d^2+\cdots),
           \] 
           where for any $i\geq 0$ we let $u_i=a_i/\pi_L^{v_E(a_i)}$ and
           \[ 
           \delta_i=v_E(\pi^nR_{E,1})-v_E(a_{-i})+\floor*{v_L(p)/(p-1)}\quad (\ \geq \floor*{v_L(p)/(p-1)}+1\ ).
           \]  
           Then $x\in \mu_{\La,1}$ and $a_{-i}x_{i}=\pi_L^wy$ for $i\geq 0$, 
           where $w=v_E(\pi^nR_{E,1})+\floor*{v_L(p)/(p-1)}$. 
           
           The convergence of the 
           right hand side of (\ref{convergentin}) and the fact that 
           $\mathbb{T}_{E/S}(\pi_L^wy)=\text{Tr}_{L/S}(\pi_L^wy)$
           would imply that 
                      $
                 \text{Tr}_{L/S}(\pi_L^{w}y)\in \pi^nC$  for all $ y\in \mathcal{O}_L$.
                     Thus $\pi_L^{w}/\pi^n \in D(L/S)^{-1}$, which in turn implies 
           $w\geq v_L(\pi^n)-v_L(D(L/S))$, that is,
                      \[
               v_E(R_{E,1})\geq -v_L(D(L/S))-\floor*{v_L(p)/(p-1)}.
           \]
                      This is a contradiction since $v_E(R_{E,1})=-v_L(D(L/S))-\floor*{v_L(p)/(p-1)}-1$
                      by (\ref{RL1}). 
           Finally, condition (III) immediately follows from equation (\ref{convergent}).

Assume now that $\La$ is an arbitrary $d$-dimensional local field $\La$  and consider the finite extension $\La/\La_{(0)}$, where $\La_{(0)}=L_0\TT$ and $L_0/S$ is unramified; for example $L_0=S$. In this case $R_{\La,1}\cong \text{Hom}_{\mathcal{O}_{\La_{(0)}}}(\mu_{\La,1}, \mathcal{O}_{\La_{(0)}})$ induces an isomorphism $R_{\La,1}/\pi^nR_{\La,1}\cong \text{Hom}_{\mathcal{O}_{\La_{_{(0)}}}}(\mu_{\La,1}, \mathcal{O}_{\La_{(0)}}/\pi^n\mathcal{O}_{\La_{(0)}})$. Now  for a given $\phi\in \text{Hom}^{seq}_C(\mu_{\La,1},C/\pi^nC)$ we fix an $x\in \mu_{\La,1}$ and consider $\phi_x\in \text{Hom}^{seq}_C(\mu_{\La_{_{(0)}},1},C/\pi^nC)$ defined by 
\[\phi_x(y)=\phi\left(\frac{x}{\pi_{_{\La_{_{(0)}}}}^{a}}y\right)\hspace{20pt} (y\in \mu_{\La_{(0)},1}),
\] 
where $a=v_{\La_{(0)}}(\mu_{\La_{_{(0)}},1})=\lfloor v_{\La_{(0)}}(p)/(p-1)\rfloor+1$. 
By the first part of the proof 
there exists an element $\psi(x)\in R_{\La_{(0)}}/\pi^nR_{\La_{(0)}}$ such that
\begin{equation}\label{gener}
\phi_x(y)=\mathbb{T}_{\La_{(0)}/S}\big( \,\psi(x)\,y\,\big) \quad (y\in \mu_{\La_{_{(0)}},1})
\end{equation}
Thus $\psi$ induces an element in $\text{Hom}_{\mathcal{O}_{\La_{(0)}}}(\mu_{\La,1},R_{\La_{(0)}}/\pi^nR_{\La_{(0)}})$. Since 
$$v_{\La_{(0)}}\big(\,D(\La_{(0)}/S)\,\big)=V_{L_0}\big(\,D(L_0/S)\,\big)=0,$$ 
then $v_{\La_{(0)}}(R_{\La_{(0)},1})=-v_{\La_{(0)}}(\mu_{\La_{(0)},1})=-a$ and so $\pi_{\La_{(0)}}^a\,\psi $  can be considered as an element in  $\text{Hom}_{\mathcal{O}_{\La_{(0)}}}(\mu_{\La,1},\mathcal{O}_{\La_{(0)}}/\pi^n\mathcal{O}_{\La_{(0)}})$. Thus there exists an element $\alpha\in R_{\La}/\pi^nR_{\La,1}$ such that $\pi_{\La_{(0)}}^a\,\psi(x)=\text{Tr}_{\La/\La_{(0)}}(\alpha\,x)$ for all $x\in \mu_{\La,1}$. This together with (\ref{gener}) yields the desired result.
 \end{proof}
 
\subsubsection{The module $R_{\La}$}
 Let $T_\La$ be the image of $F(\mu_\La)$ 
 under the formal logarithm $l_F$. This is a 
 $C$-submodule of $\La$ such that 
 $T_\La S=\La$. Indeed, let $x\in \La$ and take $n$ large enough 
 such that $\pi^nx\in \mu_{\La,1}$, then by Proposition 
 \ref{logarithm} there exist a $y\in F(\mu_{\La,1})$ such that 
 $\pi^nx=l_F(y)$, thus $x\in T_{\La}S$. 
 
 Let $R_\La$ be the dual of $T_\La$ with 
 respect to the trace pairing $\Tr$, then 
 by Proposition \ref{Riezs} and the fact that $\La=T_{\La}S$, we have the isomorphism
\[
R_{\La}\simeq \text{Hom}_C^{seq}(T_{\La},C).
\]
Let $\kappa_\La=\kappa\cap \La$, i.e.,
the subgroup of torsion points contained in $\La$. Then $l_F$ induces a 
continuous isomorphism 
$l_F: F(\mu_{\La})/\kappa_\La\to T_{\La}$. Thus we have the following result. 
 \begin{lemma}\label{contphi} 
 The generalized trace induces an injective homomorphism
       \begin{align*}
             0\to R_{\La}/\pi^nR_{\La}\ &\to \ \normalfont \text{Hom}_{C}^{seq}(T_\La,C / \pi^n C)\ \tilde{\to} \
\text{Hom}^{seq}_{C}(F(\mu_\La)/\kappa_L,C / \pi^n C). \\
              \alpha\ &\mapsto \hspace{5pt}\bigg(y\mapsto \Tr\big(\,\alpha\, y\big)\bigg)\hspace{5pt}
              \mapsto \bigg(\,x\mapsto \Tr\big(\,\alpha\, l_F(x)\,\big)\,\bigg).
       \end{align*}

 \end{lemma}
\begin{proof}
Immediate from the very definition of $R_{\La}$.
\end{proof}

\subsection{The map $\psi_{\La,n}^i$}\label{The map psi_L}

In this section we introduce the so-called Iwasawa map $\psi_{\La,n}^i$. This map plays a vital role in the construction of the explicit reciprocity laws. The main goal in this paper is to show
that, under certain conditions, the Iwasawa map is a logarithmic derivative.
One of the key results to achieve this is contained in
Proposition \ref{psi-properties}

Recall that we denote by $(\ ,\ )^i_{\La,n}$ 
the $i$th coordinate of the
paring $(\ ,\ )_{\La,n}$ with respect 
to the base $\{e_n^i\}$ of $\kappa_n$.

\begin{prop}\label{psi}
       Let  $\La$ be as in \eqref{assumptions}. For a given $\alpha\in  K_d(\La)$ there exist a unique element 
       $\psi^i_{\La,n}(\alpha)\in  R_{\La,1}/\pi^nR_{\La,1}$, such that
       \begin{equation}\label{formi}
              (\alpha,x)^i_{\La,n}=\Tr\big(\ \psi^i_{\La,n}(\alpha)\ l_F(x)\ \big) \quad \forall x\in F(\mu_{\La,1}),
       \end{equation}
       and the map $\psi^i_{\La,n}: K_d(\La) \to R_{\La,1}/\pi^nR_{\La,1}$ is a homomorphism.

\end{prop}

\begin{proof}

        Let us first take $\alpha$ to be an element of the form $\{a_1,\dots, a_d\}$ and consider the map 
        \[
        \omega:\ \mu_{\La,1}\to C/\pi^nC,
         \]
         defined by
         \[
        x\mapsto (\alpha,l_F^{-1}(x))^i_{\La,n}.
        \]
       By Proposition \ref{continuity} and Remark \ref{sequito} this map is 
       sequentially continuous and so by Lemma 
       \ref{Riezs2} there exist and element 
       $\psi^i_{\La,n}(\alpha)\in  R_{\La,1}/\pi^nR_{\La,1}$ 
       satisfying (\ref{formi}). 
       This defines  a map 
       $\psi^i_{\La,n}: \La^{*\oplus d}\to R_{\La,1}/\pi^nR_{\La,1}$ 
       satisfying the Steinberg relation, therefore it induces a map on  $K_d(\La)$.
       
\end{proof}

The following are some basic properties of the Iwasawa map $\psi_{\La,n}^i$.

  \begin{prop}\label{psi-properties}
Consider the finite extension $\Ma/\La$ with $\La$ satisfying  \eqref{assumptions}. Then
\begin{enumerate}
\item $\normalfont \text{Tr}_{\Ma/\La}(R_{\Ma,1})\subset R_{\La,1}$ 
      and we have the commutative diagram
       \[
      \begin{CD}
       K_d(\Ma ) @>{\psi_{\Ma,n}^i}>>  R_{\Ma,1}/\pi^nR_{\Ma,1}  \\
      @V{N_{\Ma/\La}}VV @VV{\normalfont \text{Tr}_{\Ma/\La}}V \\
      K_d(\La)  @>{\psi_{\La,n}^i}>>  R_{\La,1}/\pi^nR_{\La,1}
      \end{CD}
      \]
\item $R_{\La,1}\subset R_{\Ma,1}$ and we have the commutative  diagram
      \[
      \begin{CD}
       K_d(\La ) @>{\psi_{\La,n}^i}>>  R_{\La,1}/\pi^nR_{\La,1}  \\
      @V{\normalfont \text{incl}}VV @VVV \\
      K_d(\Ma)  @>{\psi_{\Ma,n}^i}>>  R_{\Ma,1}/\pi^nR_{\Ma,1}
      \end{CD}
      \]
     The right-hand vertical map is
      induced by the embedding of $R_{\La,1}$ in $R_{\Ma,1}$.
\item Let $L\supset K_m$, $(m,t)$ admissible and $m\geq n$. 
      Then for  $\alpha \in  K_d(\La)$, $\psi_{\La,n}^i(\alpha)$ 
      is the reduction of $\psi_{\La,m}^i(\alpha)$ from 
      $R_{\La,1}/\pi^mR_{\La,1}$ to $R_{\La,1}/\pi^nR_{\La,1}$, i.e., the diagram
      \[
      \begin{tikzpicture}[every node/.style={midway}]
\matrix[column sep={8em,between origins},
        row sep={3em}] at (0,0)
{ \node(R)  {$K_d(\La)$}  ; & \node(S) {$R_{\La,1}/\pi^mR_{\La,1}$}; \\
 \node(U) {$$} ; & \node(R/I) {$R_{\La,1}/\pi^nR_{\La,1}$};                   \\};
\draw[<-] (R/I) -- (R) node[anchor=east]  {$\psi^i_{\La,n}$};
\draw[<-] (R/I) -- (S) node[anchor=west]  {$reduction$};
\draw[->] (R)   -- (S) node[anchor=south] {$\psi^i_{\La,m}$};
\end{tikzpicture}
\]
commutes.
      \item If $t:(F,e_i)\to (\tilde{F},\tilde{e}_i)$ is an isomorphism, 
      then 
      \[ 
      \tilde{\psi}_{\La,n}^i=\frac{1}{t'(0)}\psi_{\La,n}^i.
      \]
\end{enumerate}
\end{prop}
\begin{proof}

See Section \ref{Proofs of Chapter B} of Appendix.

\end{proof}

 We are ready to prove the following key result.  

\begin{prop}\label{key}

            Let $\La$ and $t$ as in \eqref{assumptions}. Let
            $r(X)$ be a $t$-normalized series. Then
                                \begin{equation}\label{key-result}
     (\alpha, x)^{^i}_{_{\La,n}}=\Tr\left(\log(a_1) \left[\ \frac{-\psi^i_{L,n}
             \big(\alpha(r(x)\big)\,r(x)\, l_F'(x)}{r'(x)}\ \right] \right),
            \end{equation} 
              where $\alpha$ and $\alpha(r(x))$ are 
              elements in $K_d(\La)$ defined, respectively, by 
              \[\alpha=\{a_{1},\dots,a_{d}\}\quad \text{ and}  \quad 
               \alpha\big(r(x)\big)=\big\{r(x),a_2,\dots,a_{d}\big\}, \] 
             with $a_1\in V_{\La,1}=1+\mu_{\La,1}$  and $x\in F(\mu_\La)$. Here $\log$ is the usual logarithm.
\end{prop}
\begin{proof}
We follow the same ideas as in the proof of  Proposition 4.1 of \cite{koly}. To simplify the notation we will denote 
the normalized valuation $v_\La/v_\La(p)$ by $v$ 
and  also omit the subscripts $\La$ and $F$ from the pairing notation, the formal sum $\oplus_F$, and the formal logarithm $l_F$.  
Furthermore, we are going to denote by 
$\alpha(a)$ the element $\{a,a_2,\dots,a_{d}\}$ in $K_d(\La)$. 

Let $\alpha=\alpha(a_1)$ with $a_1\in V_{\La,1}$ and let $x\in F(\mu_{\La})$. Observe that we may assume that $r(X)=X$ since we can go 
to the isomorphic formal group $\tilde{F}=r(F(r^{-1}(X),r^{-1}(Y)))$ 
with torsion points $\tilde{e^i}=r(e^i)$ through the isomorphism 
$r: F\to \tilde{F}$. The series $\tilde{r}=X$ is $t$-normalized 
for $\tilde{F}$ and if the result were true for $\tilde{F}$ 
and $\tilde{r}$ then
\begin{align*}
(\alpha,x)^i_{F,n}&=(\alpha,r(x))^i_{\tilde{F},n}
=\Tr\big(\ \log(u)\, \big[\,-\tilde{\psi}^i_{n}\big(\alpha(r(x)\big)\,\big)\,l_{\tilde{F}}'(r(x))\,r(x)\,\big]\ \big).
\end{align*}
Since $\tilde{\psi}^i_{n}(\alpha(r(x)))=r'(0)^{-1}\psi^i_{n}(\alpha(r(x)))$ 
and $l_{\tilde{F}}(r(X))=r'(0)l_F(X)\Rightarrow l_{\tilde{F}}'(r(X))
=r'(0)l_F'(x)/r'(x)$  the result follows.

We assume therefore that $r(X)=X$. For the admissible pair $(n,t)$ 
 let $\varrho$ and $k$ be as in Definition \ref{admissible-pair}. 
 Denote by $\epsilon$ the unit $\pi^{\varrho k}/p^k$. 
Let $u\in V_{\La,1}=\{x\in \La\ :\ v(x-1)>1/(p-1)\}$ and 
$x\in \mu_\La$. By bilinearity of the pairing
\begin{equation}\label{uno}
(\alpha, x)_{n}=\big(\,\alpha,x\ominus(x\, a_1^{p^k})\,\big)_n\,\oplus\, (\alpha, x\,a_1^{p^k})_n.
\end{equation}
Now let $m=n+\varrho k $ and $y=x\, a_1^{p^k}$. Then by (5) of Proposition \ref{pair}
\begin{align*}
&(\alpha,y)_n\, =\,f^{^{(m-n)}}\big(\, (\alpha,y)_m\,\big)\,
=\,\pi^{\varrho k}(\alpha,y)_m\,
=\,\epsilon\big(\,\alpha\big(a_1^{p^k}\big)\,,y\,\big)_m\\
&=\,\epsilon\big(\,\alpha(y/x)\,,\,y\,\big)_m \,
=\,\epsilon(\alpha(y),y)_m\,\ominus\, \epsilon (\alpha(x),y)_m.
\end{align*}
But $r=X$ is $t$-normalized, hence we may replace 
$(\alpha(y),y)_m=0$ by the expression $(\alpha(x),x)_m=0$ and obtain
\begin{align}\label{dos}
(\alpha,y)_n\,=\,\epsilon \big(\alpha(x),x\big)_m\,\ominus\, \epsilon\big(\alpha(x),y\big)_m\,
=\,\epsilon\big(\alpha(x),x\ominus y\big)_m.
\end{align}
By the properties of the logarithm in Proposition \ref{logarithm} we can express
\[
a_1^{p^k}
   =\exp\left(\,\log\left(a_1^{p^k}\right)\,\right)
   =\exp\big(\,p^k\log(a_1)\,\big)=1+p^k\log(a_1)+p^{2k}w,
\]
where $w=\frac{z^2}{2!}+p^k\frac{z^3}{3!}+\cdots$, 
with $z=\log(a_1)$. Since $v(z)>1/(p-1)$ then 
$v(\frac{z^i}{i!})>1/(p-1)$ and so $v(w)>1/(p-1)$. This follows, for example, by 
Proposition  2.4 of \cite{koly}.

Since $x\ominus y\equiv x-y \pmod{xy}$  and $y=xa_1^{p^k}$ with $v(x)>0$, 
then $v(x\ominus y)=v(x-y)=v(x(a_1^{p^k}-1))>1/(p-1)$. 
Thus, using the Taylor expansion of $l=l_F$ around $X=x$ we obtain
\[
l(y)=l(x+xp^kz+xp^{2k}w)=l(x)+l'(x)(xp^kz+xp^{2k}w)+p^{2k}w_1,
\]
where $w_1=l''(x)\frac{\delta^2}{2!}+p^kl^{(3)}(x)\frac{\delta^3}{3!}+\cdots$ 
with $\delta=xz+xp^kw$. Since $v(\delta)>1/(p-1)$ then $v(w_1)>1/(p-1)$. Moreover,
$
l(y)=l(x)+l'(x)xp^kz+p^{2k}w_2,
$
 with $v(w_2)>1/(p-1)$. Then
  \[
 l(x\ominus y)= -l'(x)xp^kz-p^{2k}w_2.
 \]
  Observing that $-l'(x)xz-p^{k}w_2\in \mu_{\La,1}$ we  
 have by the isomorphism given in \ref{logarithm} 
 that there is an $\eta \in F(\mu_{\La,1})$  such 
 that $l(x\ominus y)=p^kl(\eta)=-l'(x)xp^kz-p^{2k}w_2$. 
 Thus  
 \begin{equation}\label{p-divisible}
 x\ominus y=[p^k](\eta)=[\pi^{\varrho k}](\tilde{\eta} )=f^{(\varrho k)}(\tilde{\eta}),
 \end{equation} 
 for $\tilde{\eta}=[\epsilon^{-1}]_F(\eta)$.  Since $n\leq \varrho k$ 
 then $\pi^n$ divides $\pi^{\varrho k}$ and we have that  
 $x\ominus y\in f^{(n)}(F(\mu_{\La,1}))$. Thus 
 the first term on the right hand side of 
 equation (\ref{uno}) is zero. By equations (\ref{dos}) and (\ref{p-divisible})
 and item  (5) of Proposition \ref{pair}  we have
  \begin{align}\label{ofa}
         (\alpha,x)_n=
         \epsilon(\alpha(x),x\ominus y)_m
           =(\alpha(x),\eta)_n.
 \end{align}
Since $v(\eta)>1/(p-1)$ and $(n,t)$ is admissible we can use Proposition \ref{psi},
\begin{align*}
(\alpha(x),\eta)_n&=\Tr(\ \psi^i_{n}\big(\alpha(x)\big)\ l_F(\eta)\ )
=\Tr(\ \psi^i_{n}\big(\alpha(x)\big)\ ( -l'(x)xz-p^{k}w_2)\ ).
\end{align*}

Since $\psi^i_{L,n}\big(\alpha(x)\big)\in R_{L,1}$, 
$w_2\in T_{L,1}$, $\pi^n|p^k$ and
$\Tr(R_{L,1}T_{L,1})\subset C$, 
then we can write the last expression above as
$
\Tr(\ \psi^i_{n}\big(\alpha(x)\big)\ ( -l'(x)xz)\ ).
$
From Equation (\ref{ofa}) we finally get 
\[
 (\alpha,x)^i_n=\Tr\big(\, \psi^i_{n}\big(\alpha(x)\big)\, \big[ -l'(x)xz\big]\, \big).
\]
Keeping in mind that $z=\log (a_1)$, the proposition follows.

\end{proof}

\subsection{The map $\rho_{\La,n}^i$}\label{The map rho_L}
We can define a $C$-linear structure on 
$V_{\La,1}=1+\mu_{\La,1}$ by using the isomorphism 
$\log: V_{\La,1}\to T_{\La,1}$, i.e., $cu:=\log^{-1}(c\log(u))$ for $c\in C$ and $u\in V_{\La,1}$. 
Let $x\in F(\mu_\La)$ and $\alpha'=\{a_2,\dots, a_{d}\}\in K_{d-1}(\La)$ 
be fixed. Consider the mapping
\[
      V_{\La,1}\to C/\pi^n C,
\quad
  \text{    defined by}
\quad
      a_1\mapsto (\{a_1,\dots,a_d\}, x)^i_{\La,n}.
\]
According to Proposition \ref{key} this is a continuous $C$-linear
map and we have the following 

\begin{prop}\label{rho}
Let $\alpha'\in K_{d-1}(\La)$ and $x\in F(\mu_\La)$. Consider the element   $\alpha=\{a_1\}\cdot \alpha'\in K_d(\La)$ where $a_1\in V_{\La,1}$. Then there exist a unique element 
    $\rho_{\La,n}^i(\alpha',x)\in R_{\La,1}/\pi^nR_{\La,1}$ 
    such that
   \begin{equation}\label{rhopa}
      (\alpha, x)^i_{\La,n}
      =\Tr \left(\, \log(a_1)\, \rho_{\La,n}^i(\alpha',x)\, \right ),
    \end{equation}
    Moreover,  the map 
    $\rho_{\La,n}^i:  K_{d-1}(\La)\otimes F(\mu_\La)\to R_{\La,1}/\pi^nR_{\La,1}$ 
    is a homomorphism.
\end{prop}

From Proposition \ref{rho} and Proposition \ref{key} it follows the next proposition.

\begin{prop}\label{psi-rho}
      Let $\La$ and $t$ be as in \eqref{assumptions}  and let
            $r(X)$ be a $t$-normalized series. Then for all $x\in F(\mu_{\La})$ 
           and all $\alpha'=\{a_2,\dots, a_{d}\}\in K_{d-1}(\La)$ we have
            \[
             \frac{\psi^i_{\La,n}
             \big(\,\alpha(r(x))\,\big)\ r(x)}{r'(x)} 
             =-\,\frac{\ \rho_{\La,n}^i(\alpha',x)\ }{l_F'(x)},           
            \]
            where $\alpha\big(r(x)\big)=\big\{r(x),a_2,\dots,a_{d}\big\}\in K_d(\La)$.
           
\end{prop}


\subsection{The maps $D_{\La,n}^i$}\label{The maps D_L}

Let $\La$ be as in \eqref{assumptions}. 
Define $D_{\La,n}^i:\mathcal{O}_\La^d\to 
R_{\La,1}/\pi^nR_{\La,1}$ by 
\begin{equation}\label{D-maps}
D_{\La,n}^i(a_1,\dots,a_d)=\psi_{\La,n}^i
(\{a_1\dots,a_d\})a_1\cdots a_d,
\end{equation}
and consider $R_{\La,1}/\pi^nR_{\La,1}$ 
as an $\mathcal{O}_\La$-module. 

\begin{prop}\label{Dar}The map $D_{\La,n}^i$ satisfies
\begin{enumerate}
      \item Leibniz Rule:
      \[\Scale[0.9]{D_{\La,n}^i(a_1,\dots,a_ja'_j,\dots,a_d)
         =a_jD_{\La,n}^i(a_1,\dots,a'_j,\dots,a_d)+a'_j
         D_{\La,n}^i(a_1,\dots,a_j,\dots,a_d).}\]
               \item Steinberg relation:
     $D_{\La,n}^i(a_1,\dots,a_d)=0\quad \text{if $a_j+a_k=1$ for some  $j\neq k$.}
     $
          \item Skew-symmetric:
     $$D_{\La,n}^i(a_1,\dots,a_j,\dots,a_k,\dots,a_d)=-D_{\La,n}^i(a_1,\dots,a_k,\dots,a_j,\dots,a_d).$$
        \item $D_{\La,n}^i(a_1,\dots,a_d)=0$ if some $a_j$ is a $p^k$-th power in $\OLa$ with  $\pi^n|p^k$.

\end{enumerate}
\end{prop}

\begin{proof}

Property (1) follows from the fact that 
$\psi_{\La,n}^i$ is a homomorphism. 
Property (2) follows from the Steinberg relation 
$\{a_1,\dots,a_j,\dots,1-a_j,\dots,a_d\}=1$ for elements in the Milnor K-group  
$K_d(\La)$ and property (3) follows from the fact 
that $$\{a_1,\dots,a_j,\dots,a_k,\dots,a_d\}=
\{a_1,\dots,a_k,\dots,a_j,\dots,a_d\}^{-1},$$ in 
$K_d(\La)$ ( cf.  Proposition \ref{milnorelation}). 

\end{proof}
\begin{prop}\label{relation of D}
     Let $\La$ and $t$  be as in \eqref{assumptions} and $r(X)$ a $t$-normalized series. Then 
for all  $x,y\in F(\mu_{\La})$ and all $a_2,\dots,a_{d}\in\La^*$  we have 
     \begin{equation}\label{Dcont}
                       \frac{D_{\La,n}^i\big(\,\alpha(r(x\oplus y)\,\big)}{r'(x\oplus y)}
            =
             F_X(x,y)\ 
            \frac{D_{\La,n}^i\big(\,\alpha(r(x)\,\big)}{r'(x)}+  
            F_Y(x,y)\ 
            \frac{D_{\La,n}^i\big(\,\alpha(r(y)\,\big)}{r'(y)},      
      \end{equation}
     where $\alpha(z)$ denotes $(z,a_2,\dots, a_d)\in\OLa^d$.
\end{prop}

\begin{proof}
This follows from the fact that
\[
\rho_{\La,n}^i(\alpha',x\oplus _F y)
=\rho_{\La,n}^i(\alpha',x)+
\rho_{\La,}^i(\alpha',y),
\]
for $\alpha'=\{a_2,\dots, a_d\}\in K_{d-1}(\La)$ and from differentiating $l_F(F(X,Y))=l_F(X)+l_F(Y)$ with respect to $X$ and $Y$.
\end{proof}

Let $\tilde{F}$ be the formal group $r(F(r^{-1}(X),r^{-1}(Y)))$. 
The series $r(X)=X$ is $t$-normalized for $\tilde{F}$. 
Denote by $\tilde{\oplus}$  the sum according to this formal group 
and $\tilde{D}_{\La,n}^i$, $\tilde{\psi}^i_{\La,n}$ the 
corresponding maps. According to Proposition 
\ref{psi-properties} (4) we have that 
\begin{equation}\label{dtilde}
\tilde{D}_{\La,n}^i=\frac{1}{r'(0)}D_{\La,n}^i.
\end{equation}
Therefore, Proposition \ref{relation of D} in 
terms of $\tilde{\oplus}$ and $D_{\La,n}^i$ reads as
\begin{coro}\label{Dcontinu}
         Let $\La$ be as in \eqref{assumptions}, then 
     \begin{equation}\label{Dcontigo}
              D_{\La,n}^i\big(\,\alpha(x\tilde{\oplus} y)\,)
            = \tilde{F}_X(x,y)D_{\La,n}^i\big(\,\alpha(x)\,\big)+
              \tilde{F}_Y(x,y)D_{\La,n}^i\big(\,\alpha(y)\,\big).      
      \end{equation}

\end{coro}

In order to take advantage of this differentiability property of 
the map $D_{\La,n}^i$ with respect to formal group law $\tilde{F}$, 
we will show that any element in the maximal ideal $\mu_{\La}$
can be expressed as an infinite sum with respect to the formal group law $\tilde{F}$. This is accomplished in the following subsection.


\subsubsection{Representations of elements as formal group series}

In order to simplify the notation, we introduce the following
\begin{defin}
Let $J_n=\big\{\,\overline{i}=(i_1,\dots, i_n):0\leq i_1,\dots, i_{d-1 }<p^n\big\}$. Let $\mathfrak{A}$ be the set of all series in $X_d\OLa[[X_1,\dots, X_d]]$ of the form 
\begin{equation}\label{tomita}
\bigoplus _{k=1}^{\infty}\left(\, \bigoplus_{\overline{i}\in J_n}\gamma_{\overline{i},k}^{p^n}\ X_1^{i_1}\cdots X_{d-1}^{i_{d-1}}X_d^k\right),
\end{equation}
where  $\gamma_{i,k}\in \OLa$. Here $\oplus$ denotes addition in the formal group law $F$.
\end{defin}

The following lemma will be used in Proposition \ref{formal-sums-rep} to express every element in the maximal ideal as an infinite formal group sum. 
\begin{lemma}\label{Kolyv}
For $x\in \OLa$, there exist elements $\gamma_{\overline{i}}\in \OLa$, with $\overline{i}\in J_n$, such that
\[
x
\equiv 
\sum_{\overline{i}\in J_n}
\gamma_{\overline{i}}^{p^n}\  
T_1^{i_1}\cdots T_{d-1}^{i_{d-1}} \pmod{\pi_\La}.
\]
\end{lemma}

\begin{proof}
For a direct  proof using induction see \S\ref{Proofs of Chapter B} of the Appendix. Alternatively, the proposition is equivalent to the following two facts:
\begin{enumerate}
\item  $k_L((T_1))\dots ((T_{d-1}))$ is a finite 
extension of $k_L((T_1^{p^n}))\dots ((T_{d-1}^{p^n}))$ of degree $p^{n(d-1)}$ and generated by
the elements $T_1^{i_1}\cdots T_{d-1}^{i_{d-1}}$ for $0\leq i_1,\dots, i_{d-1}<p^n$. 
\item  $k_L((T_1^{p^n}))\dots ((T_{d-1}^{p^n}))$ is the image of $k_L((T_1))\dots ((T_{d-1}))$ 
under the Frobenius homomorphism 
\[
\sigma_p:k_L((T_1))\dots ((T_{d-1}))\to k_L((T_1))\dots ((T_{d-1})),
\]
i.e., every element of $k_L((T_1^{p^n}))\dots ((T_{d-1}^{p^n}))$ 
is a $p^n$th power of an element in $k_L((T_1))\dots ((T_{d-1}))$.
\end{enumerate}
Both facts are  easily proven  by induction from the fact that, 
for a field $k$ of characteristic $p$, the extension $[k((T)):k((T^{p}))]$
has degree $p$ and $\sigma_d(k)((T^p))$ is the image of $k((T))$ under the Frobenius
homomorphism $\sigma_p:k((T))\to k((T))$.

\end{proof}
\begin{prop}\label{formal-sums-rep}
 For $y\in \mu_\La$, there exist a power series $\eta\in \mathfrak{A}$ such that
$y=\eta(T_1,\dots, \pi_\La)$.
\end{prop}
\begin{proof}
Fix $y\in\mu_\La$. Denote by  $T^{\,\overline{i}}$ 
 the product $T_1^{i_1}\cdots T_{d-1}^{i_{d-1}} $ for $\overline{i}\in J_n$ and
by $\oplus$ the formal sum $\oplus_F$. Then  Lemma \ref{Kolyv} applied to $y/\pi_\La$ gives as elements $ \gamma_{\,\overline{i},1}\in \OLa$ such that
\[
\frac{y}{\pi_\La}\,
\equiv\,
\sum_{\,\overline{i}\in J_n} 
 \gamma_{\,\overline{i},1}^{\,p^n}\,
 T^{\,\overline{i}}\,
 \equiv\,
 \frac{1}{\pi_L}\bigoplus_{\overline{i}\in J_n} 
      \gamma_{\,\overline{i},1}^{\,p^n}\,
             T^{\,\overline{i}}\,\pi_\La \pmod{\pi_\La}.
\]
In other words, 
$
y\equiv \bigoplus_{\overline{i}\in J_n} 
 \gamma_{\overline{i},1}^{p^n}\,
 T^{\,\overline{i}}\,\pi_\La \pmod{\pi_\La^2}.
$
Denote by $y_1$ the formal sum $\bigoplus_{\,\overline{i}\in J_n} 
 \gamma_{\overline{i},1}^{p^n}\,T^{\,\overline{i}}\,\pi_\La$. Suppose we have defined, for $1\leq k\leq m-1$, elements 
 \[y_k=\bigoplus_{\overline{i}\in J_n} 
 \gamma_{\,\overline{i},k}^{\,p^n}\,
 T^{\,\overline{i}}\,\pi_\La^k,\ \gamma_{\,\overline{i},k}\in \OLa,
  \quad \mathrm{ such\ that} \quad 
 y\ominus (\oplus_{k=1}^{m-1}y_k)\equiv 0 \pmod{\pi_\La^{m}} .
 \]
 Then, again by Lemma \ref{Kolyv}, there exist elements $\gamma_{\,\overline{i},m}\in \OLa$ such that
\[
\frac{1}{\pi_\La^{m}}\ \big(\,y\ominus (\oplus_{k=1}^{m-1} y_k)\,\big)\,
\equiv\,
\sum_{\overline{i}\in J_n} 
 \gamma_{\,\overline{i},m}^{\,p^n}\,
 T^{\,\overline{i}}\,
 \equiv\,
 \frac{1}{\pi_L^m}\bigoplus_{\overline{i}\in J_n} 
 \gamma_{\,\overline{i},m}^{\,p^n}\,
T^{\,\overline{i}}\,\pi_\La^m \pmod{\pi_\La}.
\]
 Denote by $y_m$ the formal sum $\bigoplus_{\overline{i}\in J_n} 
 \gamma_{\overline{i},m}^{p^n}\,
T^{\,\overline{i}}\,\pi_\La^m$. Then
 \[
 y\ominus (\oplus_{k=1}^{m-1} y_k)\equiv y_m\pmod{\pi_\La^{m+1}}
\quad \mathrm{ or}\quad 
 y\ominus (\oplus_{k=1}^{m} y_k)\equiv 0 \pmod{\pi_\La^{m+1}}.
 \]
 Therefore $y=\oplus_{k=1}^{\infty} y_k$, which is what we wanted to prove.
\end{proof}

\begin{coro}\label{sumativa}
For every $x\in \OLa$, there exists $\gamma_{\,\overline{i},k}\in \OLa$ such that
\[
\sum _{k=0}^{\infty}\
\sum_{\overline{i}\in J_n}\,
\gamma_{\overline{i},k}^{p^n}\ 
T_1^{i_1}\cdots T_{d-1}^{i_{d-1}}\pi_\La^k\,
.
\]
\end{coro}
\begin{proof}
Take $F$ to be the additive formal group $X+Y$.
\end{proof}

\subsubsection{Differentiability properties of the map $D_{\La,n}^i$}

Let $r$ be a $t$-normalized series for $F$. 
Let $\tilde{F}$ be the formal group $r(F(r^{-1}(X), r^{-1}(Y)))$.  
For $y\in \mu_{\La}$ we will denote by $\tilde{\eta}_y(X_1,\dots, X_d)$ the multivariable series 
\begin{equation}\label{tomitazo}
\tilde{\bigoplus} _{k=1}^{\infty}
\left( 
\tilde{\bigoplus}_{\overline{i}\in J_n}
\gamma_{\,\overline{i},k}^{\,p^n}\ 
X_1^{i_1}\cdots X_{d-1}^{i_{d-1}}\,X_d^k
\right), \quad (\gamma_{\,\overline{i},k}\in \OLa)
\end{equation} 
with respect to $\tilde{F}$, such that 
$y=\tilde{\eta}_y(T_1,\dots,\pi_\La)$,
whose existence is guaranteed by Proposition \ref{formal-sums-rep}. 
\begin{prop}\label{derivat}
      Let $\La$ be as in \eqref{assumptions}. 
      For $y=\tilde{\eta}_y (T_1,\dots, \pi_\La)\in \mu_{\La}$,  $\tilde{\eta_y}$ as 
       in \eqref{tomitazo},
      we have  
   \[
        D^i_{\La,n}(a_1, \dots,a_{d-1},y)
       =\sum_{i=1}^{d} \, 
       \frac{\partial \tilde{\eta}_{y}}{\partial X_i}\bigg|_{\substack{X_k=T_k,\\ k=1,\dots, d}}\,
        D^i_{\La,n}(a_1,\dots,a_{d-1},T_i), 
        \]
        for all $a_1,\dots, a_{d-1}\in \OLa$, where $T_d=\pi_\La$.
\end{prop}

\begin{proof}
     Let $y=\oplus_{k=1}^{\infty}y_k$, where
     \[
     y_k=\tilde{\bigoplus}_{\,\overline{i}\in J_n} 
 \gamma_{\,\overline{i},k}^{\,p^n}\,
 T_1^{i_1}\cdots T_{d-1}^{i_{d-1}}\,\pi_\La^k, \quad (\gamma_{\,\overline{i},k}\in \OLa)
     \]
     Thus, $\tilde{\eta}_y=\tilde{\oplus}_{m=1}^{\infty}\eta_m$, where 
     $$\eta_m(X_1,\dots, X_d)
     =\tilde{\bigoplus}_{\,\overline{i}\in J_n} 
 \gamma_{\,\overline{i},k}^{\,p^n}\,
 X_1^{i_1}\cdots X_{d-1}^{i_{d-1}}\,X_d^m. $$
     Let us fix $a_1,\dots, a_{d-1}\in \La^*$ and denote by $D(x)$ the element
     $D_{\La,n}^i(a_1,\dots ,a_{d-1}, x)$  to simplify notation.

      First notice that 
      \begin{equation}\label{limita}
      v_\La(y_k)\xrightarrow{ k\to \infty} \infty\quad  \Longrightarrow \quad  D(y_k)=0\quad \mathrm{for\ \textit{k}\ large\ enough}.
                  \end{equation}
      This follows from $D_{\La,n}^i(a_1,\dots ,a_{d-1}, x)=\psi^i_{\La,n}(a_1,\dots ,a_{d-1}, x)\,a_1\cdots a_{d-1}\,x $
     and the fact that  $\psi_{\La,n}^i$ has values in $R_{\La}/\pi^nR_{\La,n}$, i.e.,  $\pi_\La^n\psi_{\La,n}^i=0$.
     
      Thus, from
     $\tilde{\eta}_y=\tilde{\oplus}_{m=1}^{k-1}\eta_m\pmod{\pi^k}$ and equation (\ref{limita}), it is enough to consider
     the finite formal sum $\tilde{\oplus}_{m=1}^{k-1}\eta_m$. The proposition follows now from
     the  fact that $D(\gamma^{p^n})=p^n\gamma^{p^n-1}D(\gamma)=0$ (  $\pi^n|p^n$),  
     $D(xy)=yD(x)+xD(y)$ , and corollary \ref{Dcontinu}.

\end{proof}

\begin{coro}\label{dermax}

 Let $\La$ be as in \eqref{assumptions}.  
Let  $y_i=\tilde{\eta}_{y_i}(T_1\dots, \pi_\La)\in \mu_{\La}$, 
for $1\leq i \leq  d$, where $\eta_{y_i}$ is a 
multivariable series  of the form \eqref{tomitazo}. 
Then
\begin{equation}\label{esfero}
        D_{\La,n}^i( y_1,\dots, y_d )
       =  
       \det \left[  \frac{\partial \tilde{\eta} _{y_i} }{\partial X_j} 
       \right]_{i,j}\bigg|_{_{\substack{X_k=T_k,\\ k=1,\dots,d}} } 
       \  D_{\La,n}^i(T_1,\dots,T_d ),
\end{equation}
where $T_d=\pi_\La$.
\end{coro}
\begin{proof}
The proof is immediate, see Section \ref{Proofs of Chapter B} of the Appendix.
\end{proof}

From the above corollary we see that the map $ D_{\La,n}^i$ 
behaves like a multidimensional derivation.
Our goal in the following sections is to give conditions that will guarantee this. 
In particular, it will follow  that \eqref{esfero} not only holds for elements in the maximal ideal $\mu_{\La}$ but in the full ring of integers $\OLa$ and, moreover, it is independent
of the power series representing the elements $y_k$, $k=1\dots, d$.



\section{Multidimensional derivations}\label{Multidimensional derivations}
In this section we recall the main properties of
 multidimensional derivations and set them in the context needed to deduce our formulae. We start by introducing one dimensional derivations. 

We will use the following notation and assumptions. Let $\La$ be a $d$-dimensional local field 
with local uniformizers $T_1,\dots, T_{d-1}$ and $\pi_\La$. Let $W$ be an 
$\mathcal{O}_\La$-module that is $p$-adically complete, i.e., 
\[
W\cong \varprojlim W/p^nW.
\]
 For example, if $p^nW=0$ for some $n$, then $W$ is $p$-adically complete. 
Actually, this is going to be our situation, since $W$ will be  the $\mathcal{O}_\La$-module $\RL/\pi^n\RL$. 
\subsection{Derivations and the module of differentials}\label{Derivations and the module of differentials}

\begin{defin}\label{definamos}
     A derivation  of  $\mathcal{O}_\La$ into $W$ over $\mathcal{O}_K$ is a map 
     $D: \mathcal{O}_\La\to W$ such that for all $a,b\in \OLa$ we have
     \begin{enumerate}
          \item $D(ab)=aD(b)+bD(a)$.
          \item $D(a+b)=D(a)+D(b)$.
          \item $D(a)=0$ if  $a\in \OK$.          
     \end{enumerate}
  
\end{defin}

We denote by $ D_{\OK}(\OLa,W)$ the $\OLa$-module of all derivations $D:\OLa\to W$.
The universal object in the category of derivations of $\OLa$ over $\OK$  is the $\OLa$-module 
of Khaler differentials, denoted by $\Omega_{\OK}(\OLa)$. 
This is the $\OLa$-module generated by finite 
linear combinations of the  symbols $da$, 
for all $a\in \OLa$, divided out by the submodule generated by 
all the expressions of the form $dab-adb-bda$ and 
$d(a+b)-da-db$  for all $a,b\in \OLa$ and $da$ for all $a\in \OK$. 
The derivation $\mathfrak{d}: \OLa\to \Omega_{\OK}(\OLa)$ is defined by sending 
$a$ to $da$.

If $D:\OLa\to W$ is a derivation, then $\Omega_{\OK}(\OLa)$ 
is universal in the following way. There exist a unique homomorphism 
$\xi: \Omega_{\OK}(\OLa)\to W$ of $\OLa$-modules  such 
that the diagram
 \[
      \begin{tikzpicture}[every node/.style={midway}]
\matrix[column sep={8em,between origins},
        row sep={3em}] at (0,0)
{ \node(R)  {$\OLa$}  ; & \node(S) {$\Omega_{\OK}(\OLa)$}; \\
 \node(U) {$$} ; & \node(R/I) {$W$};                   \\};
\draw[<-] (R/I) -- (R) node[anchor=east]  {$D\hspace{7pt}$};
\draw[<-] (R/I) -- (S) node[anchor=west]  {$\xi$};
\draw[->] (R)   -- (S) node[anchor=south] {$\mathfrak{d}$};
\end{tikzpicture}
\]
is commutative.

Let $\hat{\Omega}_{\OK}(\OLa)$ be the $p$-adic completion of $\Omega_{\OK}(\OLa)$, i.e.,
\[
\hat{\Omega}_{\OK}(\OLa)=\varprojlim \Omega_{\OK}(\OLa) /p^n\Omega_{\OK}(\OLa).
\]
Since we are assuming that $W$ is $p$-adically 
complete, the homomorphism $\beta$ induces the homomorphism
\[
\xi:\hat{\Omega}_{\OK}(\OLa)\to W. 
\]

Denote by  $\mathfrak{D}(\La/K)\subset \OLa$ 
the annihilator ideal of the torsion part  of  $\hat{\Omega}_{\OK}(\OLa)$. 
Then we have the following proposition.

 \begin{prop}\label{Khaler} The module $\hat{\Omega}_{\OK}(\OLa)$ is generated by 
 the elements $d\,T_1,\dots, d\,T_{d-1}$, $d\pi_{\La}$,
 and there is an isomorphism of $\OLa$-modules
 \[
 \hat{\Omega}_{\OK}(\OLa)
\cong 
 \OLa ^{\oplus (d-1)} 
\oplus (\OLa/\mathfrak{D}(\La/K) ).
 \]
Moreover, if $D(\La/K)$ is as in Section \ref{The generalized trace}, then 
\[
\mathfrak{D}(\La/K)\,|\, D(\La/K).
\]

In particular, if $\La$ is the standard higher local field $L\TT$ we have the isomorphism of $\OLa$-modules
\[ \hat{\Omega}_{\OK}(\OLa)
\cong 
 \OLa dT_1\oplus \cdots 
\oplus \OLa dT_{d-1}
\oplus (\OLa/D(L/K)\OLa )d\pi_L,
\]  
where $\pi_L$ is a uniformizer for $L$ 
and $D(L/K)$ is the different 
of the extension $L/K$. Thus in this case
\[
\mathfrak{D}(\La/K)=D(L/K)\OLa.
\]
 \end{prop}
 \begin{proof} See \cite{Kurihara2} Section 12.0 (b). For an alternative proof see Section \ref{Proofs of Chapter C} of the Appendix.

 \end{proof}
 
\subsubsection{Canonical derivations}\label{Canonical Derivations} We will now define what it means to differentiate an element in $\OLa$ with respect to 
the uniformizers $T_1,\dots, T_{d-1}$ and $T_d=\pi_{\La}$. 

First we will assume $\La$ is the standard higher local field $L\TT$. Then
 we will define derivations over $\OK$ 
 \begin{equation}\label{partialito}
 D_k:\OLa\to \OLa\quad (\,1\leq k \leq d-1\,)
 \end{equation} as follows.
 Let $\La_0=L$ and
 $\La_k=\La_{k-1} \{\{T_{k}\}\}$, $k=1,\dots, d-1$. 
 For $1\leq k \leq d-1$, we define the derivation of  $\mathcal{O}_{\La_k}$ over $\OK$
  \begin{equation}\label{partial-derivat}
D_k:\mathcal{O}_{\La_k}\to \mathcal{O}_{\La_k}
\quad \mathrm{by}\quad  D_k\big(\,\sum a_i\,T_{k}^{\,i}\,\big)=\sum a_i\, i\,T_{k}^{\,i-1}\qquad (a_i\in \mathcal{O}_{\La_{k-1}}).  
 \end{equation}
 We now lift these 
 derivations to derivations of $\OLa$ over $\OK$, by induction, in the following way. 
 Suppose
 $D: \mathcal{O}_{\La_{k-1}}\to  \mathcal{O}_{\La_{k-1}}$ is a derivation of  $\mathcal{O}_{\La_{k-1}}$ over $\OK$. Then $D$
 extends to a  derivation of $\mathcal{O}_{\La_{k}}$ over $\OK$
 \[
D:\mathcal{O}_{\La_{k}}\to \mathcal{O}_{\La_{k}} 
\quad \mathrm{by}\quad
D\left(\ \sum _ia_i\ T_{k}^i\ \right):= \sum_{i} D(a_i)\ T_{k}^i.
\]
 This derivation is well-defined since $D$ is continuous 
 with respect to the valuation topology
 of $\mathcal{O}_{\La_{k-1}}$. Thus the derivations \eqref{partialito}
 are well-defined and we can now introduce the following definition.
 \begin{defin}\label{partial-derivations}
 The derivations $D_k:\OLa\to \OLa$, $1\leq k \leq d-1$  from Equation \eqref{partialito}, 
 will be denoted by $\frac{\partial}{\partial T_k}$ for $1\leq k\leq d-1$.
  \end{defin}

We now assume $\La$ is any $d$-dimensional local field not necessarily standard. Let $\La_0\subset \La$ be the standard local field defined in Section \ref{Terminology and Notation}. For $g(x)= a_0+\cdots+a_kX^k\in \mathcal{O}_{\La_0}[X]$ we denote by $\frac{\partial g}{\partial T_i}(X)$, $i=1,\dots, d$, the polynomial
\[
\frac{\partial a_k}{\partial T_i}X^k+\cdots+\frac{\partial a_0}{\partial T_i}\in \mathcal{O}_{\La_0}[X] \quad (i=1,\dots, d-1)
\] 
and 
\[\frac{\partial g}{\partial T_d}(X)=g'(x)\quad (i=d)
\]
If $a\in \OLa$, let $g(x)\in \mathcal{O}_{\La_0}[X]$ such that $a=g(\pi_\La)$. Then we denote by $\frac{\partial a}{\partial T_i}$ the element $\frac{\partial g}{\partial T_i}(\pi_\La)$, $i=1,\dots, d$. Even though this definition depends on the choice of $g(X)$, for the purpose of Proposition \ref{derival} below, this choice is immaterial.

If $p(X)$ denotes the minimal polynomial of $\pi_\La$ over the extension $\La/\La_0$, then
from Proposition \ref{Khaler} the equation $p(\pi_\La)=0$ implies that the elements $dT_1,\dots, dT_{d-1}$ and $d\pi_\La$ satisfy the relation
    \[
    \frac{\partial\, p}{\partial T_1}(\pi_{\La})\,d\,T_1+\cdots+ \frac{\partial\, p}{\partial T_{d-1}}(\pi_{\La})\,d\,T_{d-1}+p'(\pi_{\La}) \,d\,\pi_\La=0,
    \]
    and 
    \[
     v_{\La}(\mathfrak{D}(\La/K))=\min\left\{\, v_\La\left(\frac{\partial\, p}{\partial T_1}(\pi_\La)\right), \dots, v_\La\left(\frac{\partial\, p}{\partial T_{d-1}}(\pi_\La)\right),\ v_\La\big(\,p'(\pi_\La)\,\big)\, \right\}.   
    \]
    We now deduce the following proposition.
\begin{prop}\label{derival}
Let $D:\OLa\to W$ be a derivation over $\OKa$. Then the following relation holds
\[
 \frac{\partial p}{\partial T_1}(\pi_{\La})D(T_1)+\cdots+ \frac{\partial p}{\partial T_{d-1}}(\pi_{\La})D(T_{d-1})+p'(\pi_{\La}) D(\pi_\La)=0 
\] and for all $a\in \OLa$ we have
\[
D(a)= \frac{\partial a}{\partial T_1}D(T_1)+\cdots+ \frac{\partial a}{\partial T_{d-1}}D(T_{d-1})+\frac{\partial a}{\partial T_{d}}D(\pi_\La).
\]

  Conversely,    let $w_1,\dots, w_d\in W$ such that 
       \[
        \frac{\partial p}{\partial T_1}(\pi_{\La})w_1+\cdots+ \frac{\partial p}{\partial T_{d-1}}(\pi_{\La})w_{d-1}+p'(\pi_{\La}) w_d=0.  
       \]    
       Then the map $D:\OLa \to W$  defined by
          \begin{equation}\label{defderi}
              D(a)
              :=
              \sum_{k=1}^d\  
              \frac{\partial a }{\partial T_k}\ w_k ,
          \end{equation}
          is a well-defined derivation from   $\OLa$ into $W$ over $\mathcal{O}_K$.
          
            As a  particular case, if $w_1,\dots, w_d\in W$  are annihilated by  $\mathfrak{D}(\La/K)$, then \eqref{defderi} defines a derivation over $\OK$. 
          \end{prop} 
\begin{proof}
     The proof follows from the proof of Proposition \ref{Khaler} 
     and the fact that $D_{\OK}(\OLa,W)\cong \text{Hom}_{\OLa}(\hat{\Omega}_{\OK}(\OLa),W)$.
    \end{proof}

\subsubsection{Derivations on one-dimensional local fields} 
The following three propositions are taken from \cite{koly} and will be needed in the deduction of our formulae. We put them here for convenience. 

If $L$ is a finite extension of the local field $K$, 
then we denote by $\Omega_{\OK}(\OZ_L)\cong \OZ_L/D(L/K)$ 
the $\OZ_K$-module of differentials of $\OZ_L$ over $\OZ_K$.

\begin{prop}\label{module-of}
\noindent \begin{enumerate}

   \item $\Omega_{\OK}(\OZ_L)\cong \OZ_L/D(L/K)$ as $\OZ_L$-modules. 
          Moreover, the element $d\pi_L$ generates $\Omega_{\OK}(\OZ_L)$. 

   \item If $M$ is a finite extension of $L$, the homomorphism 
         $\Omega_{\OK}(\OZ_L)\to \Omega_{\OK}(\OZ_M)$ is an embedding.
         
\end{enumerate}

\begin{proof}
cf. \cite{koly} Proposition 5.1.
\end{proof}

\end{prop}

We will denote by $K_m$ (resp. $L_m$) the field 
obtained by adjoining  the $m$-th torsion 
points to $K$ (resp. $L$), i.e., $K_m=K(\kappa_m)$.  
Let $v$ denote the normalized valuation $v_M/v_M(p)$, 
for every finite extension  $M$ of $\Qp$.

\begin{prop}\label{constants}

     There are explicit positive constants $c_1$, $c_2\in \mathbb{R}$, depending only on $(F,\pi)$, such that 
     
     \begin{enumerate}
     
          \item $v(D(L_m/L))\leq m/\varrho+\log_2(m)/(p-1)+c_2$ and 
          $v(D(K_m/K))\geq m/\varrho-c_1$.
          Where $\varrho$ is the ramification index of $S$ over $\Qp$.
          
          \item Let $p_m$ be the period ( i.e., the generator of the annihilator ideal) 
                of the $\mathcal{O}_{K_m}$-submodule of 
                $\Omega_{\mathcal{O}_{K}}(\mathcal{O}_{K_m})$ 
                generated by differentials $de_m^j$, $j=1,\dots,h$. Then $v(p_m)\geq m/\varrho-c_1$. 
                
     \end{enumerate}
     
\end{prop}
\begin{proof}
cf. \cite{koly} Proposition 5.3.
\end{proof} 
\begin{remark}\label{Explicit-constants-description}
According to Kolyvagin (cf. \cite{koly} page 325) we may take for $c_2$ the constant
\[
c_2=\frac{2}{p-1}+\frac{2v(\pi)p}{(p-1)^2}+v(D(K_{2\varrho -1}/K)).
\]
As for the constant $c_1$, if $K/S$ is an unramified extension we may take 
\[
c_1=\frac{1}{\varrho\,(\,q_S^h-1\,)},
\]
where $q_S$ is the cardinality of the residue field of $S$.
\end{remark}
\begin{remark}\label{special-torsion-element}
By Proposition \ref{constants} (2), there exists a torsion point $e^j_m$, with $1\leq j\leq h$, 
which we will denote by $\mathfrak{e_m}$, such that $v(\mathfrak{p}_m)\geq m/\varrho-c_1$. Here $\mathfrak{p}_m$ denoted the period of the $\mathcal{O}_{K_m}$-submodule of 
                $\Omega_{\mathcal{O}_{K}}(\mathcal{O}_{K_m})$  generated by $d\mathfrak{e}_m$.
\end{remark}

\begin{remark}\label{Bound general L}
If $\La\supset \kappa_n$ is a $d$-dimensional local field then the inequality Proposition \ref{constants}(2) also holds, namely $v(D(\La_m/\La))\leq m/\varrho+\log_2(m)/(p-1)+c_2$. This
 follows from the fact that $v(D(\La_m/\La))\leq v(D(K_m/K))$.
\end{remark}

\begin{prop}\label{constantes}

     Suppose $K/S$ is an unramified extension and let $q=|k_S|$. 
     Let $h$ be the height of $F$ with respect to $C=\OS$, cf. 
     Proposition \ref{torsiones}. Then 
     
     \begin{enumerate}
     
          \item $v(D(K_m/K))\geq m/\varrho-1/\varrho(q^h-1)$.
          \item  $K(e_m^i)$ is totally unramified over $K$ and $e_m^i$ 
          is a uniformizer for $K(e_m^i)$. Moreover $$v(\ period\ of\ de_m^i)=m/\varrho -1/\varrho(q^h-1).$$
           In particular, if $h=1$ then $K_m=K(e_m^1)$ and $D(K_m/K)=\pi^m/\pi_1\mathcal{O}_{K_m}$, where $\pi_1$ is a uniformizer for $K_1$.     
     \end{enumerate}
     
\end{prop}

\begin{proof}
cf. \cite{koly} Proposition 5.6.
\end{proof} 
    
\subsection{Multidimensional derivations}\label{Multidimensional derivations sub}
Let $\La$, $\tilde{\La}$, $L$, $\tilde{L}$ $K$ and $W$ 
     be as in the beginning of the previous section.

\begin{defin}\label{defmult}
      
     A $d$-dimensional derivation of $\OLa^d$ into $W$ over $\OK$ is  
     map $D: \OLa^d\to W$ such that for all $a_1\dots, a_d$ and $a_1',\dots,a_d'$ in $\OLa$, and any $1\geq i\geq d$,   it satisfies
     \begin{enumerate}
      \item  Leibniz rule: 
      $\Scale[0.9]{D(a_1,\dots,a_ia_i',\dots, a_d)
                     =a_i'D(a_1,\dots, a_i,\dots, a_d)+a_iD(a_1,\dots,a_i',\dots, a_d)}$.         
          \item Linearity: $\Scale[0.9]{
          D(a_1,\dots,a_i+a_i',\dots, a_d)
                     =D(a_1,\dots, a_i,\dots, a_d)+D(a_1,\dots,a_i',\dots, a_d),}$.
                       \item Alternating: $D(a_1,\dots, a_d)=0$,
                 if $a_i=a_j$ for  $i\neq j$.
                     
          \item Constants: $D(a_1,\dots,a_d)=0$ if $a_i\in \OK$ for some $1\leq i \leq d$. 
     \end{enumerate}
  
\end{defin}
We denote by $D_{\OK}^d(\OLa^d,W)$ the $\OLa$-module 
of all $d$-dimensional derivations $D: \OLa^d\to W$. 

Consider the $d$th exterior product
$\Omega^{\,d}_{\OK}(\OLa):=\bigwedge ^d_{\OLa} \Omega_{\OK}(\OLa)$ 
(cf. \cite{Lang} Chapter 19 \S 1 ).  This is the $\OLa$-module 
$\Omega_{\OK}(\OLa)\otimes \cdots \otimes \Omega_{\OK}(\OLa)$
divided out by the $\OLa$-submodule generated by the elements
\begin{equation}\label{alternado}
x_1\otimes \cdots \otimes x_d,
\end{equation}
where $x_i= x_j$ for some $i\neq j$. 
The symbols $x_1\otimes\cdots \otimes x_d$ will be denoted by
\[
x_1\wedge \cdots \wedge x_d,
\]
instead.
For  $ \Omega_{\OK}^{\,d}(\OLa)$
we consider the $d$-dimensional derivation 
$$\mathfrak{d}:\OLa^{\,d}\ \longrightarrow \ \Omega_{\OK}^{\,d}(\OLa)$$
that sends $\adp$ to the wedge product 
$a_1\wedge \cdots \wedge a_d$. 
This $\OLa$-module is the 
universal object in the category of 
$d$-dimensional derivations of $\OLa$ over $\OK$, i.e, 
\begin{prop}\label{multidiff}
If $D:\OLa^{\,d}\to W$ is a $d$-dimensional 
derivation over $\OK$ then there exist a homomorphism
$\beta: \Omega_{\OK}^{\,d}(\OLa)\to W$ of $\OL$-modules such that 
the diagram
 \[
      \begin{tikzpicture}[every node/.style={midway}]
\matrix[column sep={8em,between origins},
        row sep={3em}] at (0,0)
{ \node(R)  {$\OLa^{\,d}$}  ; & \node(S) {$ \Omega_{\OK}^{\,d}(\OLa)$}; \\
 \node(U) {$$} ; & \node(R/I) {$W$};                   \\};
\draw[<-] (R/I) -- (R) node[anchor=east]  {$D\hspace{7pt}$};
\draw[<-] (R/I) -- (S) node[anchor=west]  {$\lambda$};
\draw[->] (R)   -- (S) node[anchor=south] {$\mathfrak{d}$};
\end{tikzpicture}
\]
is commutative.
\end{prop}
\begin{proof}
It is clear that  $\lambda$ is the 
homomorphism defined by 
$\lambda (da_1\wedge \cdots \wedge da_d )
=D(a_1,\dots, a_d)$.
\end{proof}

   \begin{prop}\label{Khaler2}
   The $\OLa$-module $\hat{\Omega}_{\OK}^{\,d}(\OLa)$ is generated by $d\,T_1
\wedge \cdots\wedge 
d\,T_{d-1}\wedge d\pi_{\La}$ and we have an isomorphism of $\OLa$-modules
\[
\hat{\Omega}_{\OK}^{\,d}(\OLa)\ \cong\ \OLa/\mathfrak{D}(\La/K)
\]
 Furthermore, if $\La$ is the standard field $L\TT$ then 
 $$ \hat{\Omega}_{\OK}^{\,d}(\OLa)
\cong (\OLa/D(L/K)\OLa)\ d\,T_1
\wedge \cdots\wedge 
d\,T_{d-1}\wedge d\pi_L$$ as $\OLa$-modules, where $\pi_L$ is a uniformizer for $L$ 
and $D(L/K)$ is the different of the extension $L/K$. 
Thus in this particular case
\[
\mathfrak{D}(\La/K)=D(L/K)\OLa.
\]
 \end{prop} 
 \begin{proof}
 This follows immediately from Lemma \ref{Khaler}.
 \end{proof}
 
 The above proposition implies
 \[
 \mathfrak{D}(\La/K)\,|\,D(\La/K).
 \]

 We will show in the proposition below how every multidimensional derivation is characterized by the derivations in Definition \ref{partial-derivations}.

\begin{prop}\label{multder}
     Let $D\in D_K^d(\OLa^d,W)$. Then  $\mathfrak{D}(\La/K)$ annihilates $D(T_1,\dots, \pi_\La)$, and 
     for $a_1,\dots,a_d\in \OLa$ we have
    \begin{equation}\label{multideri}
     D(a_1,\dots, a_d)=\det\left[ \frac{\partial a_{\text{i}} }
     {\partial T_{\text{j}}}\right]_{_{ 1\leq i,j \leq d}}  D(T_1,\dots, \pi_\La). 
     \end{equation}
    
    Conversely, we can construct a multidimensional derivation in the following way. 
    Let $w\in W$ be annihilated by  $\mathfrak{D}(\La/K)$, 
     then the map
     $$
          D(a_1,\dots, a_d):=\det\left[ \frac{\partial a_i 
          }{\partial T_j}\right]_{_{ 1\leq i,j\leq d}} \, w ,
     $$
     is well-defined and belongs to $D_K^d(\OLa^d,W)$.
     
      In other words, 
     the map $D\mapsto D(T_1, \dots, \pi_\La),$ defines an isomorphism from 
      $D_K^d(\OLa^d,W)$ to the $\mathfrak{D}(\La/K)$-torsion subgroup of $W$.
\end{prop}
\begin{proof}
This follows from Proposition \ref{Khaler2} and the fact that 
$$D_{\OK}^d(\OLa^{\,d},W)\cong \text{Hom}_{\OK}\big(\, \hat{\Omega}_{\OK}^{\,d}(\OLa),W\big).$$
\end{proof}



\section{Deduction of the formulae}\label{Deduction of the formulas}
We finally deduce the main formulae to 
describe the Kummer pairing for the 
field $\La$ and the level $n$, i.e., $(,)^i_{\La,n}$. 
The strategy of the proof is the following. We start by introducing an auxiliary finite field extension $\Ma$ of $\La$, containing sufficiently many torsion points, and
also introduce higher levels $k$ and $m$ with $k\geq m\geq n$. 
In Section \ref{description}, we show that, under certain conditions, the map $D_{\Ma,k}^i$
is a derivation. In Section \ref{Invariants of the representation tau}, we introduce
an Artin-Hasse type formula for the generalized Kummer pairing of level $k$. This formula
is characterized by an invariant attached to the Tate representation. With this invariant, we
  descend to  the level $m$ and manufacture an explicit derivation $\mathfrak{D}_{\Ma,m}^i$ in Definition \ref{Explicit-D}  and then show that it coincides with  $D_{\Ma,m}^i$. 
The final descend back to the field $\La$ and level $n$ 
will be guaranteed by Proposition \ref{phi-psi} and accomplished in full detail 
in Section \ref{Main formulas}.

For the rest of this section we will use the following notation and assumptions:
 \begin{equation}\label{assumptions-2}
\begin{cases}
\quad \beta=(k,t)\ \mathrm{admissible\ pair},\\
\quad \pi^k|D(K_t/K),\\
\quad \Ma \supset K_t.
\end{cases}
\end{equation}
Additionally,  $\pi_\Ma$ will denote a uniformizer for $\Ma$ and $\pi_t$ a uniformizer for $K_t$. We will also denote by $\Ka_t$ the field $K_t\TT$. The above conditions, and the fact that 
$D(K_t/K)=\mathfrak{D}(\Ka_t/\Ka)$ (cf. Proposition \ref{Khaler}), imply
\[
\pi^k\,|\mathfrak{D}(\Ka_t/\Ka)|\,\mathfrak{D}(\Ma/K)
\]

Note that the assumption  of \eqref{assumptions-2} on the different $D(K_t/K)$ is satisfied, for example, when $(t-k)/\varrho\geq c_1$; here $c_1$
is the constant from Proposition \ref{constants}.

\subsection{The reduction of the map $D_{\Ma,k}^i$}\label{description}

\begin{prop}\label{Disaderiv}

     Let $\Ma$ be as in \eqref{assumptions-2}. Then the reduction 
     \[
          D_{\mathcal{M},k}^i:\mathcal{O}^d_\Ma
            \longrightarrow 
                   R_{\Ma,1}/\big(\, (\pi^k/\pi_\Ma)\,R_{\Ma,1}\,\big)
     \]
     of $D_{\Ma,k}^i$ to $R_{\Ma,1}/(\pi^k/\pi_\Ma)R_{\Ma,1}$ is a 
     $d$-dimensional derivation over $\mathcal{O}_K$.

\end{prop}

\begin{proof}
     Let us fix   $\ada \in \OMa$. From Proposition \ref{derival} and  
     the fact that 
     \begin{align*}
     \mathfrak{D}(M/K)D_{\Ma,k}^i(\ada, T_j)&=0 \pmod{\pi^kR_{\Ma,1}},\quad j=1,\dots, d-1, \\
     \mathfrak{D}(M/K)D_{\Ma,k}^i(\ada, \pi_\Ma)&=0 \pmod{\pi^kR_{\Ma,1}},
          \end{align*}
      we can  construct a  derivation 
          \[
     D:\mathcal{O}_M\to R_{\Ma,1}/\pi^kR_{\Ma,1},
     \] 
          such that $D(\pi_\Ma)=D_{\Ma,k}^i(\ada, \pi_\Ma)$ and  
          $D(T_k)=D_{\Ma,k}^i(\ada, T_j)$, $j=1,\dots,d-1$, in the following way
     \begin{align*}
          D(a)= 
          \sum_{j=1}^{d-1}\frac{\partial a }{\partial T_j}D(T_j)+
          \frac{\partial a }{\partial T_d}D(\pi_\Ma),
      \end{align*}
     where $a\in \OMa$. 

     According to Proposition \ref{derivat} both $D$ and $D_{\Ma,k}^i(\ada,\cdot)$ 
     coincide in $\mu_\Ma$. But from the Leibniz rule it follows, by comparing 
     $D(\pi_\Ma x)$ and $D^i_{\Ma,k}(\ada,\pi_\Ma x)$ when $x\in \mathcal{O}_\Ma$, that 
     they coincide $\pmod{(\pi^k/\pi_\Ma)R_{\Ma,1}}$ in all of $\mathcal{O}_\Ma$.

     It follows now that 
     \[
          D_{\Ma,k}^i:\mathcal{O}^d_\Ma\longrightarrow R_{\Ma,1}/\big(\, (\pi^k/\pi_\Ma)\,R_{\Ma,1}\,\big)
     \] 
    satisfies all conditions from definition \ref{defmult} and so 
     by Proposition \ref{multder} we have that it is a $d$-dimensional derivation such that 
     \[
           D_{\Ma,k}^i(a_1, \dots,a_d)=\det 
          \left[ \frac{\partial a_{\text{i}}}{\partial T_{\text{j}}} \right]_{_{1\leq  \text{i,j}\leq d}}  
              \, D_{\Ma,k}(T_1,\dots, T_{d-1},\pi_\Ma), 
          \]
     where $a_1,\dots,a_d \in \OMa$.

\end{proof}

\subsubsection{Description of the map $\psi_{\Ma,m}^i$ in terms of $D_{\Ma,m}^i$}

Observe that for any $m\leq k$ the pair $(m,t)$ is also admissible, so
Proposition \ref{Disaderiv} also holds for this pair. Thus, 
$D^i_{\Ma,m}:\OMa^d\to  R_{\Ma,1}/(\pi^m/\pi_{_\Ma})R_{\Ma,1}$ is also a derivation
as well
and, moreover,
we can express the map $\psi^i_{\Ma,m}$ out of 
$D^i_{\Ma,m}$ 
as follows.  For $u_1,\dots,u_d$  in 
$\mathcal{O}_\Ma^*=\{x\in \mathcal{O}_\Ma:v_\Ma(x)=0\}$ we let
\begin{equation}\label{psi-from-D}
\left\{
\begin{aligned}
\psi^i_{\Ma,m}(u_1,\dots, u_{d-1},\pi_{\Ma})&=\frac{D^i_{\Ma,m}(u_1,\dots, u_{d-1},\pi_{\Ma})}{u_1\cdots u_{d-1}\pi_{\Ma}} 
\pmod{\frac{\pi^m}{\pi_\Ma^2}R_{\Ma,1}}\\
\psi^i_{\Ma,m}(u_1,\dots,u_d)&=\frac{D^i_{\Ma,m}(u_1,\dots,u_d)}{u_1\cdots u_d} \pmod{\frac{\pi^m}{\pi_\Ma}R_{\Ma,1}}\\
\psi^i_{\Ma,m}(u_1,\dots, \pi_\Ma^ku_d)&=k\psi^i_{\Ma,m}(u_1,\dots, \pi_\Ma)+\psi^i_{\Ma,m}(u_1,\dots,u_d),\ k\in \Z\\
\psi^i_{\Ma,m}(a_1,\dots,a_d)&=0, \text{whenever $a_i=a_j$ for $i\neq j$, $a_1\dots, a_d\in \Ma^*$}
\end{aligned}
\right.
\end{equation}
It is clear form the definition that this is independent from the choice of
a uniformizer $\pi_\Ma$ of $\Ma$.  Notice that the fourth
property says that $\psi^i_{\Ma,m}$ is alternate, in particular
it is skew-symmetric, i.e.,
\[
\psi^i_{\Ma,m}(a_1,\dots, a_i,\dots, a_j,\dots,a_d)=-\psi^i_{\Ma,m}(a_1,\dots, a_j,\dots, a_i,\dots,a_d).
\]
whenever $i\neq j$.

\subsubsection{Descending from $\Ma$ to $\La$ and from level $m$ to  level $n$}
The following proposition will be used in the main result ( Theorem \ref{main-theorem})   to descend from the auxiliary field $\Ma$ to the ground field $\La$ and from the level $m$ to the level $n$.

Let $\La\supset \kappa_n$ be a $d$-dimensional local field and $v$ denote the normalized 
valuation $v_{\La}/v_{\La}(p)$.

\begin{prop}\label{phi-psi}

       Let $m$, $n$ be integers  such that $v(f^{(m-n)}(x))>1/(p-1)$ for all $x\in F(\mu_\La)$ (cf. Remark \ref{acotarazo}). 
        Let $(m,t)$ be admissible and  put $\Ma=\La_t$. 
        Then $\normalfont \text{Tr}_{\Ma/\La}$  
        induces a homomorphism from 
        $R_{\Ma,1}/\pi^mR_{\Ma,1}$ to $\La/\pi^nR_{\La}$ and  we have the representation
        \begin{equation}\label{miguelina}
        (\,N_{\Ma/\La}(\alpha)\,,\,x\,)_{\La,n}^i=
        \Tr\big(\, \text{Tr}_{\Ma/\La}\big(\,\psi^i_{\Ma,m}(\alpha)\,\big)\, l_F(x)\, \big), 
        \end{equation}
        for all $\alpha\in K_d(\Ma)$ and  $ x\in F(\mu_{\La})$. In particular, $\normalfont \text{Tr}_{\Ma/\La}(\psi^i_{\Ma,m}(\alpha))$ belongs to $R_\La/\pi^nR_\La$ and it is the unique element satisfying \eqref{miguelina}.                    

\end{prop}
\begin{proof}
This proof was inspired by  Proposition 6.1 of \cite{koly}. Since $e_n^i=f^{(m-n)}(e_m^i)$ and 
$f^{(m-n)}(x)\in \mu_{\La,1}\subset \mu_{\Ma,1}$, 
then  by Proposition \ref{pair} (4) and (5) we have 
\begin{align*}
&(N_{\Ma/\La}(\alpha),x)^i_{\La,n}
=\frac{1}{\pi^{m-n}}\big(\,\alpha\,,\,f^{(m-n)}(x)\,\big)^i_{\Ma,m}\\
&=\frac{1}{\pi^{m-n}}\mathbb{T}_{\Ma/S} 
     \left(\, \psi^i_{\Ma,m}(\alpha)\, l\big(f^{(m-n)}(x)\big)\, \right)
=\Tr \big(\ \big[\, \text{Tr}_{\Ma/\La}\big(\, \psi^i_{\Ma,m}(\alpha)\,\big)\,\big]\, l(x)\ \big).
\end{align*}

From the condition on $m$ we have that $\pi^{m-n}T_{\La}\subset T_{\La,1}$. 
Thus, after taking the dual with respect to $\Tr$ we obtain
\[
\frac{1}{\pi^{m-n}}R_{\La}\supset R_{\La,1},\ \mathrm{or},\  \pi^nR_{\La}\supset \pi^mR_{\La,1}.
\]
Then from $\text{Tr}_{\Ma/\La}(\RM)\subset \RL$, cf. Proposition \ref{psi-properties} (1),  it follows that
\begin{equation}\label{trapa}
\text{Tr}_{\Ma/\La}(\pi^m\RM)\subset \pi^m\RL\subset \pi^nR_{\La}.
\end{equation}

It then follows that $\text{Tr}_{\Ma/\La}(\psi^i_{\Ma,m}(\alpha))\in \La/\pi^nR_{\La}$, 
and moreover since $(N_{\Ma/\La}(\alpha),x)^i_{\La,n}\in C/\pi^nC$ 
then $\text{Tr}_{\Ma/\La}(\psi^i_{\Ma,m}(\alpha))\in R_{\La}/\pi^nR_{\La}$ 
and the uniqueness follows from Lemma \ref{contphi}. 

\end{proof}
\begin{remark}\label{acotarazo} For $m$ in Proposition \ref{phi-psi} we can take  any
\[
m> n+\varrho\log_p\big(v_\La(p)/(p-1)\big)+\varrho/(p-1).
\]

The proof of this claim can be found in Proposition 6.2 of \cite{koly}. In Section \S\ref{Proofs of Chapter D} we reproduce the proof.

\end{remark}
\subsection{Artin-Hasse-type formulae for the Kummer pairing}\label{Invariants of the representation tau}

Having shown that the reduction of $D_{\Ma,m}^i$ is a 
multidimensional derivation we know then, after 
Proposition \ref{multder}, that this map is 
completely characterized by its value at the local uniformizers $T_1,\dots, T_{d-1}$, $\pi_\Ma$.  In this section, we will normalize this description using torsion points of the formal group $F$ instead. Concretely, we will show that there exist a torsion element $\mathfrak{e}_t\in \kappa_t$ (cf. Definition \ref{torsionados}) such 
that the derivation $D_{\Ma,m}^i$ can be described by its values at $T_1,\dots, T_{d-1}$, $\mathfrak{e}_t$. This description is done via an invariant associated to an Artin-Hasse type formula for the Kummer pairing of level $k$.

\subsubsection{The invariant $\overline{c}_{\beta,i}$}\label{Definition of the invariants}
Recall that we have fixed a basis $\{e^i\}_{i=1}^h$ 
for $\kappa=\varprojlim \kappa_n$. We also denoted by $e^i_n$ the 
reduction of $e^i$ to $\kappa_n$. 
Clearly $\{e_n^i\}$ is a basis for $\kappa_n$. 

Let $\Ma$, $\Ka_t$ and $\beta=(k,t)$ be as in \eqref{assumptions-2}. Let $\Ma_n=\Ma(\kappa_n)$.
 The action of $G_\Ma=Gal(\overline{\Ma}/\Ma)$ on $\kappa$ 
 defines a continuous representation
 \begin{equation}\label{Tau}
\tau: G_\Ma\to  GL_h(C).
\end{equation}
 
The reduction of $\tau$ 
to $GL_h(C/\pi^nC)$ is the analogous representation 
of $G_{\Ma}$ on $\kappa_n$ and will be denoted by $\tau_n$. 
This clearly induces an embedding 
$\tau_n:G(\Ma_n/\Ma)\to GL_h(C/\pi^nC)$.

If $a\in \Ma^*$, then $$\tau_{k+t}\big(\,\Upsilon_{\Ma}\big(\,\{\Td,a\}\,\big)\big)$$ 
is congruent to the identity matrix $I \pmod{\pi^t}$ because 
the Galois group $G(\Ma_{k+t}/\Ma)$ fixes $\kappa_t$ and so 
correspond in $GL_h(C/\pi^{k+t}C)$ to matrices $\equiv I \pmod{\pi^t}$, i.e., 
there exist characters $\chi_{\Ma,\beta:i,j}:\Ma^*\to C/\pi^k C$ such that
\[
\tau_{k+t}\big(\,\Upsilon_{\Ma}\big(\,\{\Td,a\}\,\big)\,\big)
=I+\pi^t\big(\,\chi _{\Ma,\beta:i,j}(a)\,\big) \in GL_h(C/\pi^{k+t}C).
\]
For $\Ma=\Ka_t$ we simply write $\chi _{\beta:i,j}$. 
By Proposition \ref{norm} (4) 
we have that $$N_{\Ma/\Ka_t}\{\Td,a\}=\{\,\Td,N_{\Ma/\Ka_t}(a)\,\},$$ 
where $\Ka_t=K_t\TT$, and Proposition \ref{pair} (4) implies
\begin{equation}\label{chi-norm}
\chi _{\Ma,\beta:i,j}(a)=\chi _{\beta:i,j}(N_{\Ma/\Ka_t}(a))
\end{equation}
The definition of the pairing $(,)_{\Ma,k}$ 
implies, for $v\in \kappa_t$, that
\[
\left(\ \big(\,\{\Td,a\},v\,\big)_{\Ma,k}^i\, \right)=\big(\,\chi_{\Ma,\beta:i,j}(a)\,\big)_{i,j}\ (v^j),
\]
as an identity of column vectors, where the right hand side is the product of the matrix $\big(\,\chi_{\Ma,\beta:i,j}(a)\,\big)_{i,j}$ with the column vector $(v^j)$ formed by the coordinates of $v$ with respect to to $\{e_t^i\}$. 
In particular, for $v=e^j_t$ we have
 \[
 \big(\,\{\Td,a\},e^{\,j}_t\,\big)_{\Ma,k}^i=\chi_{\Ma,\beta:i,j}(a).
 \]

According to Proposition \ref{rho} we see that $\chi_{\Ma:i,j}$  
 uniquely determines a constant 
 $c_{\Ma,\beta:i,j}\in R_{\Ma,1}/\pi^kR_{\Ma,1}$ such that
 \begin{equation}\label{invariantes-de}
 \chi_{\Ma,\beta:i,j}(u)=\mathbb{T}_{\Ma/S}\big(\,\log(u)\,c_{\Ma,\beta:i,j}\,\big)\quad \forall u\in V_{\Ma,1}.
 \end{equation}
 Namely, $\rho_{\Ma,k}^i(\Td,e_t^j)=c_{\Ma,\beta:i,j}$.

Observe that  $c_{\Ma,\beta:i,j}$ is, by Equation \eqref{chi-norm},
 the image of $c_{\beta,i,j}:=c_{\Ka_t,\beta:i,j}$ under the map 
\[
R_{\Ka_t,1}/\pi^kR_{\Ka_t,1}\to R_{\Ma,1}/\pi^kR_{\Ma,1}
\]
($R_{\Ka_t,1}\subset R_{\Ma,1}$). 
So we will denote $c_{\Ma,\beta:i,j}$ by $\overline{c}_{\beta:i,j}$.

\begin{defin}\label{torsionados}
For the torsion element $\mathfrak{e}_t$  
in Remark \ref{special-torsion-element} we will denote the constants 
$c_{\beta :i,j}$ and 
$\overline{c}_{\beta:i,j}$ by $c_{\beta:i}$ and $\overline{c}_{\beta:i}$, respectively, so that
 \begin{equation}\label{invariantes-de-special}
 \big(\,\{\Td,u\},\mathfrak{e}_t\,\big)_{\Ma,k}^i
    =\mathbb{T}_{\Ma/S}\big(\,\log(u)\,\overline{c}_{\beta:i}\,\big)\quad \forall u\in V_{\Ma,1}.
 \end{equation}
 The choice of this $\mathfrak{e}_t$ is independent of $i$.
  \end{defin}
  
  Note that Equation  \eqref{invariantes-de-special} defines  an Artin-Hasse-type formula for the Kummer pairing, in which the
 constant $\overline{c}_{\beta:i}$ is characterized by the value of the Kummer pairing  at the torsion point $\mathfrak{e}_t$. We will see in the next section how to construct the $D_{\Ma,m}^i$
 from of the value  $\overline{c}_{\beta:i}$.

Finally, observe that $c_{\beta:i,j}$ is an invariant of the isomorphism 
class of $(F,e_j)$ because if 
$g:(F,e_j)\to (\tilde{F},\tilde{e}_j)$ is such 
isomorphism then $\tilde{\rho}_{\Ma,k}^i(\Td,g(x))=\rho_{\Ma,k}^i(\Td,x)$.

From Proposition \ref{psi-rho} we conclude that
\begin{prop}\label{D(e)=c}
Let $\Ma$ and $\beta=(k,t)$ be as in \eqref{assumptions-2}. 
If $r(X)$ is a $t$-normalized series for $F$, then
\begin{equation}\label{Gan}
\begin{split}
   D^i_{\Ma,k}\big(\,T_1,\dots,T_{d-1},r(e_t^{\,j})\,\big)=
 -r'(e_t^{\,j})\ T_1\cdots T_{d-1}\ \frac{\overline{c}_{\beta:i,j}}{l'(e_t^j)}.
\end{split}
\end{equation}
In particular, this holds for the torsion element $\mathfrak{e}_t$ and the invariant $\overline{c}_{\beta:i}$.
\end{prop}

\subsubsection{Explicit description of $D_{\Ma,m}^i$ using the invariant $\overline{c}_{\beta:i}$}
Let $\La$ be a $d$-dimensional local field. Define
\begin{equation}\label{Rprima}
R'_{\La,1}:
=\big\{   x\in \La\, 
        :\, 
      v_\La(x)\geq -v_{\La}(D(\La/S))-\floor*{ v_{\La}(p)/(p-1)} \, \big\} .
\end{equation}
Note that 
\[v_\La(x)\geq -v_{\La}(D(\La/S))-\floor*{\frac{v_\La(p)}{p-1}}\quad \mathrm{iff}\quad
v_\La(x)\geq -v_\La(D(\La/S))-\frac{v_{\La}(p)}{p-1} .\]
This holds since  $v_\La(x)$ and $v_\La(D(\La/S))$ are integers, therefore the conditions
\[
\floor*{\frac{v_\La(p)}{p-1}}\geq -v_\La(D(\La/S))-v_\La(x)\quad
\mathrm{and}\quad
\frac{v_\La(p)}{p-1} \geq -v_\La(D(\La/S))-v_\La(x),
\]
are equivalent by the very definition of the integral part of a real number. Comparing with Equations \eqref{Rprima} and \eqref{RL1} we see that 
$
\RL'=\pi_\La\RL.
$ 

If $\Ma/\La$ is a finite extension of $d$-dimensional local fields, then clearly 
\begin{equation}\label{RM/RL}
\RM'=(1/D(\Ma/\La))\RL'.
\end{equation}

\begin{prop}\label{cij}
Let $\beta=(k,t)$, $K_t$ and $\Ka_t$ be as in  \eqref{assumptions-2}. 
Let $a_j\in \OZ_{K_t}$ such that 
$de^j_t=a_jd\pi_t$ in $\Omega_{\OZ_K}(\OZ_{K_t})$. Then
\[
c_{\beta:i,j}\in 
\frac{a_jR'_{\Ka_t,1}
+\frac{\pi^k}{\pi_t}R'_{\Ka_t,1}
+\pi^kR_{\Ka_t,1}}{\pi^kR_{\Ka_t,1}},
\]
\end{prop}
\begin{proof}

The proof is the same as of Proposition 6.5 of \cite{koly}. We will reproduce it again in
 Section \ref{Proofs of Chapter D} of the Appendix.
\end{proof}
\begin{lemma}\label{Carota}
Let $\Ma$ be as in \eqref{assumptions-2}. Let  $\mathfrak{e}_t$ and $\overline{c}_{\beta:i}$ be as in Definition \ref{torsionados}. Let $ \mathfrak{b}\in \OZ_\Ma$  such that 
$dT_1 \wedge \cdots \wedge d\mathfrak{e}_t=
\mathfrak{b}\,dT_1 \wedge \cdots \wedge d\pi_\Ma$ in 
$\hat{\Omega}_{\OK}^d(\OZ_{\Ma})$; this $\mathfrak{b}$ 
 exists by Proposition \ref{Khaler2}.
Then there exists a 
$\gamma_i\in \RM/(\pi^k/\pi_{\Ma})\RM$ such that 
\begin{equation}\label{condition}
-\overline{c}_{\beta:i}/l'(\mathfrak{e}_t)=\mathfrak{b}\,\gamma_i.
\end{equation}
 Moreover, 
all such $\gamma_i$'s coincide when reduced to 
\[\RM/(\pi^m/\pi_{\Ma})\RM,\]
for $m\leq k$ satisfying  $m/\varrho \leq k/\varrho-v(D(M/K))+t/\varrho-c_1$; $c_1$ is the constant from Proposition \ref{constants}. 
\end{lemma}
\begin{remark}
We can interpret the element $\gamma_i \pmod{(\pi^k/\pi_{\Ma})\RM}$ as follows. Choose a polynomial $g(X)$  with coefficients in the ring of integers of the maximal subextension of $M$ unramified over $K$, such that $\mathfrak{e}_t=g(\pi_{\Ma})$. Then
\[
\gamma_i=-\frac{\overline{c}_{\beta:i}}{l'(\mathfrak{e}_t)\,g'(\pi_{\Ma})}
\]
Moreover, modulo $(\pi^m/\pi_{\Ma})\RM$, this is independent of the choice of $g(X)$.
\end{remark}
\begin{proof}[Proof of Lemma \ref{Carota}]
Let $\lambda_{i}$ be a representative for 
$c_{\beta:i}$ in $R_{\Ka_t,1}$. 
By Proposition \ref{cij}, 
$\lambda_i\in aR'_{\Ka_t,1}+(\pi^k/\pi_t)R'_{\Ka_t,1}$  
where $d\mathfrak{e}_t=ad\pi_t$. Let $b\in \mathcal{O}_\Ma$ 
such that $dT_1\wedge \cdots \wedge d\pi_t=b\,dT_1\wedge \cdots \wedge d\pi_\Ma$, in particular $D(\Ma/\Ka_t)=b\,\OZ_\Ma$, 
and set $\mathfrak{b}=a\,b$. We have, by \ref{RM/RL}, that
\[
\frac{1}{b}\RKt'=(1/D(\Ma/\Ka_t))\RKt'=\RM'\subset \RM,
\]  
\[
\frac{\pi^k}{\pi_t}\RKt'\subset \frac{\pi^k}{\pi_{\Ma}}\frac{D(\Ma/\Ka_t)}{\pi_t/\pi_{\Ma}}\RM'\subset \frac{\pi^k}{\pi_{\Ma}}\RM,
\] 
where the last inclusion follows since $D(\Ma/\Ka_t)$ 
is divisible by $\pi_t/\pi_\Ma$; this follows from  
the general inequality $v_\Ma(D(\Ma/\Ka_t))\geq e(\Ma/\Ka_t)-1$ 
(cf. \cite{cassels} Chapter 1 Proposition 5.4). It thus follows
\[
-\frac{\overline{c}_{\beta:i}}{l'(\mathfrak{e}_t)}=
-\frac{\lambda_{i}}{l'(\mathfrak{e}_t)}\in \mathfrak{b}\frac{\RM}{(\pi^k/\pi_{\Ma})\RM}.
\]

Now let us prove the second assertion. Since 
$-\overline{c}_{\beta:i}/l'(\mathfrak{e}_t)=\mathfrak{b}\,\gamma_i$, then $\gamma_i$ is 
uniquely defined $\pmod{(\pi^k/\pi_{\Ma}\,\mathfrak{b})\RM}$.
Let $m\leq k$, thus $\gamma_i$ is uniquely defined $\pmod{(\pi^m/\pi_{\Ma})\RM}$  
as long as $\pi^m|(\pi^k/\mathfrak{b})$. But this condition is fulfilled 
when $m/\varrho\leq k/\varrho-v(D(M/K))+t/\varrho-c_1$. 
Indeed,  $dT_1 \wedge \cdots \wedge d\mathfrak{e}_t=\mathfrak{b}\,dT_1 \wedge \cdots \wedge d\pi_\Ma$ implies that $v(\mathfrak{p}_t)+v(\mathfrak{b})=v(\mathfrak{D}(\Ma/K))\,(\,\leq v(D(\Ma/K))\,)$, with $\mathfrak{p}_t$ as 
in Remark \ref{special-torsion-element}, therefore by the same remark
\begin{align*}
m/\varrho \,\leq\, k/\varrho-v(D(\Ma/K))+t/\varrho-c_1
  \,\leq\, k/\varrho-v(D(\Ma/K)) + v(\mathfrak{p}_t)
      \,=\, k/\varrho-v(\mathfrak{b}).
\end{align*}
\end{proof}

\begin{defin}\label{Explicit-D} Using the same notation and assumptions of Lemma \ref{Carota}, we 
construct the  explicit $d$-dimensional derivation 
$\mathfrak{D}_{\Ma,m}^i:\OMa^d\to R_{\Ma,1}/(\pi^m/\pi_{\Ma}) R_{\Ma,1}$, using Proposition \ref{multder}, in the following way
       \[
           \mathfrak{D}_{\Ma,m}^{i}(a_1,\dots ,a_d):=\det 
          \left[ \frac{\partial a_i}{\partial T_j} \right]_{_{1\leq i,j\leq d}}   
             \, \, \ T_1\cdots T_{d-1}\ \gamma_{i}, 
          \]
     where $a_1,\dots, a _d\in \OMa$. Moreover,
from  $\mathfrak{D}_{\Ma,m}^i$ we can construct
an explicit logarithmic derivative $\mathfrak{Dlog}_{\Ma,m}^i:K_d(\Ma)\to R_{\Ma,1}/(\pi^m/\pi_{\Ma}^2) R_{\Ma,1}$, by setting
\begin{equation*}\label{psi-from}
\left\{
\begin{aligned}
\mathfrak{Dlog}^i_{\Ma,m}(u_1,\dots, u_{d-1},\pi_{\Ma})&=\frac{\mathfrak{D}^i_{\Ma,m}(u_1,\dots, u_{d-1},\pi_{\Ma})}{u_1\cdots u_{d-1}\pi_{\Ma}} 
\pmod{\frac{\pi^m}{\pi_{\Ma}^2}R_{\Ma,1}}\\
\mathfrak{Dlog}^i_{\Ma,m}(u_1,\dots,u_d)&=\frac{\mathfrak{D}^i_{\Ma,m}(u_1,\dots,u_d)}{u_1\cdots u_d} \pmod{\frac{\pi^m}{\pi_{\Ma}}R_{\Ma,1}}\\
\mathfrak{Dlog}^i_{\Ma,m}(u_1,\dots, \pi_{\Ma}^k\,u_d)&=k\mathfrak{Dlog}^i_{\Ma,m}(u_1,\dots, \pi_{\Ma})+\mathfrak{Dlog}^i_{\Ma,m}(u_1,\dots,u_d),\ k\in \Z\\
\mathfrak{Dlog}^i_{\Ma,m}(a_1,\dots,a_d)&=0, \text{whenever $a_j=a_k$ for $j\neq k$, $a_1\dots, a_d\in \Ma^*$}
\end{aligned}
\right.
\end{equation*}
where $u_1,\dots,u_d$ are  in 
$\mathcal{O}_\Ma^*=\{u\in \mathcal{O}_\Ma:v_\Ma(u)=0\}$. 
\end{defin}

We now give conditions on when the map $D_{\Ma,m}^i$ ( respectively $\psi_{\Ma,m}^i$)
and the explicit $\mathfrak{D}_{\Ma,m}^i$ (respectively $\mathfrak{Dlog}^i_{\Ma,m}$) coincide.

\begin{prop}\label{Derivativa}
Let $\Ma$ be as in \eqref{assumptions-2}.  Let $m\leq k$ such that
$m/\varrho\leq k/\varrho+t/\varrho-v(D(M/K))-c_1$; $c_1$ is the constant from Proposition \ref{constants}.
Then the reduction of 
$$D^i_{\Ma,m}:\OMa^d \to R_{\Ma,1}/\pi^mR_{\Ma,1}$$ 
to $R_{\Ma,1}/(\pi^m/\pi_{\Ma})R_{\Ma,1}$ is a $d$-dimensional 
derivation over $\OK$. Moreover, this $d$-dimensional 
derivation coincides with the explicit $d$-dimensional derivation $\mathfrak{D}_{\Ma,m}^i$
from Definition \ref{Explicit-D}.
\end{prop}
\begin{remark}
In particular,  $$D^i_{\Ma,m}(\Td,e_t^j)=-\Tdc\overline{c}_{\beta:i,j}/l'(e^j_t)  \pmod{\pi^m/\pi_{\Ma}}R_{\Ma,1}$$ for all $1\leq j\leq h$.
\end{remark}
\begin{proof}
First notice that, by Proposition \ref{psi-properties} (3), 
$$D^i_{\Ma,m}:\OMa ^d\to R_{\Ma,1}/(\pi^m/\pi_{\Ma})R_{\Ma,1}$$ 
is the reduction of 
$$D^i_{\Ma,k}:\OMa^d \to R_{\Ma,1}/(\pi^k/\pi_{\Ma})R_{\Ma,1}$$ 
and the latter is a  $d$-dimensional derivation over $\OK$ 
according to Proposition \ref{Disaderiv}. Also,  according Proposition \ref{D(e)=c} we have
\begin{equation}\label{condition-prima}
D^i_{\Ma,k}\big(\,\Td,\mathfrak{e}_t\,\big)=-\Tdc\ \overline{c}_{\beta:i}/l'(\mathfrak{e}_t).
\end{equation}
Moreover, for $\mathfrak{b}$ as in Lemma \ref{Carota}, 
we have that $$D^i_{\Ma,k}\big(\,\Td,\mathfrak{e}_t\,\big)=
    \mathfrak{b}\,D^i_{\Ma,k}(\Td,\pi_{\Ma}).$$ 
    This  together with \eqref{condition-prima} imply 
    that $D^i_{\Ma,k}(\Td,\pi_{\Ma})/T_1\cdots T_{d-1}$ is 
    one of the $\gamma_i$'s that satisfies (\ref{condition}). 
    Therefore, by the second assertion of Lemma  \ref{Carota} 
    and by the way $\mathfrak{D}^i_{\Ma,m}$ was constructed 
    in Definition \ref{Explicit-D}, we conclude that the three 
    maps $D^i_{\Ma,k}$, $D^i_{\Ma,m}$ and $\mathfrak{D}^i_{\Ma,m}$ 
    coincide $\pmod{(\pi^m/\pi_{\Ma})R_{\Ma,1}}$. This concludes the proof.

The  statement  in the remark holds because $D_{\Ma,m}^i$ is a multidimensional derivation and from the formula
for $D_{\Ma,m}^i(T_1,\dots, T_{d-1},r(e_t^j))$ in equation (\ref{Gan}).
\end{proof}

\subsection{Main formulae}\label{Main formulas}
Let $\La\supset K_n$ be a $d$-dimensional local field. Let $\varrho$ be the ramification index of $S/\Qp$ and 
\begin{equation}\label{condition-m}
m=n+2+\floor*{\ \varrho\, \log_p\big(\,v_L(p)/(p-1)\,\big)+\varrho/(p-1)}. 
\end{equation}
Take $k$ an integer large enough such that
\begin{equation}\label{desigualdad}
t/\varrho+k/\varrho\geq m/\varrho+c_1+v(D(\Ma/K))\quad \text{and}
\end{equation}
\begin{equation}\label{petello}
 k+\varrho+1\geq c_1\varrho,
\end{equation}
where  $t=2k+\varrho+1$, $\Ma=\La_t$, and $c_1$ is the constant from Proposition \ref{constants}. Condition (\ref{desigualdad}) holds, for example, when 
\begin{equation}\label{examples}
\frac{k}{\varrho}\geq \frac{m}{\varrho}+\frac{\log_2(2k+\varrho+1)}{p-1}+c_1+c_2+v(D(\La/K)),
\end{equation}
where  $c_2$ is the constant from Proposition \ref{constants};
see the proof of Theorem \ref{main-theorem}  for a deduction of \eqref{desigualdad} from (\ref{examples}).
We now formulate the main result.

\begin{thm}\label{main-theorem}
Let $\La$, $\Ma$ and $m$ be as above.  Then
\begin{equation}\label{phi-2}
(N_{\Ma/\La}(\alpha),x)^i_{\La,n}
=\Tr\big(\normalfont\, \text{Tr}_{\Ma/\La}\big(\,\mathfrak{Dlog}^i_{\Ma,m}(\alpha)\,\big)\, l(x)\,\big)
=\mathbb{T}_{\Ma/S}\big(\, \mathfrak{Dlog}^i_{\Ma,m}(\alpha)\, l(x)\, \big), 
\end{equation}
for all $\alpha \in K_d(\Ma)$ and all $ x\in F(\mu_\La)$. Here $\mathfrak{Dlog}^i_{\Ma,m}$
is the explicit logarithmic derivative constructed in Definition \ref{Explicit-D}. 
\end{thm}
\begin{proof}
By considering the tower $\La_t\supset \La \supset K$ and 
the upper bound in Proposition \ref{constants} (1) and Remark \ref{Bound general L}, we get
\[
v(D(\Ma/K))= v\big(D(\La_t/\La)\big)+v\big(D(\La/K)\big)\leq t/\varrho+\log_2(t)/(p-1)+c_2+v\big(D(\La/K)\big).
\]
Adding $m/\varrho+c_1$ we obtain, by (\ref{examples}), the inequality 
(\ref{desigualdad}). The definition of $t$ clearly implies that $(k,t)$
is admissible and condition (\ref{petello}) 
implies $(t-k)/\varrho\geq c_1$. Thus $\Ma$, $t$, $k$ and $m$ defined as above,
satisfy the hypothesis of Proposition \ref{Derivativa} and of Definition (\ref{Explicit-D}). 
The result now follows from Proposition \ref{phi-psi}, the 
Remark \ref{acotarazo}  and Equation (\ref{psi-from-D}), i.e., the map $\psi^i_{\Ma,m}$ and $\mathfrak{Dlog}^i_{\Ma,m}$ coincide. 
It remains only to check that 
$\text{Tr}_{\Ma/\La}((\pi^m/\pi_\Ma^2)R_{\Ma,1})\subset \pi^nR_{\La}$, so that
$\text{Tr}_{\Ma/\La}(\mathfrak{Dlog}^i_{\Ma,m}(\alpha))$ is well defined in $R_{\La}/\pi^nR_{\La}$.
 To do this we notice 
that  condition (\ref{condition-m})  implies
\[
m-1> n+\varrho\log_p\big(\,v_\La(p)/(p-1)\,\big)+\varrho/(p-1),
\]
and we can apply  Remark \ref{acotarazo}  to $m-1$ and  get, by Equation (\ref{trapa}), that
\[
\text{Tr}_{\Ma/\La}\left(  \pi^{m-1}\frac{\pi}{\pi_\Ma^2} R_{\Ma,1} \right)\subset \text{Tr}_{\Ma/\La}(\pi^{m-1}R_{\Ma,1})\subset \pi^nR_{\La} .
\]
Bearing in mind that $\pi_\Ma^2|\pi$, since $\pi^k|D(\Ma/K)$ ( and $D(\Ma/K)=D(\Ma/\Ka)$) implies $e(\Ma/\Ka)>1$ (\ i.e., $\Ma/\Ka$ is not unramified).
\end{proof}



\section{Appendix}

\subsection{Higher-dimensional local fields}\label{higher fields}
The reader is welcome to visit the nice paper \cite{Zhukov} for an account on higher local fields
and all the results not proven in this section. In the 
 rest of this section $E$ will denote a local 
 field, $k_E$ its residue field and $\pi_E$ 
 a uniformizer for $E$.
 
         \begin{defin}\label{higher-local-field}
         $\Ka$ is an $d$-dimensional local field, i.e., 
                           a field for which there is a chain of fields $\Ka_{d}=\Ka$, 
                          $\Ka_{d-1}$, $\dots$, $\Ka_0$ such that $\Ka_{i+1}$ is a 
                           complete discrete valuation ring with residue field $\Ka_i$,
                              $0\leq i \leq d-1$, and $\Ka_0$ is a finite field of
                            characteristic $p$. 
          \end{defin}

If $k$ is a finite field then $\Ka=k((T_1))\dots ((T_d))$ 
is a $d$-dimensional local field with $$\Ka_i=k((T_1))\dots ((T_i)),\quad 1\leq i \leq d.$$
If $E$ is a local field, then 
$\Ka=E\{\{T_1\}\}\dots \{\{T_{d-1}\}\}$ is defined inductively as $E_{d-1}\{\{T_{d-1}\}\}$
  where $E_{d-1}=E\{\{T_1\}\}\dots \{\{T_{d-2}\}\}$. We have that $\Ka$ is a $d$-dimensional local field with 
  residue field $\Ka_{d-1}=k_{E_{d-1}}((T_{d-1}))$, and by induction $$k_{E_{d-1}}=k_E((T_1))\dots ((T_{d-2})).$$ 
Therefore $\Ka_{d-1}=k_E((T_1))\dots ((T_{d-1}))$. These fields are  called the $standard$ fields.

From now on we will assume $\Ka$ has 
$mixed$ characteristic, i.e., char($\Ka$)$=$0 
 and char($\Ka_{d-1}$)$=p$. The following 
 theorem classifies all such fields.

\begin{thm}[Classification Theorem]
Let $\Ka$ be an $d$-dimensional local field of mixed characteristic. 
Then $\Ka$ is a finite extension of a standard field 
$$E\{\{T_1\}\}\dots \{\{T_{d-1}\}\},$$ where $E$ is a
 local field, and there is a finite extension of $\Ka$
which is a standard field.
\end{thm}
\begin{proof}
    cf. \cite{Zhukov} $\S$ 1.1 \textit{Classification Theorem}.
    \end{proof}

\begin{defin}
An $d$-tuple of elements $t_1,\dots,t_d\in \Ka$ is called a 
system of local parameters of $\Ka$, if $t_d$ is a prime 
in $\Ka_d$, $t_{d-1}$ is a unit in $\mathcal{O}_{\Ka}$ but its
 residue in $\Ka_{d-1}$ is a prime element of $\Ka_{d-1}$, and so on.
\end{defin}

For the  standard field $E\{\{T_1\}\}\dots \{\{T_{d-1}\}\}$
 we can take as  a system of local parameters 
  $t_d=\pi_E$, $t_{d-1}=T_{d-1},\dots, t_1=T_1$.

\begin{defin}\label{higher-rank-valuation}
We define a discrete valuation of rank $d$ to be the 
map $\textsf{v}=(v_1,\dots,v_d): \Ka^*\to \mathbb{Z}^d$, 
$v_d=v_{\Ka_d}$, $v_{d-1}(x)=v_{\Ka_{d-1}}(x_{d-1})$ where 
$x_{d-1}$ is the residue in $\Ka_{d-1}$ of $x t_d^{-v_n(x)}$,
 and so on. 
\end{defin}

Although the valuation depends, for $n>1$, on the 
choice of $t_2,\dots, t_d$, it is independent in the
 class of equivalent valuations. 



  

\subsubsection{Topology on $\Ka$}\label{Topology on K}

We define the topology on $E\{\{T_1\}\}\dots \{\{T_{d-1}\}\}$ 
by induction on $d$.  
For $d=1$ we define the topology to be the topology of a 
one-dimensional local field.
Suppose we have defined the topology on a standard $d$-dimensional 
local field $E_d$ and let $\Ka=E_d\{\{T\}\}$. 
Denote by $P_{E_d}(c)$ the set $\{x\in E_d: v_{E_d}(x)\geq c\}$. 
Let $\{V_i\}_{i\in \mathbb{Z}}$  be a sequence of
neighborhoods of zero in $E_d$ such that 
\begin{equation}\label{vecindad}
\left\{
\begin{aligned}
&\text{1. there is a $c\in \mathbb{Z}$ such that $P_{E_d}(c)\subset V_i$ 
           for all $i\in \mathbb{Z}$.}\\
& \text{2. for every $l \in \mathbb{Z}$ we have $P_{E_d}(l)\subset V_i$ 
          for all sufficiently large $i$.}\\
\end{aligned}
\right.
\end{equation}
and put $\mathcal{V}_{\{V_i\}}=\{\sum b_iT^i\ :b_i\in V_i \}$.
These sets form a basis of neighborhoods of $0$ for a topology 
on $\Ka$. For an arbitrary $d$-dimensional local field $L$ of 
mixed characteristic we can find, by the Classification Theorem, 
a standard field that is a finite extension of $L$ and we can 
give $L$ the topology induced by the standard field.

\begin{prop}\label{topo}
      Let $\La$ be a $d$-dimensional local field of mixed 
      characteristic with the topology defined above.

      \begin{enumerate}

      \item $\La$ is complete with this topology. Addition is a  
            continuous operation and 
            multiplication by a fixed $a\in \La$ is a continuous map.
            
      \item Multiplication is a sequentially continuous map, i.e., if $x\in \La$ and $y_k\to y$ 
      in $\La$ then $xy_k\to xy$.

      \item This topology is independent of the choice of the 
            standard field above $\La$.

      \item If $\Ka$ is a standard field and $\La/\Ka$ is finite, then 
            the topology above coincides with the natural vector space 
            topology as a vector space over $\Ka$.

      \item The reduction map $\mathcal{O}_{\La}\to k_{\La}=\La_{d-1}$ is 
            continuous and open (where $\mathcal{O}_{\La}$ is given the 
            subspace topology from $\La$, and $k_{\La}=\La_{d-1}$ the 
            $(d-1)$-dimensional topology).

      \end{enumerate}
      
\end{prop}

\begin{proof}

      All the proofs can be found in \cite{Morrow} Theorem 4.10.

\end{proof}

\subsubsection{Topology on $\mathcal{K}^*$}\label{Topology on K*}\label{topology on K-star}
Let $\mathcal{R}\subset \Ka=\Ka_d$ be a set of representatives of 
the last residue field  $\Ka_0$. Let $t_1,\dots, t_d$ be a fixed system of local parameters for $\Ka$, i.e., $t_d$ is  a uniformizer for $\Ka$, $t_{d-1}$ is a unit in $\mathcal{O}_{\Ka}$ but its residue in $\Ka_{d-1}$ is a uniformizer element of $\Ka_{d-1}$, and so on. Then 
\[
\Ka^* = \mathcal{V}_{\Ka}\times \langle t_1 \rangle  \times \cdots \times \langle t_d \rangle \times  \mathcal{R}^*,
\]
where the group of  principal units $V_{\Ka}=1+M_{\Ka}$ and $\mathcal{R}^*=\mathcal{R}-\{0\}$. From this observation we have the following,
\begin{prop}\label{topo-prod}
We can endow $\Ka^*$ with the product of the induced topology from $\Ka$ on the group $\mathcal{V}_{\Ka}$ and the discrete topology on $\langle  t_1 \rangle  \times  \cdots \times  \langle t_d \rangle \times  \mathcal{R}^*$. In this topology we have,
\begin{enumerate}
\item Multiplication is sequentially continuous, i.e., if $a_n\to a$ and $b_n\to b$ then $a_n b_n\to ab$.
\item Every Cauchy sequence with respect to this topology converges in $\Ka^*$.
\end{enumerate}  
\end{prop}
\begin{proof}
See \cite{Zhukov} Chapter 1 \S1.4.2. 
\end{proof} 

\subsection{Formal groups}\label{Formal Groups}

 In this section we will state some useful results in the theory of formal groups.

\subsubsection{The Weiertrass lemma }\label{The Weiertrass lemma}

Let $E$ be a discrete valuation field of zero characteristic 
with integer ring $\mathcal{O}_E$ and maximal ideal $\mu_{E}$.
\begin{lemma}[Weierstrass lemma]\label{Weierstrass}
     Let $g=a_0+a_1X+\cdots \in \mathcal{O}_E[[X]]$ be such that
       $a_0,\dots, a_{n-1}\in \mu_E$, $n\geq 1$, and 
       $a_n \not\in \mu_E$. Then there exist a unique monic polynomial 
       $c_0+\cdots+ X^n$ with coefficients in $\mu_E$ and a 
       series $b_0+b_1X\cdots$ with coefficients in $\mathcal{O}_E$ and  
       $b_0$ a unit, i.e., $b_0\neq \mu_E$, such that 
      \[
      g=(c_0+\cdots+ X^n)(b_0+b_1X\cdots).
      \]
\end{lemma}

\begin{proof}
      See \cite{Lang} IV. \S 9 Theorem 9.2.
      \end{proof}

\subsubsection{The group $F(\mu_{\Ma})$}\label{The groups F(mu)}

Let $\Ka$ be a $d$-dimensional local field containing 
the local field $K$, say, 
$\Ka=K\{\{T_1\}\}\cdots\{\{T_{d-1}\}\}$. 
Denote by $F(\mu_{\Ka})$ the group with underlying set  
$\mu_{\Ka}$ and operation defined by the formal group $F$. 
More generally, if $\Ma$ is an algebraic extension of $\Ka$ we define
\[
F(\mu_{\Ma}):=\bigcup_{\substack{\Ma\supset\La \supset \Ka, \\ 
[\La/\Ka]<\infty}} F(\mu_{\La}).
\]
An element $f\in \text{End}(F)$ is said to be an 
isogeny if the map $f:F(\mu_{\bar{\Ka}})\to F(\mu_{\bar{\Ka}})$ 
induced by it is surjective with finite kernel.

If the reduction of $f$ in $k_K [[X]]$, $k_K$ the residue field of $K$, 
is not zero then it  is of the form $f_1(X^{p^h})$
 with $f_1'(0)\in \mathcal{O}_{K}^*$, cf. \cite{koly} Proposition 1.1. In this case we say that $f$ has finite height. 
If on the other hand the reduction of $f$ is zero we say it has infinite
height.

\begin{prop}

      $f$ is an isogeny if and only if  $f$ has finite height. 
      Moreover, in this situation $|\ker f|=p^h$.

\end{prop}

\begin{proof}

If the height is infinite, the coefficients of $f$ are 
 divisible by a uniformizer of the local field $K$, 
so $f$ cannot be surjective. Let $h<\infty$ and
 $x\in \mu_{\La}$ where $\La$ is a finite extension 
of $\Ka$. Consider the series $f-x$ and apply  
Lemma \ref{Weierstrass} with $E=\La$, i.e.,  

\[
f-x=(c_0+\cdots +X^{p^h})(b_0+b_1X+\cdots),
\]

where $c$\textquoteright s $\in \mu_{\La}$, 
$b$\textquoteright s $\in \mathcal{O}_{\La}$ and 
$b_0\in \mathcal{O}_{\La}^{*}$. Therefore the equation 
$f(X)=x$ is equivalent to the equation $c_0+\cdots +X^{p^h}=0$ 
and since the $c\text{\textquoteright s} \in \mu_{\La}$ 
every root belongs to $\mu_{\overline{{\Ka}}}$.

Moreover, the polynomial $P(X)=c_0+\cdots +X^{p^h}$ is separable 
because $f'(X)=\pi t(X)$, $t(X)=1+\dots$ is an invertible series 
and $f'=P'(b_0+b_1X+\cdots)+P(b_0+b_1X+\cdots)'$ so  $P'$ can 
not vanish at a zero of $P$. We conclude that $P$ has $p^h$ 
roots, i.e.,  $|\ker\ f|=p^h $.

\end{proof}

\begin{prop}\label{torsiones}
Denote by $j$ the degree of inertia of $S/\Qp$ and by $h_1$ the  height of $f=[\pi]_F$. Then $j$ divides $h_1$, namely $h_1=jh$. Let $\kappa_n$ be the kernel of $f^{(n)}$. Then
            \[
      \kappa_n\simeq (C/\pi^nC)^h \quad \mathrm{and} \quad
      \varprojlim \kappa_n\simeq C^h, 
      \]
      as $C$-modules. This $h$ is called the 
      height of the formal group with respect to $C=\mathcal{O}_S$.

\end{prop}

\begin{proof}
cf. \cite{koly} Proposition 2.3.
\end{proof}

\begin{remark} Notice that since the coefficients of $F$ 
are in the local field $K$ then 
$\kappa_n\subset \overline{K}$ for all $n\geq 1$. 
\end{remark}

\subsubsection{The logarithm of the formal group}\label{The logarithm of the formal group}

We define the logarithm of the formal group $F$ to be
the series

\[
l_F=\int_0^X\frac{dX}{F_X(0,X)}
\]

Observe that since $F_X(0,X)=1+\cdots \in \OK[[X]]^{*}$ 
then $l_F$ has the form

\[
X+\frac{a_2}{2}X^2+\cdots+\frac{a_n}{n}X^n+\cdots
\]
where $a_i\in \OK$.
\begin{prop}\label{logarithm}
        Let E be a field of characteristic 0 that is complete 
        with respect to a discrete valuation,  $\mathcal{O}_E$ 
        the valuation ring of $E$ with maximal ideal $\mu_E$ 
        and valuation $v_E$. Consider a formal group $F$ 
        over $\mathcal{O}_E$, then
        \begin{enumerate}
        
                \item The formal logarithm induces a homomorphism
                        \[
                        l_F: F(\mu_E)\to E
                        \]
                        with the additive group law on $E$.
                        
                \item The formal logarithm induces the isomorphism
                        \[
                        l_F: F(\mu_{E}^r)\ \tilde{\longrightarrow}\  \mu_{E}^r
                        \]
                        for all $r\geq [v_E(p)/(p-1)]+1$ and 
                        \[
                        v_E(l(x))=v(x)\quad (\ \forall x\in \mu_E^r\ ).
                        \]
                In particular, this holds for  
                $\mu_{E,1}=\{x\in E:\ v_E(x)>{v_E(p)/(p-1)}+1\ \}$. 
\end{enumerate}

\end{prop}

\begin{proof}
\cite{silver} IV Theorem 6.4 and Lemma 6.3.
\end{proof}

\begin{lemma}\label{muta}
Let $E$ and $v_E$ as in the previous proposition. Then
\[
v_E(n!)\leq \frac{(n-1)v_E(p)}{p-1},
\]
and $v_E(x^n/n!)\to \infty$ as $n\to \infty$ for $x\in \mu_{E,1}$.
\end{lemma}
\begin{proof}
The first assertion can be found in \cite{silver} IV. Lemma 6.2. 
For the second one notice that
\begin{multline*}
v_E(x^n/n!)\,\geq\, nv_E(x)-v(n!)
\geq\, nv_E(x)-(n-1)\frac{v_E(p)}{p-1}\\
=\,v_E(x)+(n-1)\left( v_E(x)- \frac{v_E(p)}{p-1}\right).
\end{multline*}
Since we are assuming that $x\in \mu_{E,1}$, i.e., $v_E(x)>v_E(p)/(p-1)$, then 
 $v_E(x^n/n!)\to \infty$ as $n\to \infty$.
\end{proof}

\begin{lemma}\label{seqcont}

     Let $\La$ be  a $d$-dimensional local field 
     containing the local field $K$,  
      $g(X)=a_1X+\frac{a_2}{2}X^2+\cdots+
     \frac{a_n}{n}X^n+\cdots$ and $h(X)=a_1X+\frac{a_2}{2!}X^2+\cdots+
     \frac{a_n}{n!}X^n+\cdots$ with $a_i\in \OK$. 
     Then $g$ and $h$ define, respectively, maps $g: \mu_{\La}\to \mu_{\La}$
     and $h: \mu_{\La,1}\to \mu_{\La,1}$ 
     that  are sequentially continuous in the Parshin 
     topology.
     
\end{lemma}

\begin{proof}

    We may assume $\La$ is a standard $d$-dimensional local field. 
    Let $\mathcal{V}_{\{V_i\}}$ be a basic neighborhood of 
    zero that we can consider to be a subgroup of $\La$, 
    and let $c>0$ such that 
    $P_{\La}(c)\subset \mathcal{V}_{\{V_i\}}$. If $x_n\in \mu_{\La}$ for all $n$,
    then there exists an $N_1>0$ 
    such that $v_{\La}(x_n^i/i), v_{\La}(x^i/i)>c$ 
    for all $i>N_1$ and all $n$; because 
    $iv_\La(x_n)-v(i)\geq iv_\La(x_n)-\log_p(i)\geq i-\log_p(i) \to \infty$ as $i\to \infty$.
    On the other hand, if $x_n\in \mu_{\La,1}$ for all $n$, then there exists an $N_2>0$ 
    such that $v_{\La}(x_n^i/i!), v_{\La}(x^i/i!)>c$ 
    for all $i>N_2$ and all $n$ by Lemma \ref{muta}. Then, for $N=\max\{N_1,N_2\}$, we have
         \[
    \sum_{i=N+1}^{\infty}a_i\frac{x_n^i-x^i}{i},\ 
    \sum_{i=N+1}^{\infty}a_i\frac{x_n^i-x^i}{i!}\in P_\La(c)\subset \mathcal{V}_{\{V_j\}}.
    \]
    
    Now, since multiplication is sequentially continuous 
    and $x_n\to x$ then 
    
    \[
    \sum_{i=1}^{N}a_i\frac{x_n^i-x^i}{i}\to 0,\ \sum_{i=1}^{N}a_i\frac{x_n^i-x^i}{i!}\to 0 \quad \text{as $n \to \infty$}
    \]
    
    Thus for $n$ large enough we have that 
    $$g(x_n)-g(x),\ h(x_n)-h(x)= \sum_{i=1}^{N}+\sum_{i=N+1}^{\infty}
     \in \mathcal{V}_{\{V_i\}}.$$
     
\end{proof}

\begin{remark} \label{sequito} In particular, $\log:\mu_\La\to \mu_{\La}$, $l_F:\mu_\La\to \mu_{\La}$ and 
$\exp_F=l^{-1}_F:\mu_{\La,1}\to \mu_{\La,1}$ are sequentially continuous.
\end{remark}


\subsection{Proofs of propositions and lemmas in section \ref{The Kummer Pairing}}\label{Proofs of Chapter A}

\begin{proof}[Proof of Proposition \ref{milnorelation}]

To simplify the notation we will assume $m=2$.
\begin{enumerate}
\item Noticing that $(1-a)/(1-1/a)=-a$ it follows that
\begin{align*}
      \{a,-a\}=\{a,1-a\}\{a,1-1/a\}^{-1}
           =\{a,1-1/a\}^{-1}
           =\{1/a,1-1/a\}
           =1
\end{align*}

\item This follows immedeately from the previous item \begin{align*}
      &\{a,b\}\{b,a\}=\{-b,b\}\{a,b\}\{b,a\}\{-a,a\}
           =\{-ab,b\}\{-ab,a\}
           =\{-ab,ab\}
           =1
\end{align*}

\end{enumerate}

\end{proof}

\begin{proof}[Proof of Theorem \ref{Kato} (2)]
 Let $\Ma$ be a finite abelian extension of $\La$, thus
 Gal$(\La^{ab}/\Ma)$ is an open neighborhood of $G_{\La}^{ab}$.
 Let $x_n$ be a convergent sequence to the zero element of $K_d(\La)$. 
 Since $N_{\Ma/\La}(K_d^{top}(\Ma))$ is a open subgroup of $K_d^{top}(\La)$ 
 by Proposition \ref{topo-norm}, then 
 \[\overline{x_n}\in N_{\Ma/\La}(K_d^{top}(\Ma))\quad (n>>0),
 \]
 where $\overline{x_n}$ is the image of $x_n$ in $K^{top}_d(\La)$. 
 Thus, there exist $y_n\in K_d(\Ma)$ and $\beta_n\in \Lambda_{m}(\La)$ such that
 \[x_n=\beta_nN_{\Ma/\La}(y_n)\quad (n>>0).
 \]
 From equation (\ref{Fesenko}) we have that $\beta_n\in \cap_{l\geq 1}lK_d(\La)$ 
 which implies that $\Upsilon_{\La}(\beta_n)$ is the identity element in $G_{\La}^{ab}$
 (because $G_{\La}^{ab}$ is a profinite group). Therefore
 \[
 \Upsilon_{\La}(x_n)=\Upsilon_{\La}(N_{\Ma/\La}(y_n))\quad (n>>0),
 \]
 but the element on the right hand side of this equality is the identity on Gal$(\La^{ab}/\Ma)$ 
 by the second item of this Theorem. It follows that $\Upsilon_{\La}(x_n)$ 
 converges to the identity element of $G_\La^{ab}$.
\end{proof}

\begin{proof}[Proof of Proposition \ref{pair}]
The first 5 properties follow from the definition of the pairing and Theorem \ref{Kato}.

          Let us prove property 7. Let $f^{(n)}(z)=y$ an take a 
          finite Galois extension $\mathcal{N}\supset \Ma(z)$ over $\La$. 
          Let $G=G(\mathcal{N}/\La)$ and $H=G(\mathcal{N}/\Ma)$, $w=[G:H]$, and
          $V:G/G'\to H/H'$ the transfer homomorphism. 
          Let $g=\Upsilon_\La(a)$, then by Theorem \ref{Kato} we have 
          $V(\Upsilon_\La(a))=\Upsilon_\Ma(a)$. The explicit computation of 
          $V$ at $g\in G$ proceeds as follows (cf. \cite{Scott} $\S$ 3.5 ). 
          Let $\{c_i\}$ be a set of representatives of 
          for the right cosets of $H$ in $G$, i.e., $G=\sqcup\ Hc_i$. Then for each 
          $c_i$, $i=1,\dots w$ there exist a $c_j$ such that 
          $c_igc_j^{-1}=h_i\in H$ and no two $c_j$'s are equal; 
          this is because $c_ig$ belongs to one and only one of the right cosets $Hc_j$. 
          Then $$V(g)=\prod_{i=1}^wh_i.$$
          
          Also, notice that  since $gc_j^{-1}=c_i^{-1}h_i$ then
          \[
          h_i(z)\ominus z=c_i^{-1}(h_i(z)\ominus z)=g(c_j^{-1}(z))\ominus c_i^{-1}(z).
          \]
          So we have
          \begin{align*}
          &(a,y)_{\Ma,n}= \Upsilon_\Ma(a)(z)\ominus z
                       =V(g)(z)\ominus z
                       =(\prod_{i=1}^wh_i)z\ominus z
                       =\oplus_{i=1}^w (h_i(z)\ominus z)\\
                       &=\oplus_{i=1}^w (g(c_j^{-1}(z))\ominus c_i^{-1}(z))
                       =g(\oplus_{i=1}^wc_j^{-1}(z))\ominus (\oplus_{i=1}^wc_j^{-1}(z))
                       =(a, N_{\Ma/\La}^F(y))_{\La,n}              
          \end{align*}
          the last equality being true since $g=\Upsilon_{\La}(a)$ and 
          $$f^{(n)}(\oplus_{i=1}^wc_j^{-1}(z))
          =\oplus_{i=1}^wc_j^{-1}(y)=N_{\Ma/\La}^F(y).$$

\end{proof}

\begin{proof}[Proof of Proposition \ref{normseries}]
First, the coefficients of $s$ are in $\mathcal{O}_K$ because 
$s^{\sigma}=s$ for every $\sigma \in G_K=
Gal(\overline{K}/K)$. 
Now, applying Lemma \ref{Weierstrass} to $s$ and $f^{(n)}$, 
we get $s=Ps_1$ and $f^{(n)}=Qf_1$, where $P$ and $Q$ is a 
monic polynomials and $s_1, f_1\in \OK[[X]]^{*}$. 
Since $s(F(X,v))=s(X)$ for all $v\in \kappa_n$, then $P(v)=0$ 
for all $v\in \kappa_n$ and  so $Q=\prod_{v\in \kappa_n}(X-v)$ 
divides $P$. This implies that $s$ is divisible by $f^{(n)}$, 
i.e., $s=f^{(n)}(a_0+a_1X+\cdots)$. In particular,
\[
s-f^{(n)}.a_0=f^{(n)}.(\ a_1X+\cdots).
\]
But from $s(F(X,v))=s(X)$ we see that $a_1X+\cdots $ 
must satisfy the same property and so $a_1v+\cdots =0$, 
for all $v\in \kappa_n$. Therefore this series 
is also divisible by $f^{(n)}$ and repeating the 
process we get $s=r_g(f^{(n)})$. Let us compute 
now $c(r_g)$. Taking the logarithmic derivative on s and then multiplying by $X$ we get
\[
\frac{s'(X)}{s(X)}X=\sum_{v\in \kappa_n} \frac{g'(F(X,v))F_X(X,v)X}{g(F(X,v))},
\]
which implies
\[
\frac{s'(0)}{\prod_{0\neq v \in \kappa_n}g(v)}=g'(0),
\]
From $s'=r'_g( f^{(n)} ) f^{(n)'} $ we obtain
\[
r_g'(0)=\frac{s'(0)}{f^{(n)'}(0)}=\frac{c(g)\prod_{0\neq v \in \kappa_n}g(v)}{\pi^n}.
\]
Each $g(v)$  is associated to $v$, $0\neq v\in \kappa_n$, 
then $\prod_{0\neq v \in \kappa_n}g(v)$ is associated to 
$\prod_{0\neq v \in \kappa_n}v$, but the latter is associated 
to $\pi^n$ from the equation $f=Pf_1$. Then $c(r_g)\in \OK^{*}$. 
Finally, we will show that 
$(\{a_1,\dots,a_{i-1}, r_g(x),a_{i+1},\dots,a_d \}, x)=0$. 
Let $L$ be a local field containing $\kappa_n$, $\La=L\TT$, $x\in F(\mu_\La)$ and $z$ such 
that $f^{(n)}(z)=x$. Then
\[
r_g(x)=\prod_{v\in \kappa_n} g(z\oplus_Fv)=\prod_i N_{\La(z)/\La}(g(z_i)),
\]
where the $z_i$ are pairwise non-conjugate over $\La$ distinct roots 
of $f^{(n)}(X)=x$, so
\begin{align*}
\{a_1,\dots,a_{i-1}, r_g(x),a_{i+1},\dots,a_d \}&=\{a_1,\dots,a_{i-1},N_{\La(z)/\La}(\prod_i g(z_i) ),a_{i+1},\dots,a_d \} \\
&=N_{\La(z)/\La}(\{a_1,\dots,a_{i-1}, \prod_i g(z_i) ,a_{i+1},\dots,a_d \}),
\end{align*}
The last equality follows from  Proposition 
\ref{norm} (1) and (4). The result now follows 
from Proposition \ref{pair}.
\end{proof}

\subsection{Proofs of Section \ref{The maps psi and rho}}\label{Proofs of Chapter B}

\begin{proof}[Proof of Proposition \ref{Riezs}]
Assume first that $\La$ is the standard higher local field $L\TT$. The proof
is done by induction in $d$. If $d=1$ the result is known. 
Suppose the result is true for $d\geq 1$ and let $\La=E\{\{T_d\}\}$
where $E=L\TT$.

 Let $\phi : \La\to S$ be a sequentially continuous 
 $C$-linear map and define, for each $i\in \Z$, 
 the sequentially continuous map $\phi_i(x)=\phi(xT_d^i)$ 
 for all $x\in E$. Then clearly 
 $\phi_i\in \text{Hom}_C(E,S)$ and by the induction hypothesis
  we know that there exists  
 an $a_{-i}\in E$ such that 
 $\phi(xT_d^i)=\mathbb{T}_{E/S}(a_{-i}x)$ for all 
 $x\in E$. Let $\alpha=\sum a_iT_d^i$, we must show that

 \begin{enumerate}[I.]
     \item $\min \{v_E(a_{i})\}>-\infty$.
     \item $v_{E}(a_{-i})\to \infty$ as $i \to \infty$ 
     (i.e., conditions (I) and (II) imply that $\alpha\in \La$).
     \item $\phi (x)=\Tr(\alpha x)$, $\forall x\in \La$.
 \end{enumerate}

For any $x=\sum x_iT^i \in \La$ we have, 
by the sequential continuity of $\phi$ that
\begin{equation}\label{convergent}
\phi(x)=\sum_{i\in \Z}\phi(x_iT_d^i)=\sum_{i\in \Z}\mathbb{T}_{E/S}(a_{-i}x_i).
\end{equation}

Suppose (I) was not true, then there exist 
a subsequence $\{a_{n_k}\}$ such that 
$v_E(a_{n_k})\to -\infty$ as $n_k\to \infty$ 
or as $n_k\to -\infty$. In the first case we take 
an $x=\sum x_iT_d^i \in \La$ such that $x_i$ is 
equal to $1/a_{n_k}$   if $i=- n_k$ 
and 0 if $i\neq -n_k$. So $a_{-i}x_i$=1 if $i=- n_k$ 
and 0 if $i\neq- n_k$. Then the sum on the right of 
(\ref{convergent}) would not converge. In the second 
case we  take   $x_i$ to be  equal to $1/a_{n_k}$ if 
$i=- n_k$ and 0 if $i\neq -n_k$. So 
$a_{-i}x_i=1$ if $i=- n_k$ and 0 if 
$i\neq - n_k$ and again the sum on the right would not converge.

Suppose (II) was not true. Then  
$v_L(a_{n_k})<M$ for some positive integer $M$ and 
a  of negative integers  $n_k\to -\infty$. Then take $x=\sum x_iT_d^i \in \La$ such 
that $x_i$ is equal to $1/a_{n_k}$ for $i=-n_k$ and 0 for $i \neq -n_k$. 
So $a_{-i}x_i=1$ if $i=-n_k$ and 0 if $i\neq -n_k$  and the 
sum on the right of (\ref{convergent}) would not converge.

Finally, (III) follows by noticing that by (I) and (II) 
the sum  $\sum_{i\in \Z} a_{-i}x_i$  converges and
\[
\sum_{i\in \Z}\text{Tr}_{L/S}(a_{-i}x_i)=\text{Tr}_{L/S}(\sum_{i\in \Z}a_{-i}x_i)=\Tr(x\alpha).
\]

Assume now that $\La$ is an arbitrary $d$-dimensional local field. Then let $\La_{(0)}$
be the standard local field from Section \ref{Terminology and Notation}. Since $\La/\La_{(0)}$
is a finite extension, then $\text{Tr}_{\La/\La_0}$ induces a pairing $\La\times \La\to _{\La_{(0)}}$. From this and the first part of the proof the result now follows. 
\end{proof}

\begin{proof}[Proof of equation (\ref{RL1})]
By induction on $d$. For $d=1$ this is proven in \cite{koly} \S 4.1. 
Suppose it is true for $d\geq 1$, and let $\La=E_d\{\{T_d\}\}$, where $E_d=L\TT$.
If $x=\sum_{i\in \Z}x_iT_d^i\in R_{\La,1}$, 
then since $\mu_{E_d,1}\subset \mu_{\La,1}$ we have
that also $\mu_{E_d,1}T_d^{-i}\subset \mu_{\La,1}$ and
\[
\Tr(x\ \mu_{E_d,1}T_d^{-i})\subset C,
\]
which implies $\mathbb{T}_{E_d/S}(x_i\mu_{E_d,1})\subset C$, 
since $\Tr=\mathbb{T}_{E_d/S}\circ c_{\La/E_d}$.
By induction hypothesis we have 
$v_{E_d}(x_i)\geq\  -v_L(D(L/S))-\floor*{v_L(p)/(p-1)}-1$
for all $i\in \Z$, therefore
\[
v_\La(x)=\min v_{E_d}(x_i)\geq -v_L(D(L/S))-\floor*{v_L(p)/(p-1)}-1.
\]
Conversely, if $v_\La(x)
=\min v_{E_d}(x_i)
\geq 
-v_L(D(L/S))-\floor*{v_L(p)/(p-1)}-1$,
then $v_{E_d}(x_i)
\geq\  -v_L(D(L/S))-\floor*{v_L(p)/(p-1)}-1$
for all $i\in \Z$. Then, by the induction hypothesis 
$\mathbb{T}_{E_d/S}(x_i\mu_{E_d,1})\subset C$ for all $i\in Z$, and therefore
\[
\Tr(x\mu_{\La,1})=
\sum_{i\in \Z}
\Tr(x_iT_d^i\ \mu_{\La,1})=
\sum_{i\in \Z} \mathbb{T}_{E_d/S}(x_i\ \mu_{E_d,1})
\subset C.
\]
Thus, identity (\ref{RL1}) holds for standard local fields $L\TT$. In the general case of an arbitrary $d$-dimensional local field  $\La$, it is enough to consider the finite extension $\La_{(0)}$ from Section \ref{Terminology and Notation}.
\end{proof}

\begin{proof}[Proof of Proposition \ref{psi-properties}]
\begin{enumerate}
    \item From the identity $\mathbb{T}_{\Ma/S}= \Tr\circ \text{Tr}_{\Ma/\La}$ and the fact that $\mu_{\La,1}\subset \mu_{\Ma,1}$ we obtain 
          \[\mathbb{T}_{\La/S}(\text{Tr}_{\Ma/\La}(R_{\Ma,1})\mu_{\La,1})
          =\mathbb{T}_{\La/S}(\text{Tr}_{\Ma/\La}(R_{\Ma,1} \mu_{\La,1}))
          \subset \mathbb{T}_{\Ma/S}(R_{\Ma,1} \mu_{\Ma,1})
          \subset C,\]
          from which follows that 
          $ \text{Tr}_{\Ma/\La}(R_{\Ma,1})\subset R_{\La,1}$. 
          Now by Proposition \ref{pair} (4)
          we have, for $b\in K_d(\Ma)$ and 
          $x\in F(\mu_{\La,1})$, that
      \[
                (N_{\Ma/\La}(b),x)_{\La,n}=(b,x)_{\Ma,n}=\mathbb{T}_{\Ma/S}(\psi_{\Ma,m}^i(b)l(x))=\mathbb{T}_{\La/S}(\text{Tr}_{\Ma/\La}(\psi_{\Ma,m}^i(b))l(x)).
             \]
     It follows from Proposition \ref{psi} that 
     \[\psi_{\La,m}^i(N_{\Ma/\La}(b))=\text{Tr}_{\Ma/\La}(\psi_{\Ma,m}^i(b)).
     \]
      \item This is proved in a similar fashion to the previous 
      property but this time using Proposition \ref{pair} (7).
      
      \item This follows from Proposition \ref{psi} and Proposition \ref{pair} (5). Indeed, since
            $e_n^i=f^{(m-n)}(e_m^i)$  and $(a,x)_n=(a,f^{(m-n)}(x))_m$ we get 
            \[
            \pi^{m-n}(a,x)_n^i=(a,f^{(m-n)}(x))_m^i \pmod{\pi^m C}
            \]
            That is
            \begin{align*}                      
            \pi^{m-n}\ \Tr(\psi_{\La,n}^i(a) \ l_F(x))
            =&\Tr(\psi_{\La,m}^i(a) \ l_F(f^{(m-n)}(x)))\\
            =&\pi^{m-n}\ \Tr(\psi_{\La,m}^i(a) \ l_F(x)) \pmod{\pi^m C}
            \end{align*}
            Upon dividing by $\pi^{m-n}$ the result follows.
      \item This property follows from Proposition \ref{pair} (8) and $l_{\tilde{F}}(t)=t'(0)l_F$.
      
\end{enumerate}

\end{proof}

\begin{proof}[Proof of Lemma \ref{Kolyv}]
For a $d$-dimensional local field $\La$, let $\La_{(0)}$ be as in Section \ref{Terminology and Notation}. Since $\La$ and $\La_{(0)}$ have the same residue field, it is enough to prove the result for a standard higher local field. Thus we assume $\La$ is standard and proceed by induction on $d$. For $d=1$, the result follows since $k_L$ is perfect.
Suppose it is proved for $R=\mathcal{O}_{\La_{d}}$, where 
$\La_{d}=L\{\{T_1\}\}\dots\{\{T_{d-1}\}\}$. Let $\La=\La_{d}\{\{T_d\}\}$ and
 $x\in \OLa$. Then $x\equiv\sum_{j\geq m}a_jT_{d}^{j}\pmod{\pi_L}$, $a_j\in R$. Thus, by the induction hypotheses
\begin{equation*}
\Scale[0.8]{
\begin{split}
x&\equiv \sum_{0\leq i_d<p^n} 
         T_d^{i_d}  \left(  \sum_{m\leq i_d+kp^n}a_{i_d+kp^n}T_d^{kp^n} \right) \pmod{\pi_L}\\
&\equiv \sum_{0\leq i_d<p^n} T_d^{i_d}
            \left(\sum_{m\leq i_d+kp^n}\left( \sum_{0\leq i_1,\dots, i_{d-1}< p^n}\gamma_{i_1,\dots, i_{d-1}, i_d;k}^{p^n}\ T_1^{i_1}\cdots T_{d-1}^{i_{d-1}} \right)T_d^{kp^n} \right) \pmod{\pi_L} \\
&\equiv \sum_{0\leq i_d<p^n} T_1^{i_1}\cdots T_{d-1}^{i_{d-1}}T_d^{i_d}\left(\sum_{k}\left( \sum_{0\leq i_1,\dots, i_{d-1}< p^n}\gamma_{i_1,\dots, i_{d-1}, i_d;k}^{p^n}\ T_d^{kp^n}\right) \right) \pmod{\pi_L}\\
&\equiv \sum_{0\leq i_1,\dots, i_d< p^n}\gamma_{i_1,\dots, i_{d}}^{p^n}\ T_1^{i_1}\cdots T_d^{i_d} \pmod{\pi_L}
\end{split}}
\end{equation*}
where
$
\gamma_{i_1,\dots, i_{d}}=\sum_{k}
\gamma_{i_1,\dots, i_{d-1}, i_d;k}\ 
T_d^{k}
$
and  regrouping terms is valid since the series are absolutely convergent in the Parshin topology. 
Also by  noticing that the congruence
\[
\Scale[0.9]{
\sum_{c\leq k}b_k^{p^n}T^{kp^n}\equiv 
\left( \sum_{c\leq k}b_kT_d^k\right) ^{p^n} \pmod{\pi_L}},
\]
holds in $k_{\La_{d-1}}((T_d))$, where $k_{\La_{d-1}}$ is the residue field of $\La_{d-1}$.

\end{proof}

\begin{proof}[Proof of Corollary \ref{dermax}]
This follows  from Proposition \ref{derivat} and Proposition \ref{Dar} (3).
Indeed, let us illustrate the proof in the case $d=2$, i.e., $\La$ is a 2-dimensional local field with  a systems of local uniformizers $T_1$ and $T_2=\pi_\La$. 
To simplify the notation 
we will denote $D_{\La,n}^i$ by $D$. From Proposition \ref{derivat} we have
\begin{equation}
\begin{split}
D(\eta_1(T_1,T_2),\eta_2(T_1,T_2))=\frac{\partial \eta_1}{\partial X_1}\frac{\partial \eta_2}{\partial X_1}\bigg|_{\substack{X_i=T_i,\\ i=1,2}}D(T_1,T_1)
+\frac{\partial \eta_1}{\partial X_1}\frac{\partial \eta_2}{\partial X_2}\bigg|_{\substack{X_i=T_i,\\ i=1,2}}D(T_1,T_2)\\
+\frac{\partial \eta_1}{\partial X_2}\frac{\partial \eta_2}{\partial X_1}\bigg|_{\substack{X_i=T_i,\\ i=1,2}}D(T_2,T_1)
+\frac{\partial \eta_1}{\partial X_2}\frac{\partial \eta_2}{\partial X_2}\bigg|_{\substack{X_i=T_i,\\ i=1,2}}D(T_2,T_2)
\end{split}
\end{equation}
But $D(T_1,T_1)=D(T_2,T_2)=0$, $D(T_2,T_1)=-D(T_1,T_2)$ from Proposition \ref{Dar} (3), therefore
\begin{equation}
\begin{split}
D(\eta_1(T_1,T_2),\eta_2(T_1,T_2))=
\left( \frac{\partial \eta_1}{\partial X_1}\frac{\partial \eta_2}{\partial X_2}
-\frac{\partial \eta_1}{\partial X_2}\frac{\partial \eta_2}{\partial X_1}\right)\bigg|_{\substack{X_i=T_i,\\ i=1,2}} D(T_2,T_1)
).
\end{split}
\end{equation}
The corollary follows.
\end{proof}

\subsection{Proofs of Section \ref{Multidimensional derivations}}\label{Proofs of Chapter C}
\begin{proof}[Proof of Proposition \ref{Khaler}]
To simplify the notation let us denote 
 $ \hat{\Omega}_{\OK}(\OLa)$ by $\hat{\Omega}$, where $\Omega=\Omega_{\OK}(\OLa)$. 
We will start by showing  that
 $\Omega/\pi_L^n\Omega$ is generated by $d\pi_\La$ and $dT_1,\dots, dT_{d-1}$ for all $n$.
 
Let $x\in \OLa$, then by corollary \ref{sumativa}, we have that
\[
x=\sum _{k=0}^{\infty}\left( \sum_{0\leq i_1,\dots, i_{d-1} <p^n}\gamma_{i,k}^{p^n}\ T_1^{i_1}\cdots T_{d-1}^{i_{d-1}}\pi_\La^k\right) 
.
\]
Therefore, in $\Omega/p^n\Omega$, we can consider the truncated sum
\[
\sum _{k=0}^{m}\left( \sum_{0\leq i_1,\dots, i_{d-1} <p^n}\gamma_{i,k}^{p^n}\ T_1^{i_1}\cdots T_{d-1}^{i_{d-1}}\pi_\La^k\right)  
,
\]
 where $m$ is such that $p^n|\pi_{\La}^{m+1}$. Thus, $dx$ is 
 generated by $d\pi_L$ and $d\,T_i$, $i=1,\dots, d-1$ in $\Omega/p^n\Omega$.

We will assume the notation of Section \ref{Canonical Derivations} and let $T_d=\pi_{\La}$. Let $b_i=\frac{\partial p}{\partial T_i}(\pi_{\La})$, $i=1,\dots, d$, and let $\mathfrak{D}$ be the ideal of $\OLa$ such that
\begin{equation}\label{annihilator-ideal}
v_{\La}(\mathfrak{D})=\min_{1\leq i \leq d}\{v_{\La}(b_i)\}.
\end{equation}
Without loss of generality we may assume that $v_{\La}(b_d)=\min_{1\leq i \leq d}\{v_{\La}(b_i)\}$, and then we define 
\[
w=\frac{b_1}{b_d}dT_1+\dots \frac{b_{d-1}}{b_d}d\,T_{d-1}+d\,T_{d} \in \Omega.
\] 
It is clear that $\mathfrak{D}w=0$ and also that $dT_1\,\dots, dT_{d-1}$, $w$  generate $\Omega/p^n\Omega$ for all $n$.

 We will 
 show now  that
 \begin{equation}\label{proje}
\frac{\Omega}{\pi_L^n\Omega} \simeq  \frac{\OLa}{\pi^n\OLa}  \oplus \cdots 
\oplus \frac{\OLa}{\pi^n\OLa} 
\oplus \frac{\OLa}{\pi^n\OLa+\mathfrak{D}\OLa} 
  \end{equation}
 for all $n\geq 1$. These isomorphisms are compatible: $\simeq_{n+1}\equiv \simeq_n\pmod{p^n}$,
 then we can take
  the projective limit $\varprojlim$ to  obtain the result. This will imply in particular that $\mathfrak{D}$ is the annihilator ideal of the torsion part of  $ \hat{\Omega}_{\OK}(\OLa)$, i.e., $\mathfrak{D}(\La/K)=\mathfrak{D}$.

In order to construct the isomorphism (\ref{proje}) we consider the derivations 
 $D_k:\O_{\La}\to \OLa$ for $k=1,\dots ,d-1$ and $D_d:\OLa\to \OLa/\mathfrak{D}$ as follows
 \[
D_i(g(\pi_{\La}))=\frac{\partial g}{\partial T_i}(\pi_{\La})-\frac{b_i}{b_d}\frac{\partial g}{\partial T_d}(\pi_\La) \quad i=1,\dots, d-1
 \]
 and 
 \[
D_d(g(\pi_\La))=\frac{\partial g}{\partial T_d}(\pi_{\La})
 \]
 for $g(x)\in \mathcal{O}_{\La_{(0)}}[X]$. It is clear from the very definition that these are well-defined derivations which are independent of the choice of $g(x)$.
 Define the map 
 \[\partial:\OLa \ \longrightarrow \
  \frac{\OLa}{p^n\OLa}
\oplus  \frac{\OLa}{p^n\OLa} \oplus \cdots  
\oplus  \frac{\OLa}{p^n\OLa+\mathfrak{D}\OLa}
 \]
 by 
 \[
 a\to (\overline{D_1}(a), \dots, \overline{D_{d}}(a))
 \]
 where $\overline{D_k}$ is the reduction of $D_k$. This is a well-defined derivation of $\OLa$ 
 over $\OK$ and 
 by the universality of $\Omega$,
  this induces a homomorphism of $\OLa$-modules
 \[
 \partial: \frac{\Omega}{p^n\Omega}\ \longrightarrow \
   \frac{\OLa}{p^n\OLa}
\oplus   \cdots  \oplus
 \frac{\OLa}{p^n\OLa+ \mathfrak{D}\OLa}. 
 \]
 Let us show that $\partial$ is an isomorphism. Indeed, it is clearly 
 surjective since for 
 $(a_0,\dots, a_{d-1})\in 
 \OLa \oplus \cdots 
 \oplus \OLa
 \oplus (\OLa/D(L/K)\OLa )$ we have that 
 \[
 \partial(a_1dT_1+\cdots+ a_{d-1}dT_{d-1}+a_dw)=(a_1,\dots,a_{d}).
 \]
 since 
 \[
 \overline{D_k}(w)=
 \begin{cases}
 1, & k=d,\\
 0, & 1\leq k\leq d-1, 
 \end{cases}
 \hspace{25pt}
 \overline{D_k}(dT_i)=
 \begin{cases}
 0, & k\neq i,\\
 1,  & k=i\neq d,
 \end{cases}
 \]  
 Also, $\partial$ is injective for if 
 $a=a_1dT_1+\cdots a_{d-1}dT_{d-1}+a_dw \in \Omega/p^n\Omega$ 
 is  such that $\partial(a)=0$, then $\overline{D_k}(a)=a_k=0$ in $\OLa/p^n\OLa$, 
 for $1\leq k\leq d-1$, and 
 $\overline{D_d}(a)=a_d=0 
 \pmod{p^n\OLa+\mathfrak{D}\OLa}$. 
 But then $a_dw=0$, since 
 $\mathfrak{D}w=0$, and 
 therefore $a=0 \mod{p^n\Omega}$. This concludes the proof. 
 
 Notice that if $\La$ is the standard higher local field $L\TT$, then we take as  a system of uniformizers $T_1,\dots, T_{d-1}$ and $\pi_L$, and in this case $b_i=0$ for $i=1,\dots, d-1$, form which the second claim in the statement of the proposition follows.
 
\end{proof}

\subsection{Proofs of Section \ref{Deduction of the formulas}}\label{Proofs of Chapter D}

\begin{proof}[Proof of Remark \ref{acotarazo}]
Let $k=m-n$. Then $f^{(k)\prime}$ is divisible by $\pi^k$, 
which implies that every term $a_iX^i$ of the series $f^{(k)}$
satisfies $v(a_i)+v(i)\geq k/\varrho$. If $v(a_i)>1/(p-1)$ then
there is nothing to prove. If on the other hand $v(a_i)\leq 1/(p-1)$
then $v(i)\geq k/\varrho-1/(p-1)$. In this case
\[
v(x^i)
\geq \frac{i}{v_\La(p)}
\geq \frac{p^{\frac{k}{\varrho}-\frac{1}{p-1}}}{v_\La(p)}
>\frac{1}{p-1}
\]
for all $x\in \mu_\La$, since $k/\varrho-1/(p-1)>\log_p(v_\La(p)/(p-1))$. Then
\[
v(f^{(k)}(x))>\frac{1}{p-1}
\]
for all $x\in \mu_\La$.
\end{proof}
\begin{proof}[Proof of Proposition \ref{cij}]
We begin by taking a representative $\lambda_{i,j}$ of $c_{\beta:i,j}$ in $R_{\Ka_t,1}$. 
We have to show that 
\begin{equation}\label{pls}
\lambda_{i,j}\in a_jR'_{\Ka_t,1}+\left( \frac{\pi^k}{\pi_t} \right) R'_{\Ka_t,1}.
\end{equation}
Let $M\supset K_t$, $\pi_M$ and $\pi_t$ uniformizers 
for $M$ and $K_t$, respectively, and $\Ma=M\TT$, $\Ka_t=\Ka_t\TT$. 
Let $b\in \OZ_M$ such that $d\pi_t=bd\pi_M$; 
this exist by Proposition \ref{module-of}. 
Then $D(M/K_t)=b\OZ_M$. Set $\beta_j=ba_j\in \OZ_M$. 
Clearly, $de_t^j=\beta_jd\pi_M$. 
By Proposition \ref{Disaderiv}, 
$$D_{\Ma,k}^i:\OZ_\Ma^d\to R_{\Ma,1}/(\pi^k/\pi_M)R_{\Ma,1}$$
is a $d$-dimensional derivation over $\OZ_K$, 
which together with Proposition \ref{D(e)=c} implies
\begin{align*}
r'(e_t^j)\ \beta_j\ D_{\Ma,k}^i(\Td,\pi_M)
&=D^i_{\Ma,k}(\Td,r(e_t^j))\\
&=-r'(e_t^j)\ T_1\cdots T_{d-1}\ \frac{\overline{c}_{\beta:i,j}}{l'(e_t^j)} 
\pmod{(\pi^k/\pi_M)R_{\Ma,1}}.
\end{align*}
Recall that $\overline{c}_{\beta:i,j}$ is the image of $c_{\beta:i,j}$  
under the map 
$R_{\Ka_t,1}/\pi^kR_{\Ka_t,1}\to R_{\Ma,1}/\pi^kR_{\Ma,1}$;
$R_{\Ka_t,1}\subset R_{\Ma,1}$. This identity implies
\begin{equation}
\lambda_{i,j}\in \beta_jR_{\Ma,1}+\left( \frac{\pi^k}{\pi_M} \right) R_{\Ma,1}.
\end{equation}
Then
\begin{align*}
v_\Ma(\lambda_{i,j})&\geq \min \{\ v_\Ma(\beta_jR_{\Ma,1})\ ,
\ v_\Ma\left(\ (\pi^k/\pi_M) R_{\Ma,1}\ \right)\ \}\\
&\geq \min \left\lbrace  v_\Ma(\beta_j)-v_{\Ma}(D(M/S))-\frac{e(M)}{p-1}-1\ , \right. \\
& \hspace{100pt} \left. v_\Ma(\frac{\pi^k}{\pi_M})-v_{\Ma}(D(M/S))-\frac{e(M)}{p-1}-1\   \right\rbrace
\end{align*}

We will further assume that $M$ is the local field 
obtained by adjoining to $K_t$ the
roots of the Eisenstein polynomial $X^n-\pi_t$, $(n,p)=1$. 
Then $e(M/\Qp)=n\,e(K_t/\Qp)$ and $D(M/K_t)=n\pi_M^{n-1}=\pi_t/\pi_M$. 
\begin{align*}
v_\Ma(\lambda_{i,j})\geq \min \{\ &v_\Ma(a_j)-v_{\Ma}(D(K_t/S))-\frac{e(M)}{p-1}-1\ ,\\
                       & \hspace{100pt} v_\Ma(\frac{\pi^k}{\pi_{t}})-v_{\Ma}(D(K_t/S))-\frac{e(M)}{p-1}-1\ \}
\end{align*}
Dividing everything by $e(M/K_t)=n$ and noticing 
that $v_{\Ma}(x)=e(M/K_t)v_{\Ka_t}(x)$ for $x\in \Ka_t$ we obtain
\begin{multline*}
v_{\Ka_t}(\lambda_{i,j})
\geq \min \{\ v_{\Ka_t}(a_j)-v_{\Ka_t}(D(K_t/S))-\frac{e(K_t)}{p-1}-\frac{1}{n}\, ,\\
                       v_{\Ka_t}(\frac{\pi^k}{\pi_{K_t}})-v_{\Ka_t}(D(K_t/S))-\frac{e(K_t)}{p-1}-\frac{1}{n}\ \}
\end{multline*}
Letting $n\to \infty$ we obtain 
\begin{multline*}
v_{\Ka_t}(\lambda_{i,j})\geq
 \min \big\{\ v_{\Ka_t}(a_j)-v_{\Ka_t}(D(K_t/S))-e(K_t)/(p-1), \\
                       v_{\Ka_t}(\pi^k/\pi_{t})-v_{\Ka_t}(D(K_t/S))-e(K_t)/(p-1)\, \big\}
\end{multline*}
which implies (\ref{pls}).
\end{proof}

\end{document}